\documentclass[11pt]{article}
\usepackage{amsmath,amssymb,amsthm,eucal, mathrsfs, mathtools}
\usepackage{xcolor}
\usepackage[all]{xy}

\title{The Cayley transform in complex, real and graded $K$-theory}
\author{Chris Bourne$^{\dag,\,\spadesuit}$, 
Johannes Kellendonk$^{\S}$, 
Adam Rennie$^{\ddag}$
\thanks{email: 
\texttt{chris.bourne@tohoku.ac.jp, 
kellendonk@math.univ-lyon1.fr, renniea@uow.edu.au}}
\\[3pt]
${}^{\dag}$WPI-Advanced Institute for Materials Research (WPI-AIMR), Tohoku University,\\
 Sendai, Japan\\ 
${}^{\spadesuit}$RIKEN iTHEMS, Wako, Saitama, Japan \\
${}^{\S}$Institute Camille Jordan, 
Universit\'{e} Claude Bernard Lyon 1,\\ 
 Villeurbanne, France\\
${}^{\ddag}$School of Mathematics and 
Applied Statistics, University of Wollongong,\\
Wollongong, Australia\\
}

\advance\topmargin by -\headheight
\advance\topmargin by -\headsep
\evensidemargin=\oddsidemargin
\makeatletter
\def\section{\@startsection{section}{1}{\z@}{-3.5ex plus -1ex minus
  -.2ex}{2.3ex plus .2ex}{\large\bf}}
\def\subsection{\@startsection{subsection}{2}{\z@}{-3.25ex plus -1ex
  minus -.2ex}{1.5ex plus .2ex}{\normalsize\bf}}
\makeatother

\numberwithin{equation}{section} 

\theoremstyle{plain} 
\newtheorem{thm}{Theorem}[section]
\newtheorem{lemma}[thm]{Lemma}
\newtheorem{prop}[thm]{Proposition}
\newtheorem{cor}[thm]{Corollary}

\theoremstyle{definition} 
\newtheorem{defn}[thm]{Definition}

\theoremstyle{remark} 
\newtheorem{rmk}[thm]{Remark}
\newtheorem{example}[thm]{Example}

\DeclareMathOperator{\Cliff}{{\C\ell}} 
\DeclareMathOperator{\Ker}{Ker} 
\DeclareMathOperator{\Dom}{Dom}   
\DeclareMathOperator{\End}{End}   

\newcommand{\A}{\mathcal{A}}  
\newcommand{\C}{\mathbb{C}}   
\newcommand{\Cc}{\mathcal{C}} 
\newcommand{\D}{\mathcal{D}}  
\renewcommand{\H}{\mathcal{H}}  
\newcommand{\K}{\mathcal{K}}  
\newcommand{\N}{\mathbb{N}}   
\newcommand{\op}{\circ}       
\newcommand{\ox}{\otimes}     
\newcommand{\hox}{\hat\otimes} 
\newcommand{\R}{\mathbb{R}}   
\newcommand{\Z}{\mathbb{Z}} 

\newcommand{\stroke}{\mathbin|}     

\def\pairL_#1(#2|#3){{}_{#1}(#2\stroke#3)} 
\def\pairR(#1|#2)_#3{(#1\stroke#2)_{#3}} 
\def\scal<#1|#2>{\langle#1\stroke#2\rangle} 

\newbox\ncintdbox \newbox\ncinttbox 
	\setbox0=\hbox{$-$}
	\setbox2=\hbox{$\displaystyle\int$}
	\setbox\ncintdbox=\hbox{\rlap{\hbox
		to \wd2{\hskip-.125em \box2\relax\hfil}}\box0\kern.1em}
	\setbox0=\hbox{$\vcenter{\hrule width 4pt}$}
	\setbox2=\hbox{$\textstyle\int$}
	\setbox\ncinttbox=\hbox{\rlap{\hbox
		to \wd2{\hskip-.175em \box2\relax\hfil}}\box0\kern.1em}

\renewcommand{\epsilon}{\varepsilon}

\def\calC{\mathcal{C}}

\def\calK{\mathcal{K}}

\def\calA{\mathcal{A}}

\newcommand{\ol}{\overline}

\theoremstyle{definition}

\DeclareMathOperator{\Mult}{Mult}

\newcommand{\Cl}{\mathbb{C}\ell}

\newcommand{\Cd}{\Cc} 
\newcommand{\Ci}{\Cc^{-1}} 
\newcommand{\hot}{\hat\otimes}
\newcommand{\kg}{\rho}

\newcommand{\rs}{{\mathfrak r}}
\newcommand{\osu}{{\rm OSU}}

\definecolor{MyBlue}{cmyk}{1,0.13,0,0.63}
\definecolor{MyGreen}{cmyk}{0.91,0,0.88,0.52}
\newcommand{\mylinkcolor}{MyBlue}
\newcommand{\mycitecolor}{MyGreen}
\newcommand{\myurlcolor}{black}

\usepackage{hyperref}
\hypersetup{%
  bookmarksnumbered=true,bookmarksopen=false,%
  plainpages=false,
  linktocpage=true,%
  colorlinks=true,breaklinks=true,%
  linkcolor=\mylinkcolor,citecolor=\mycitecolor,urlcolor=\myurlcolor,%
  pdfpagelayout=OneColumn,%
  pageanchor=true,%
}

\newcommand{\hf}{\overline{\underline{h}}}
\newcommand{\Adam}{\mathfrak A}
\newcommand{\lot}{\hat\otimes} 

\begin{document}

\maketitle

\vspace{-2pc}

\begin{abstract}
We use the Cayley transform to provide an explicit isomorphism at the level of cycles 
from van Daele $K$-theory to $KK$-theory for graded $C^*$-algebras with a 
real structure. Isomorphisms between $KK$-theory and  complex or real $K$-theory 
for ungraded $C^*$-algebras are a special case of this map. In all cases our map is compatible with the 
 computational techniques required in
physical and geometrical applications, in particular index pairings and Kasparov
products. We provide applications to real $K$-theory and 
topological phases of matter.
\end{abstract}
\maketitle

\parindent=0.0in
\parskip=0.05in

\vspace{-1.6pc}

\section{Introduction}
\label{sec:intro}
This paper presents explicit isomorphisms between $K$-theory 
and (unbounded) $KK$-theory for possibly graded $C^*$-algebras with or without 
a real structure. We use a $K$-theory due to van Daele~\cite{vanDaele1,vanDaele2} to 
accommodate graded real or complex algebras. They key ingredient of our construction 
is the Cayley transform on Hilbert $C^*$-modules, which exchanges unbounded self-adjoint 
regular operators with unitary operators.

The isomorphism $DK(A) \cong KKR(\Cl_{1,0},A)$ 
 is already 
known from work by Roe~\cite{Roe04} and extended 
by Kubota~\cite[Theorem 5.11]{Kubota15a}. 
Roe shows that $DK(A)$ is isomorphic to $KK(\Cl_1,A)$ 
and, for algebras with a real structure $\rs_A$, $DK(A^{\rs_A})$ is isomorphic to 
$KKO(\Cl_1,A^{\rs_A})$, emulating the 
proof given in \cite[Section 17.5]{Blackadar} for standard $K$-theory. 
The resulting isomorphism is, however, not 
given at the level of cycles.

{\em The isomorphism we present is tailored to the needs of 
the applications, especially as they involve index pairings,
and their more sophisticated cousins, Kasparov products.}

Our work relies, ultimately, on results of Wood \cite{Wood}, whose 
proof of Bott periodicity informed van Daele and others. Wood's  proof is 
expressed in terms of odd operators with square $-1$ 
(a trivial difference to the picture of van Daele that uses odd self-adjoint 
unitaries), and Wood's methods were flexible enough to be used widely. 
Our use of the Cayley transform is steeped in long tradition, and in 
particular has been applied to graded $K$-theory by Trout \cite{TroutGraded}. 
Utilising the Cayley transform to
obtain a converse of the functional calculus, Trout obtains a homotopy theoretic
definition of graded $K$-theory.

Our work is motivated by results of the second author (and 
more recently~\cite{AlldridgeMax}), who 
showed that 
homotopy classes of gapped Hamiltonians (with 
prescribed symmetries) are classified directly in 
terms of van Daele $K$-theory, 
\cite{AlldridgeMax, Kellendonk15}. The work in \cite{AlldridgeMax, Kellendonk15}
complements and extends  links of complex and real $K$-theory 
to topological states of matter, \cite{FM13, Kitaev09, Kubota15b,Thiang16}. 

Our isomorphism identifies a van Daele 
element with a concrete unbounded Kasparov cycle, 
which can then be used to take 
pairings/products quite explicitly in terms of cycles. Such
computations are
compatible with the bulk-edge correspondence 
as in \cite{BKR, Kubota15b}. We present
some examples and show how our technique facilitates computations.
The Appendix provides sufficient conditions to be able to compute 
such products explicitly on the level of cycles.

We first give a review of van Daele $K$-theory, $KK$-theory and Kasparov's 
stabilisation theorem 
in Section \ref{sec:prelim}. We use refinements of van Daele $K$-theory 
due to Roe \cite{Roe98,Roe04}, describe the relationship between
these two approaches, and prove an excision isomorphism.

In Section \ref{sec:complex_K_to_KK}, we review the Cayley transform 
for ungraded Hilbert $C^*$-modules, which we use to construct an
explicit isomorphism between odd 
$K$-theory and (unbounded) $KK$-theory for trivially graded complex $C^*$-algebras. 
An isomorphism $KK(\C,A)\to K_0(A)$ is also given using the graph projection 
of the (unbounded) operator of the $KK$-cycle.

In Section \ref{sec:DK_KK}, we introduce a Cayley transform 
on graded $C^*$-modules, where the key difference with the ungraded case is that 
the transform  interchanges a \emph{pair} of odd self-adjoint unitaries with 
an unbounded, odd, self-adjoint and regular operator anti-commuting 
with an odd self-adjoint unitary. We remark that some of our constructions 
are similar to the converse functional calculus used by Trout~\cite{TroutGraded} 
to study graded $K$-theory. 

The graded Cayley transform is used to prove our main result 
in Section \ref{subsec:Cayley_DK_iso}, where 
an explicit isomorphism $DK(A) \to KKR(\Cl_{1,0}, A)$ and inverse is constructed.
While our map from van Daele $K$-theory to $KK$-theory and its inverse are explicit, 
the proof that we obtain isomorphisms is somewhat involved. This is because 
a generic countably generated $C^*$-module  need 
not be full, and so Morita equivalence relates the compact endomorphisms on 
the $C^*$-module to an ideal of the coefficient algebra. 
To prove that our maps give isomorphisms, we demonstrate 
the compatibility of the Cayley transform with Morita equivalence and non-full $C^*$-modules.

Some applications of our isomorphism are considered in Sections \ref{sec:app-real} and \ref{sec:TI_application}. 
In Section \ref{subsec:unitary_KO}, we consider the case that 
$A = B \otimes \Cl_{r,s}$ for some trivially graded
$B$ with real structure $\rs_B$ 
and so the image of our Cayley isomorphism is $KO_{1+s-r}(B^{\rs_B})$ or $K_{1-s-r}(B)$ 
if we ignore the real structure.  
We show how in these special cases we recover 
unitary descriptions of real $K$-theory as studied in~\cite{BL15} 
and~\cite[Section 5.6]{Kellendonk15}.
In Section \ref{subsec:DK_bdry}, we show how our Cayley isomorphism 
interacts with the boundary map in van Daele $K$-theory. We can use this 
result to explicitly write  the boundary map in $KK$-theory, 
$KKR(\Cl_{1,0}, A) \xrightarrow{\delta} KKR(\C, I)$  coming 
from the semi-split short exact sequence
$$
   0 \to I \to E \to A \to 0
$$
such that each map respects the $\Z_2$-grading and real structure. 

Finally in Section \ref{sec:TI_application} we use our Cayley map to write 
down Kasparov modules representing bulk and 
boundary invariants of topological insulators. 
We note that we do not specify our 
algebra of observables and work with a generic complex $C^*$-algebra, possibly 
graded and possibly with a real structure implementing the anti-linear symmetries of the system. 
Of particular note are the boundary invariants, 
where the Kasparov modules we construct 
are explicitly linked to a lift of the Hamiltonian under the bulk-boundary short 
exact sequence. Such lifts are typically related to 
half-space Hamiltonians and 
edge spectra. Hence our work complements recent results~\cite{AlldridgeMax, SBToniolo} 
that express the boundary $K$-theory class using the half-space Hamiltonian.

{\bf Acknowledgements} We would like to thank Alex Kumjian,
David Pask, and Aidan Sims for pointing out some graded issues that were a bit odd. 
We also thank Christopher Max for sharing the results of~\cite{AlldridgeMax} with us, and Matthias Lesch for helpful conversations. 
Finally, Bram Mesland has once again provided timely and useful advice. 
This work was supported by World Premier International Research Center Initiative (WPI), 
MEXT, Japan. AR  was partially supported by the BFS/TFS project Pure Mathematics in Norway and 
CB is supported by a JSPS Grant-in-Aid for Early-Career Scientists (No. 19K14548).
All authors thank the Erwin Schr\"{o}dinger Institute program Bivariant $K$-Theory in Geometry and 
Physics for hospitality and support during the production of this work.


\section{Preliminaries} \label{sec:prelim}

Our paper is concerned with operator $K$-theory, van Daele $K$-theory and
$KK$-theory. The Cayley transform gives us a method to pass between these 
theories.

{\bf Conventions}: We assume that all $C^*$-algebras we encounter have 
a countable approximate identity ($\sigma$-unital). 
Recall that a real $C^*$-algebra is a $C^*$-algebra over the field of real numbers.  
In contrast, a Real $C^*$-algebra, written with large R, is a complex $C^*$-algebra $A$ with a \emph{real structure}, 
that is, an anti-linear multiplicative map $\rs_A:A\to A$ which is of order $2$. Elsewhere, e.g.\ in \cite{Kellendonk16}, 
Real $C^*$-algebras are also called $C^{*,r}$-algebras. 
The subalgebra of elements fixed 
under $\rs_A$ is a real $C^*$-algebra $A^{\rs_A}$. 
In specific situations where the context is clear, 
we will omit the subscript and simply denote a real structure by $\rs$.  
All $C^*$-algebras 
are $\Z_2$-graded (possibly trivially graded) unless otherwise stated, and we use 
$\Z_2$-graded tensor products $\hat\otimes$.

The real Clifford algebra $Cl_{p,q}$ is the algebra generated by $p$ self-adjoint odd elements 
$e_1,\dots,e_p$ with square 1 and $q$ skew-adjoint odd elements $f_1,\dots,f_q$ with square $-1$ 
which all pairwise anti-commute. We denote by 
$\Cl_{p,q}$ the Real $C^*$-algebra generated by $e_1, \ldots, e_p$ and $f_1, \ldots, f_q$.  
That is, its elements are complex linear combinations of products of these generators, equipped with the {real structure} 
$\mathfrak{r}$ such that $e_j^\mathfrak{r} = e_j$, $f_j^\mathfrak{r} = f_j$. 
It is immediate that $\Cl_{p,q} \cong \Cl_{p+q}$ as complex  algebras, and $\Cl_{p,q}^\mathfrak{r} = Cl_{p,q}$.
We will make frequent use of the Pauli matrices, where to establish notation we recall
\begin{align*}
&\sigma_1=\begin{pmatrix}0 & 1\\ 1 & 0 \end{pmatrix}, 
&&\sigma_2=\begin{pmatrix} 0 & -i\\ i & 0\end{pmatrix},
&&\sigma_3=\begin{pmatrix} 1 & 0\\ 0 & -1\end{pmatrix}.
\end{align*}
We will freely take advantage of the isomorphism $\Cl_{1,1} \cong C^*(\sigma_1, -i\sigma_2)$ with real 
structure by entrywise complex conjugation.

\subsection{Van Daele $K$-theory}
\label{subsec:vD-K}

\begin{defn}
Let $A$ be a real or complex $C^*$-algebra. We say that $A$ has a balanced $\Z_2$-grading 
if $A$ contains an odd self-adjoint unitary\footnote{Odd self-adjoint unitaries are called super-symmetries by 
Roe~\cite{Roe98, Roe04}.} 
(OSU). That is, there is an odd element $e$ satisfying $e=e^* = e^{-1}$. 
In particular $A$ is unital. If $A$ has a real structure $\rs_A$, we also require $e^{\rs_A} = e$.
\end{defn}

Let $V(A)=\bigsqcup_k\pi_0(\mbox{\rm OSU}(M_k(A)))$, the disjoint union of homotopy classes of OSUs in $M_k(A)$, 
the $k\times k$ matrices with entries in $A$, $k\geq 1$. 
Here the grading and real structure on $A$ are extended to $M_k(A)$ entrywise. 
The set $V(A)$ is an abelian semigroup with direct sum as operation,
$[x]+[y] = [x\oplus y]$. The Grothendieck group obtained from this semigroup will be denoted 
$GV(A)$. 
The semigroup homomorphism $d:V(A)\to \N$ taking the 
value $k$ on $M_k(A)$ induces a group
homomorphism $d:GV(A)\to\Z$.

\begin{defn} 
\label{defn:vD-K}
If $A$ is unital and has a balanced $\Z_2$-grading, then
the van Daele group of $A$ is $DK(A)=\Ker(d:GV(A)\to\Z)$. 

If $A$ is unital but is not balanced, then we set $DK(A)=\Ker(d:GV(A\hat\otimes \Cl_{1,1})\to\Z)$. 
The complex and real case is given by ignoring the real structure or passing to the real subalgebra 
$A^{\rs_A}\hat\otimes Cl_{1,1}$.

If $A$ is not unital then we set
$DK(A)=\Ker(q_*:DK(A^\sim)\to DK(\C))$ where $q:A^\sim\to\C$ quotients the minimal unitisation $A^\sim$ by the ideal $A$,  
replace $\C$ by $\R$ if $A$ is real, see \eqref{eq-DK-Roe}.

Elements of $DK(A)$ are formal differences of OSUs denoted by $[x]-[y]$. 
\end{defn}
We elaborate on the non-unital case.
Since $\C$ (or $\R$) is trivially graded, the relevant 
exact sequence needed to define $q_*$ is
$$
 0 \to A \hat\otimes \Cl_{1,1} \to A^\sim\hat\otimes \Cl_{1,1} 
 \xrightarrow{q\hat\otimes \mathrm{Id}} \C\otimes \Cl_{1,1}\to 0
$$ 
so that 
$DK(A) = \{[x]-[y]\in DK(A^\sim\hat\otimes \Cl_{1,1}):\,(q\hat\otimes \mathrm{Id})_*[x] = (q\hat\otimes 1)_*[y]\}$. 
Adapting  \cite[Proposition 3.7]{vanDaele1} to Roe's formulation, one can 
find representatives $x'$ for $[x]$ and $y'$ for $[y]$ such that 
$q\hat\otimes \mathrm{Id}(x') = q\hat\otimes \mathrm{Id}(y')$. Thus, for non-unital $A$
\begin{equation}\label{eq-DK-Roe}
DK(A) = \{[x]-[y] :\, x,y\in \mbox{\rm OSU}(M_n(A^\sim\hat\otimes \Cl_{1,1})),\ \ 
x-y\in M_n(A\hat\otimes \Cl_{1,1})\}.
\end{equation}

This is Roe's version of van Daele $K$-theory \cite{Roe98, Roe04}. 
As already mentioned, Roe shows that van Daele's 
$K$-groups are isomorphic to $KK$-groups from which 
we infer that they share all the standard properties of 
$K$-theory, though often we can only exploit this 
easily for ungraded algebras.

If $A$ is complex with a real structure $\rs_A$, then we 
sometimes denote the van Daele $K$-theory group by 
$DK(A,\rs_A)$ to emphasise this. Clearly $DK(A, \rs_A) \cong DK(A^{\rs_A})$.

If $A$ is balanced graded, one may ask if we could equivalently use $A\hat\otimes \Cl_{1,1}$ to define $DK(A)$. 
The following lemma shows that such a choice leads to consistent definitions of van Daele $K$-theory.
\begin{lemma}[\cite{Roe98}]\label{lem-Cl} Let $A$ be balanced graded.
The map 
\begin{equation}
[x]-[y]\mapsto [\frac{x+y}2 \lot 1_2 +   \frac{x-y}2 \lot \sigma_3] - [1\lot \sigma_1]=\left[\begin{pmatrix} x & 0\\ 0 & y\end{pmatrix}\right]-\left[\begin{pmatrix} 0 & 1\\ 1 & 0\end{pmatrix}\right]
\label{eq:ubiquitous-isomorphism}
\end{equation}
furnishes an isomorphism between $DK(A)$ and $DK(A\hot \Cl_{1,1})$. 
\end{lemma}
Roe refers to van Daele \cite{vanDaele1} for the proof. For the convenience of the reader we 
reproduce it below after having introduced van Daele's picture. We also note that for 
balanced graded $A$ the
identification of $M_2(A)$ with entrywise grading with $A\hat\otimes \Cl_{1,1}$ depends on 
a choice of OSU in $A$, cf. Equation \eqref{eq-psi}.

\subsubsection{Base-points, and van Daele's picture}

Van Daele's original definition of the version of $K$-theory which is named after him 
\cite{vanDaele1,vanDaele2} requires a choice of base point. 
This is a choice of OSU $e\in A$ if $A$ is balanced graded, or 
$e\in A\hat\otimes \Cl_{1,1}$ if not. This OSU is then used 
to embed $M_k(A)$ into $M_{k+1}(A)$ 
via $x\mapsto x\oplus e$. The semigroup $V_e(A) =  \bigcup_{k}\pi_0(\mbox{\rm OSU}(M_{k}(A)))$  
is such that the union is no longer disjoint as $\pi_0(\mbox{\rm OSU}(M_{k}(A)))$ 
is  identified with a subset of $\pi_0(\mbox{\rm OSU}(M_{k+1}(A)))$ via the above embedding. 

The semigroup $V_e(A)$ depends on $e$ up to homotopy.
It has a unit element, the class of $e$.
Van Daele's version of the $K$-group is thus given by the Grothendieck
group $DK_e(A) = GV_e(A)$, where we include the chosen base point in our notation.
If we denote for a moment the corresponding homotopy 
classes in $V_e(A)$ by $[x]_e$ 
then $[x] \mapsto [x]_e$ induces a map  $\alpha_e:GV(A)\to GV_e(A)$,
$\alpha_e([x]-[y])= [x]_e - [y]_e$ between the corresponding Grothendieck groups. 
Restricted to the kernel of $d$ the map $\alpha_e$ is a group isomorphism.
We therefore arrive at two (isomorphic) presentations of the van Daele 
$K$-theory group. 
\emph{In the sense that it uses the Grothendieck completion of a semigroup, 
Roe's formulation 
exhibits van Daele $K$-theory 
as a relative theory. In contrast,  van Daele's formulation 
expresses the elements as 
relative to a chosen base point.}

A particularly handy situation arises if the base point $e$ is 
homotopic to $-e$ (along a homotopy of OSUs in $A$). 
Then $V_e(A)$ is a group with inverse given by $-[x]_e = [-exe]_e$. 
If $A$ is balanced this can always be achieved: choose a base point $e\in A$. Then $M_2(A)$ contains the \osu\ $e\oplus -e$ 
which is homotopic to its negative  via the path 
$\begin{pmatrix} e \cos(t) & e \sin( t) \\ e \sin( t) & -e\cos( t) \end{pmatrix}$ for 
$t \in [0,\pi]$. The map 
$\varphi_e:GV_e(A)\to GV_{e \oplus -e}(M_2(A))$ given by 
\begin{equation}
\varphi_e([x]_e -[y]_e) = 
\left[ \begin{pmatrix}x & 0 \\ 0 &  -eye \end{pmatrix} \right]_{e\oplus -e}
\label{eq:varphi}
\end{equation}
is then an isomorphism of groups. 
If $A$ is not balanced but unital, then we start with $A\hat\otimes \Cl_{1,1}$ 
and choose $e=1\lot \sigma_1$ as base point. The van Daele $K$-group of $A$ is thus 
isomorphic to $GV_{\sigma_1 \oplus -\sigma_1}(M_4(A))$ (as it is originally defined in \cite{vanDaele1}).

If $A$ is balanced graded, then $M_2(A)$ (with component-wise extension of the grading) is isomorphic to $A\hat\otimes \Cl_{1,1}$, 
though the isomorphism depends on the choice of base point $e$. Given such an OSU, 
the isomorphism  is given by $\psi_e: A\hat\otimes \Cl_{1,1}\to M_2(A)$ 
\begin{equation}\label{eq-psi}
\psi_e(x\hot 1) = \begin{pmatrix} x & 0 \\ 0 & (-1)^{|x|}exe \end{pmatrix} ,\quad
\psi_e(1\hot \sigma_1) = \begin{pmatrix} 0 & e \\ e & 0 \end{pmatrix} ,\quad
\psi_e(1\hot i\sigma_2) = \begin{pmatrix} 0 & e \\ -e & 0 \end{pmatrix} .
\end{equation}
Note that $\psi_e^{-1}$ maps $e\oplus -e$ to $e\lot 1$ which is homotopic to $1\lot \sigma_1$ 
via $t\mapsto \cos(t)e\hox 1+\sin(t)1\hox \sigma_1$. 
These facts imply that the isomorphism of Lemma~\ref{lem-Cl} is 
given by the composition of four isomorphisms
$
\alpha_{1\hot \sigma_1}^{-1}\circ\psi_e^{-1}\circ \varphi_e\circ \alpha_e$.

\subsubsection{Excision  for van Daele $K$-theory}
Excision for $DK$ can be deduced from excision for ordinary $K$-theory
when the algebra is trivially graded, but seems not to have been 
addressed for graded algebras.

For  a balanced graded algebra $A$ with a (closed two-sided graded) ideal $I$ we 
define the relative van Daele group
$$
DK(A,A/I):=\{[x]-[y]:\,\, x,y\in \mbox{\rm OSU}(M_n(A)),\ \
x-y\in M_n(I)\}.
$$
Here $[\cdot]$ denotes homotopy classes in $\mbox{\rm OSU}(M_n(A))$.
If $A$ is not balanced, but only unital, then we use again $A\hox \Cl_{1,1}$ 
in place of $A$ in the above definition of the relative group; the ideal is then 
$I\hox \Cl_{1,1}$. Again this is reasonable as the map from Lemma~\ref{lem-Cl} provides an isomorphism between 
$DK(A,A/I)$ and $DK(A\hox \Cl_{1,1},A/I\hox \Cl_{1,1})$ in the case that $A$ is balanced. Indeed, 
for any OSU $y\in M_n( A)$, $w = \frac1{\sqrt{2}}(1-y\lot \sigma_1)$ is an even unitary which is homotopic to $1$ and therefore
$\frac{x+y}2 \lot 1_2 +   \frac{x-y}2 \lot \sigma_3$ is homotopic to $w(\frac{x+y}2 \lot 1_2 +   \frac{x-y}2 \lot \sigma_3)w^*$. Finally, a computation shows that 
$w(\frac{x+y}2 \lot 1_2 +   \frac{x-y}2 \lot \sigma_3)w^* - 1\lot\sigma_1\in I\hox \Cl_{1,1}$ provided $x-y\in I$.

Equation \eqref{eq-DK-Roe} can now be interpreted in the way that $DK(A) = DK(A^\sim,A^\sim/A)$ for a non-unital algebra $A$.
\begin{prop}[Excision for $DK$]
\label{prop:excision}
Let $I$ be a (closed two-sided graded) ideal in the unital algebra $A$.
Then $DK(I)\cong DK(A,A/I)$.
\end{prop}
\begin{proof} Since $A$ is unital it contains $I^\sim$ and hence any element 
$[x]-[y]\in DK(I^\sim,I^\sim/I)$ can be understood as an element of $DK(A,A/I)$. This defines a
map $DK(I)\to DK(A,A/I)$.

We first show this map is surjective. Since $I^\sim$ need not be balanced we work with $I^\sim\hox \Cl_{1,1}$ 
and consequently also with $A\hox \Cl_{1,1}$. Recall that $ DK(A,A/I)$ is thus generated by elements of the form 
$[\frac{x+y}2 \lot 1_2 +   \frac{x-y}2 \lot \sigma_3] - [1\lot \sigma_1]$ such that $x-y\in M_n(I)$.  
Then with $w = \frac1{\sqrt{2}}(1-y\lot \sigma_1)$,
$[\frac{x+y}2 \lot 1_2 +   \frac{x-y}2 \lot \sigma_3] = [w(\frac{x+y}2 \lot 1_2 +   \frac{x-y}2 \lot \sigma_3)w^*]$ and
$w(\frac{x+y}2 \lot 1_2 +   \frac{x-y}2 \lot \sigma_3)w^* - 1\lot\sigma_1\in I\hox \Cl_{1,1}$. 
Thus $\frac{x+y}2 \lot 1_2 +   \frac{x-y}2 \lot \sigma_3$ has a representative in $M_n(I^\sim)\hox \Cl_{1,1}$. 

For injectivity, we let 
$[\frac{x+y}2 \lot 1_2 +   \frac{x-y}2 \lot \sigma_3] - [\frac{x'+y'}2 \lot 1_2 +   \frac{x'-y'}2 \lot \sigma_3]$ 
be trivial in $ DK(A,A/I)$. There are then (perhaps after stabilisation) paths $x(t)$ and $y(t)$ of 
OSUs in $M_n(A)$ such that $x(0)=x$, $y(0)=y$, $x(1)=x'$, $y(1)=y'$ and, 
for all $t\in [0,1]$, $x(t)-y(t)\in M_n(I)$. We let $w(t) = \frac1{\sqrt{2}}(1-y(t)\lot \sigma_1)$. 
Then $w(t)(\frac{x(t)+y(t)}2 \lot 1_2 +   \frac{x(t)-y(t)}2 \lot \sigma_3)w(t)^*$ is a homotopy in $M_n(I^{\sim})\hox \Cl_{1,1}$ 
between two representatives of 
$[\frac{x+y}2 \lot 1_2 +   \frac{x-y}2 \lot \sigma_3]$ and $[\frac{x'+y'}2 \lot 1_2 +   \frac{x'-y'}2 \lot \sigma_3]$ 
in $M_n(I^{\sim})\hox \Cl_{1,1}$.
\end{proof}

Let us also consider a description of van Daele $K$-theory using 
base points in the case that $A$ is not unital.
We say that $A$ is \emph{weakly balanced graded} if its multiplier algebra contains an  
OSU (the grading on $A$ extends uniquely to the multiplier algebra of $A$).
Importantly $A\hox\Cl_{1,1}$ is always weakly balanced graded.
Having fixed an OSU $e$ in the multiplier algebra, we define $A^{\sim e}$ to be the subalgebra 
of the multiplier algebra generated by $A$ and $e$. We use the notation 
$e_n = e^{\oplus n}$.
\newcommand{\tcr}{\textcolor{red}}
\newcommand{\tcb}{\textcolor{blue}}
\begin{lemma}
\label{lem:unitisation}
Let $A$ be a non-unital and weakly
balanced graded algebra with base point $e$ in the multiplier algebra. Then 
\begin{equation}\label{eq-form-e}
DK_e(A) 
:=  \big\{[x]-[y]\in DK_{e}(A^{\sim e}):\,x- (e_k \oplus -e_{n-k}),\,y-(e_k \oplus -e_{n-k})\in M_{n}(A),\ \mbox{some }n,k\big\}
\end{equation}
is isomorphic to $DK(A)$
(Definition \ref{defn:vD-K}). 
\end{lemma}
\begin{proof}
Since $A$ is weakly balanced, $A$ is an ideal in $A^{\sim e}$. Hence $DK(A) \cong DK(A^{\sim e},A^{\sim e}/A)$, 
that is, $DK(A)$ is given by elements $[x]-[y]$, $x,y\in {\rm OSU}(M_n(A^{\sim e}))$ 
with $x-y \in M_n(A)$. Hence $\alpha_e^{-1}$ induces an injective map from $DK_e(A)$ into $DK(A^{\sim e},A^{\sim e}/A)$.

Let $q : A^{\sim e} \to A^{\sim e}/A \cong \Cl_{1,0}$ be the natural projection with $\tilde e = q(e)$ the generator of $\Cl_{1,0}$. The only 
OSUs of $\Cl_{1,0}$ are $\tilde e$ and $-\tilde e$ and they are not homotopic. 
Therefore for a given OSU $x\in M_n(A^{\sim e})$, 
$q(x)$ is homotopic $\tilde e_k\oplus (-\tilde e)_{n-k}$ for some $k$.
Hence there is an even unitary $\tilde w\in M_n(\Cl_{1,0})$, homotopic to $1$ 
along a path of even unitaries, such that $\tilde w q(x) \tilde w^* =  \tilde e_k\oplus (-\tilde e)_{n-k}$. 
We lift $\tilde w$ to a unitary $w\in A^{\sim e}$ (via $\tilde e\mapsto e$). 
Then $w x w^*$ is homotopic to $x$ along a path of OSUs. 
Now let $y$   be an OSU in $M_n(A^{\sim e})$ such that $x-y\in M_n(A)$. 
Then $q(y)$ is homotopic to $\tilde e_k\oplus (-\tilde e)_{n-k}$ with $k$ the same as $q(x)$. 
Similarly, we can find an even unitary $v\in A^{\sim e}$, homotopic to $1$, such that 
$q(vyv^*) =  \tilde e_k\oplus (-\tilde e)_{n-k}$. 
We thus have found $[wxw^*]-[vyv^*] \in DK_e(A)$ which is a preimage of 
$[x]-[y]\in DK(A^{\sim e},A^{\sim e}/A)$. The excision isomorphism of 
Proposition \ref{prop:excision} completes the proof.
\end{proof}

If we study homotopy classes of OSUs where there is a canonical 
or simple choice of base point, then our picture simplifies.
For example, if $A$ is a unital and trivially graded 
algebra, then OSUs of $A \otimes \Cl_{1,1}$ are of the form
$$
U = \begin{pmatrix} 0 & u^* \\ u & 0 \end{pmatrix}
$$ 
where $u\in A$ is unitary. If
we choose $e=1\otimes \sigma_1$ as base point 
then we see that the map $U\mapsto u$ identifies 
$V_e(A \otimes \Cl_{1,1})$ with the homotopy classes 
of unitaries in $(A\otimes {\mathcal K})^\sim$. 
Hence, we recover the group $K_1(A)$ or $KO_1(A^{\rs_A})$.

\subsection{$C^*$-modules and $K$-theory}
\label{subsec:pain}

Given a countably generated right-$A$ $C^*$-module $X_A$ 
we denote by $(\cdot\mid\cdot)_A$ the $A$-valued inner-product, $\End_A(X)$ the adjointable 
endomorphisms of $X_A$ and $\End_A^0(X)\subset \End_A(X)$ the ideal of 
compact endomorphisms, see \cite{Lance}. Any right-$A$ $C^*$-module 
$X_A$ naturally gives rise to an ideal $J_X = \ol{\mathrm{span}(X | X)_A}$ (closure in the norm of $A$). 
The module $X_A$ is called full if $J_X=A$. 
In this section we work with $\Z_2$-graded $C^*$-modules over $\Z_2$-graded 
algebras. The ungraded case is analogous but simpler. 
Endomorphism algebras will always have a $\Z_2$-grading
inherited from acting on a $\Z_2$-graded module. We say that 
$X_A$ is {\em balanced graded}  if $\End_A(X)$ admits an OSU.

Recall that the standard graded Hilbert module over $A$ is given by 
$\hat{\H}_A:=\hat{\H}\hox A$ where $\hat\H$ is the graded infinite dimensional separable 
Hilbert space $\ell^2(\N)\oplus {\ell^2(\N)}^\op\cong \ell^2(\N)\otimes \C^2$ with  grading operator $1\otimes \sigma_3$. 
There is a standard OSU on $\hat{\H}_A\cong (\ell^2(\N)\otimes \C^2\hox A)_A$ given by 
\begin{equation}
Z={\rm Id}_{\ell^2(\N)}\ox\sigma_1\hox 1_{\mathrm{Mult}(A)}
\label{eq:zed}
\end{equation}
The Kasparov stabilisation theorem says that for any countably generated 
$X_A$, $(X \oplus \hat{\H})_A \cong \hat{\H}_A$ 
as graded Hilbert $A$-modules~\cite{KasparovStinespring}.
The compact endomorphisms on $\hat{\H}_A$ are 
$\End_A^0(\hat{\H}_A) \cong \hat\calK \hat\otimes A$ with $\hat \K$ being the graded 
algebra of compact operators on $\hat \H$. If $A$ is balanced with an OSU $e$, then 
we can apply the isomorphism from Equation \eqref{eq-psi} to obtain 
$\hat\calK \hat\otimes A\cong \K \otimes A$, 
where $\K$ is the trivially graded compact operators in which we absorbed a copy of 
$\Cl_{1,1}$.

\begin{lemma}[Morita invariance for $DK$]
\label{lem:DK-morita}
Let $X_A$ be a countably generated $C^*$-module over
the $C^*$-algebra $A$ and define the ideal $J_X=\ol{{\rm span}(X|X)_A}$.
Then $X$ is full as a module over $J_X$ and there is an isomorphism
\begin{equation}
\zeta_X:\,DK(\End^0_A(X)) \xrightarrow{\simeq} DK(J_X).
\label{eq:another-fucking-iso}
\end{equation}
\end{lemma}
\begin{proof}
The algebra $\End^0_A(X)^\sim \hat\otimes \Cl_{1,1}$ is 
 balanced graded, and by definition $X$ is a full module over $J_X$.
Recall that the group $DK(\End^0_A(X))$ is made up of differences 
$[U]-[V]$ with $U,\,V$ OSUs over $\End^0_A(X)^\sim \hat\otimes \Cl_{1,1}$
such that $U-V\in \End^0_A(X)\hat\otimes \Cl_{1,1}$.

Let $K$ be an odd self-adjoint 
finite rank endomorphism with $U-V-K$ small enough
in norm so that $V+K$ is invertible. This is possible since
$$
U=V+U-V=V+K+(U-V-K).
$$
Because the finite rank operators are stable under the holomorphic functional calculus~\cite[Lemma 6.3]{KNR},
we can take $\tilde{U}=\mathrm{phase}(V+K)$ such that $\tilde{U}-V$ is finite rank.
The path $[0,1]\ni t\mapsto \mathrm{phase}(V+K+t(U-V-K))$ gives a homotopy from $\tilde{U}$
to $U$, and so every class $[U]-[V]\in DK(\End^0_A(X))$ is represented
by a class $[\tilde{U}]-[V]$ with $\tilde{U}-V$  finite rank. Any two such representatives
are homotopic (in $DK(\End^0_A(X))$) by construction. 
Thus 
the difference
$
\begin{pmatrix} \tilde{U} & 0\\ 0 & V\end{pmatrix}
-\begin{pmatrix} V & 0\\ 0& V\end{pmatrix}
$
is finite rank 
and $\begin{pmatrix} V & 0\\ 0& V\end{pmatrix}$ 
is homotopic to $\begin{pmatrix} 0 & 1\\ 1 & 0\end{pmatrix}$.
Now let $W:(X\oplus\hat{\H})_{J_X\hox\Cl_{1,1}} \to\hat{\H}_{ J_X\hox\Cl_{1,1}}$ be a stabilisation isomorphism, and
$W_2=W\oplus W$ two copies of $W$.
Then for any OSU $Z$ on $\hat{\H}_{J_X\hox\Cl_{1,1}}$,
\[
 \left[ \begin{pmatrix} \tilde{U} & 0 & 0 & 0 \\ 0 & Z & 0 & 0 \\ 0 & 0 & V & 0 \\ 0 & 0 & 0 & Z \end{pmatrix} \right] - 
  \left[\begin{pmatrix} 0 & 0 & 1 & 0\\ 0 & Z & 0&0\\ 1 & 0 & 0 & 0\\ 0 & 0 & 0 & Z\end{pmatrix}\right] 
  =: [\tilde{U}\oplus Z\oplus V\oplus Z]-[\hat{Z}]
\]
defines a class in $DK(\End_{J_X}^0(X\oplus\hat{\H}))$ and 
$(\mathrm{Ad}_W)_\ast :DK(\End_{J_X}^0(X\oplus\hat{\H}))\to DK(\End^0_{{J_X}\hox\Cl_{1,1}}\!(\hat{\H}))$.
Note that $\hat Z$ is homotopic to $V \oplus Z\oplus V\oplus Z$ and the latter differs from 
$\tilde{U}\oplus Z\oplus V\oplus Z$ by a finite rank operator.
Therefore $W_2 \hat Z W_2^*$ is homotopic to $W_2 (V \oplus Z\oplus V\oplus Z) W_2^*$ and the latter differs from 
$W_2(\tilde{U}\oplus Z\oplus V\oplus Z) W_2^*$ by a finite rank operator, i.e. a 
matrix over ${J_X}\hat\otimes \Cl_{1,1}$.
It follows that 
$$
[W_2(\tilde{U}\oplus Z\oplus V\oplus Z)W_2^*]-[W_2\hat{Z}W_2^*]
$$
is a well-defined element in $DK({J_X})$. Thus we have a well-defined and clearly injective map
\begin{align*} \label{eq:explicit_excision}
DK(\End^0_A(X)) \ni [U]-[V] &\mapsto \zeta_X([U]-[V])\\
&= [W_2(\tilde{U}\oplus Z\oplus V\oplus Z)W_2^*]-[W_2\hat{Z}W_2^*]\in DK(M_2({J_X})).
\end{align*}

For surjectivity, suppose that 
$R,\,S\in M_n(J_X^\sim \hox\Cl_{1,1})$ are  OSUs with
$R-S\in M_n({J_X}\hox\Cl_{1,1})$. We consider the corresponding 
class 
$[\tfrac{1}{2}(R+S)\otimes 1 + \tfrac{1}{2}(R-S) \otimes \sigma_3] - [1\hox \sigma_1] = [R\oplus S]-[1\hox \sigma_1]$ 
via the isomorphism \eqref{eq:ubiquitous-isomorphism}. 
As operators on $\hat{\H}_{{J_X} \hox \Cl_{1,1}}$, we note that 
$Z={\rm Id}_{\ell^2(\N)}\ox\sigma_1\hox 1_{\mathrm{Mult}({J_X}) \hox \Cl_{1,1}}$ 
and $\mathrm{Id} \ox 1\hox \sigma_1$ anti-commute and so are homotopic by a 
path of OSUs. Therefore, using the stabilisation map 
$W: (X \oplus \hat{\H})_{ {J_X} \hox \Cl_{1,1}} \to \hat{\H}_{ {J_X} \hox \Cl_{1,1}}$,  we can define a class 
$[W^*(R\oplus S)W |_{X_A}] - [W^* Z W |_{X_A}] \in DK(\End_A^0(X))$, whose representative OSUs
may be homotopied so that the difference is a finite rank operator. 

We then apply $\zeta_X$ and obtain the class
\begin{align*}
&\left[W_2\begin{pmatrix} W^*(R\oplus S)W|_{X_A} & 0&0&0\\ 0&Z&0&0\\ 0 &0 & W^*  ZW |_{X_A}&0\\ 0 & 0 & 0 & Z\end{pmatrix}W_2^*\right]
-[W_2\hat{Z}W_2^*] 
=[R\oplus S]-[1 \ox\sigma_1] \\
&=[R]-[S],
\end{align*}
where the last equality following from an application of the isomorphism 
\eqref{eq:ubiquitous-isomorphism}.
\end{proof}

\begin{cor}
\label{cor:rel-morita}
Let $X_A$ be a balanced graded 
countably generated $C^*$-module over
the $C^*$-algebra $A$ and define the ideal 
${J_X}=\ol{{\rm span}(X|X)_A}$. 
There is an isomorphism (abusively still called $\zeta_X$) 
$\zeta_X:\,DK(\End_A(X),\End_A(X)/\End^0_A(X)) \xrightarrow{\simeq} DK({J_X})$ 
 given by applying excision and then
\begin{equation}
 \zeta_X([U]-[V])= \big[W_2 (\tilde{U}\oplus Z \oplus  V\oplus Z) W_2^{-1}\big]-\big[W_2\hat{Z}W_2^*\big] .
\label{eq:yet-another-fucking-iso}
\end{equation}
\end{cor}
\begin{proof}
The excision map $DK(\End_A(X),\End_A(X)/\End^0_A(X))\to DK(\End_A^0(X))$
simply picks representatives of each class which lie in $\End^0_A(X)^\sim\hox\Cl_{1,1}$, as shown in Proposition \ref{prop:excision}. 
We then apply the Morita isomorphism of
Lemma \ref{lem:DK-morita}.
\end{proof}

Finally, we can take all
spaces and unitaries to be Real, and carry through the same 
discussion without issue. We observe that Exel's elegant
Fredholm index proof \cite{Exel} of Equation \eqref{eq:another-fucking-iso}
for complex $K$-theory
would require the development of Fredholm theory in our setting,
and so we have opted for this more direct route.

\subsection{$KK$-theory with real structures}
\label{subsec:KK}

We now briefly review Real Kasparov theory or $KKR$-theory~\cite{Kasp}. 
A complex Hilbert $C^*$-module 
$X_B$ is a Real Hilbert $C^*$-module if 
there is an antilinear map $\mathfrak{r}_X:X_B\to X_B$,
called the real structure, 
such that $(x^{\rs_X})^{\rs_X} = x$,
$x^{\mathfrak{r}_X} \cdot b^{\mathfrak{r}_B} = (x \cdot b)^{\mathfrak{r}_X}$ and 
$( x_1^{\mathfrak{r}_X} \mid x_2^{\mathfrak{r}_X} )_B = \big(( x_1 \mid x_2 )_B\big)^{\mathfrak{r}_B}$. 
The real structure on the $C^*$-module induces 
a real structure $\mathfrak{r}$ on $\End_B(X)$ 
via $S^\mathfrak{r} x = \big( S(x^{\mathfrak{r}_X} ) \big)^{\mathfrak{r}_X}$. 
Representations of Real algebras 
$\pi: A \to \End_B(X)$ should be compatible with this real 
structure, $\pi(a^{\mathfrak{r}_A}) = \pi(a)^\mathfrak{r}$.

 We will often work with unbounded operators on $C^*$-modules, 
 see~\cite[Chapter 9]{Lance}. We  recall 
 that a densely defined closed 
 self-adjoint operator $D:\Dom (D)\subset X_B \to X_{B}$ 
is \emph{regular} if the operator $1+D^2:\Dom (D^2) \to X_{B}$ has dense range.  
We write that $D^\rs = D$ if $(\Dom (D))^{\rs_X} \subset \Dom (D)$ and 
$(Dx^{\rs_X})^{\rs_X} = Dx$ for all $x \in \Dom (D)$.
We also recall the graded commutator, where for endomorphisms
$S,\,T$ with homogenous parity 
$[S,T]_\pm = ST - (-1)^{|S| \,|T|}TS$.

\begin{defn}
Let $A$ and $B$ be $\Z_2$-graded Real $C^*$-algebras.
A Real unbounded Kasparov module $(\calA, {}_\pi{X}_B, D)$ 
consists of
\begin{enumerate}
\item a Real and
$\Z_2$-graded $C^*$-module ${X}_B$, 
\item a Real and graded $*$-homomorphism $\pi:A \to \End_B(X)$, 
\item an
unbounded self-adjoint, regular and odd operator $D=D^\rs$ and a dense $*$-subalgebra 
$\calA\subset A$ such that for all $a\in \calA\subset A$, 
$\pi(a)\Dom(D)\subset\Dom(D)$ and
\begin{align}
  & [D,\pi(a)]_\pm \,\in\, \End_B(X)\;,   &&\pi(a)(1+D^2)^{-1}\,\in\, \End_B^0(X)\;.
  \label{eq:defn}
\end{align}
\end{enumerate}
If both algebras and the module are trivially graded
but the self-adjoint regular operator $D$ still satisfies 
the conditions \eqref{eq:defn}, then we have an
odd Kasparov module. 
\end{defn}

We will often omit the representation $\pi:A\to \End_B(X)$ if the context is clear. 
An unbounded Kasparov modules represent a class 
$[(A, X_B, D(1+D^2)^{-1/2})] \in KKR(A,B)$~\cite{BJ}, 
though we remark that the group $KKR(A,B)$ depends on the 
choice of real structures for $A$ and $B$. 
If $A$ and $B$ are ungraded and $(\calA, X_B, D)$ is an odd unbounded Kasparov 
module, then we can turn it into a graded Kasparov module 
$\big( \calA\hat\otimes \Cl_{0,1}, \, X_B \hat\otimes \C^2, \, D\hat\otimes \sigma_1 \big)$, where 
the generator of the left $\Cl_{0,1}$-action is represented by the matrix $-i\sigma_2 =\begin{pmatrix} 0 & -1 \\ 1 & 0 \end{pmatrix}$. 

If $(\calA, X_B, D)$ is a Real unbounded Kasparov module, then we can 
ignore the real structures and obtain a complex unbounded Kasparov module and 
class in $KK(A, B)$. 
If we restrict the Real $C^*$-module $X_B$ to the 
elements fixed under $\rs_X$, we obtain a real $C^*$-module $X_{B^{\rs_B}}^{\rs_X}$. 
Similarly, the Real left action of $A$ becomes a real left action 
$\pi: A^{\rs_A} \to \End_{B^{\rs_B}}( X^{\rs_X} )$. 
We do not lose any information by restricting Real Kasparov modules to real $C^*$-modules 
and algebras. Similarly, 
real Kasparov modules can be complexified to obtain Real Kasparov modules and 
so $KKR(A, B) \cong KKO(A^{\rs_A}, B^{\rs_B})$. 
If the algebra $B$ is trivially graded, we can also 
consider real $K$-theory, where $KKR(\Cl_{r,s}, B) \cong KKO(Cl_{r,s}, B^{\rs_B}) \cong KO_{r-s}(B^{\rs_B})$.

\begin{rmk}[Normalisation of classes in $KKR(\Cl_{1,0}, A)$] \label{rk:KK(Cl_1,A)_normalisation}
The $KK$ group we  focus on in 
this manuscript is $KKR(\Cl_{1,0}, A)$. 
We review some basic simplifications of 
representatives of $KK$-classes, cf.~\cite[Section 17.4]{Blackadar}.
To emphasise the generator $e$ of the Clifford representation, 
we will denote the 
Kasparov module $(\Cl_{1,0}, X_A, T)$ (bounded or unbounded) 
by $(e, X_A, T)$. 

We can assume without loss of generality that 
$e^2 = 1_X \in \End_A(X)$. 
(If $e^2$ acts as a projection $P$, we can restrict the 
module to $PX_A$, replacing $T$ by $PTP$ and
the remaning part of the Kasparov module will be degenerate.) This therefore 
allows us to assume that  $X_A$ is a balanced graded $C^*$-module.
Lastly, we can guarantee that the generator $e$ of the $\Cl_{1,0}$ action on 
$X_A$ anti-commutes (graded-commutes) with the operator $T$. 
If $Te+eT \neq 0$, then we can take the perturbation 
$\tilde{T} = \frac{1}{2}(T - eTe)$ which anti-commutes with $e$ without
changing the $KKR$-class.
\hfill $\diamond$
\end{rmk}

\section{A Cayley isomorphism of complex $K$-theory and  $KK$-theory} \label{sec:complex_K_to_KK}

In this section, we consider the more familiar complex ungraded $K$-theory and 
define an isomorphism to (odd) $KK$-theory using the Cayley transform.
Our map  provides an alternative 
approach to the well-known isomorphism which the 
reader can find in \cite[Section 17.5]{Blackadar}. 
To highlight its usefulness, we show that our map is well-suited for 
the constructive form of the Kasparov product which is based on 
unbounded $KK$-cycles. 
The more general case of graded algebras and real structures is
more efficiently handled with van Daele $K$-theory and will be 
studied in Section \ref{sec:DK_KK}.

\subsection{The Cayley transform on ungraded Hilbert modules}

Here we briefly recall and expand on some results from~\cite[Chapter 9, 10]{Lance}.

\begin{prop}\label{prop:Cayley_properties}
Let $A$ be a $C^*$-algebra,  $X_A$ a Hilbert $C^*$-module and 
$T$ a self-adjoint regular (possibly unbounded) and right-$A$-linear operator on $X_A$. 
Then 
$$
\Cd(T):= (T+i)(T-i)^{-1} \in \End_A(X)
$$
is a unitary operator. If $T$ has compact resolvent then $\Cd(T)-1 \in \End_A^0(X)$.
Similarly, if $V \in \End_A(X)$ is unitary, then 
$$
\Ci(V) = i(V+1)(V-1)^{-1}
$$
is a (possibly unbounded) self-adjoint regular operator with domain 
$(V-1)X_A$. If $V-1 \in \End_A^0(X)$, then $\Ci(V)$ has compact resolvent.
\end{prop}
\begin{proof} 
The proof of self-adjointness and regularity can be found in~\cite[Chapter 10]{Lance}. 
If $T$ has compact resolvent, then 
$$
   (T+i)(T-i)^{-1} - 1 = 2i(T-i)^{-1} \in \End_A^0(X).
$$
Similarly, if $V-1 \in \End_A^0(X)$, then  a short calculation yields
$$
1+\Ci(V)^2=4\big((V-1)(V^*-1)\big)^{-1},
$$
whence $(1+\Ci(V)^2)^{-1} \in \End_A^0(X)$ and so is compact as an endomorphism. 
 Taking the square root remains inside the compact operators. 
\end{proof}
\begin{rmk}
\label{rmk:cay-spec}
Note that if $T$ is invertible, then $-1$ is not in the spectrum of  $\Cd(T)$. 
If $T$ is bounded, then $1$ is not in the spectrum of  $\Cd(T)$.
\hfill $\diamond$
\end{rmk}
Despite the suggestive notation, $\Cd$ and $\Ci$ are not complete inverses of each other. If 
$U \in \End_A(X)$ is unitary, then $\Cd \circ \Ci(U)$ is the restriction of $U$ to the $C^*$-module 
$\ol{(U-1)X_A}$, the closure of $\Dom(\Ci(U))$ in $X_A$, and may not recover all of $X_A$ in general. 
However, we can recover the 
essential information of $U$ at the level of $K$-theory classes.

\begin{lemma} \label{lem:circle}
Let $A$ be a unital $C^*$-algebra and $U \in A$ unitary. Define the ideal 
$J_U = \ol{A(U-1)A}$. Then $U$ defines a class in $[U|] \in K_1(J_U)$.
With
$\iota_{U}: J_U \hookrightarrow A$ the inclusion,
$(\iota_{U})_\ast ([U|])=[U]\in K_1(A)$.
\end{lemma}
\begin{proof}
That  $[U]\in K_1(A)$ defines $[U|]\in K_1(J_U)$ follows
since
$U|=1\bmod J_U$. Let $q:\,A\to A/J_U$ be the quotient map,
and observe that $[q(U|)]=[1]=0\in K_1(A/J_U)$. So $q(U|)$ is stably homotopic
to $1$, whence there exists $w\in M_{n+1}(A/J_U)$ a unitary in the connected component
of the identity such that
$$
w(q(U|)\oplus 1_n)w^{-1}=1_{n+1}.
$$
Now lift $w$ to a unitary $\tilde{w}$  over $A$ in the connected component of the identity. 
Then
$$
\tilde{w}(U|\oplus 1_n)\tilde{w}^{-1}=\tilde{w}(U\oplus 1_n)\tilde{w}^{-1}=1_{n+1}\,\bmod J
$$
and so
\[
\iota_*([U|])=\iota_*([\tilde{w}(U|\oplus 1_n)\tilde{w}^{-1}])=\iota_*([\tilde{w}(U\oplus 1_n)\tilde{w}^{-1}])=[\tilde{w}(U\oplus 1_n)\tilde{w}^{-1}]=[U]\in K_1(A).\qedhere
\]
\end{proof}

\subsection{The Cayley transform and odd $K$-theory}

Let $A$ be a $C^*$-algebra and consider 
a unitary element $u\in M_N(A)$ (or $M_N(A^\sim)$ if $A$ is non-unital).
We consider the (inverse) Cayley transform $\Ci(u)$ of $u$ as an
unbounded self-adjoint regular operator 
on a suitable $A$-module. Namely, 
$$
\Dom(\Ci(u))=(u-1)A^N, \quad \Ci(u)v=i(u+1)(u-1)^{-1}v,\qquad v\in (u-1)A^N.
$$
Let $\ol{(u-1)A^N_A}$ be the closure of $\Dom(\Ci(u))$ in $A^N$. 

\begin{prop} 
\label{prop:K_1_to_KK}
The triple $(\C,\ol{(u-1)A^N_A},\Ci(u))$ is an unbounded odd Kasparov module. 
\end{prop}

\begin{proof}
Most of the result immediately follows from Proposition \ref{prop:Cayley_properties}.
We compute that $(1+\Cc(u)^2)^{-1} 
= \frac{1}{4}(u-1)(u^*-1)\in M_N(A)$ (as opposed to $M_N(A^\sim)$) 
and so is compact as an endomorphism. 
As in the proof of Proposition \ref{prop:Cayley_properties}, taking the square root 
remains inside the compact operators.
\end{proof}

\begin{thm} 
\label{thm:complex_cayley}
The map $K_1(A)\to KK^1(\C,A)$ defined by sending a unitary to the class of its Cayley
transform, 
$$
   K_1(A) \ni [u] \mapsto \big[ ( \C, \, \ol{(u-1)A^N_A}, \, \Ci(u) ) \big] \in KK^1(\C, A),
$$
is well-defined and an isomorphism. The inverse is provided by the Cayley transform 
on densely defined self-adjoint and regular operators,
$$
KK^1(\C,A)\ni[(\C, \, X_A, \, T)] \mapsto \iota_* \circ \zeta_X([\Cd(T)])
$$
where $\zeta_X:\,K_1(\End^0_A(X))\to K_1(J_X)$ is the Morita isomorphism for 
$J_X = \ol{\mathrm{span}(X | X)_A}$ and $\iota:\,J_X\to A$ is the inclusion.
\end{thm}

\begin{proof}
Additivity is clear, and if $u=1_A$ 
then the module $X_A$ is the zero module, and $\Ci(u)=0$.
Hence the resulting class is zero in $KK^1(\C,A)$.

Now suppose that $\{u_t\}_{t\in[0,1]}$ is a norm 
continuous path of unitaries over $A^\sim$. Define
a right $A\otimes C([0,1])$-module by
$$
\mathcal{X}_{A\otimes C([0,1])}=\{f:[0,1]\to A^N:\,f(t)\in \ol{(u_t-1)A^N_A} \ \mbox{for all}\ t\in[0,1]\}
$$
Similarly we define
$$
\Dom(\Ci(u_\bullet))=\{f:[0,1]\to A^N:\,f(t)\in \Dom(\Ci(u_t))\ \mbox{for all}\ t\in[0,1]\}
$$
and $(\Ci(u_\bullet)f)(t)=\Ci(u_t)f(t)$. Then because $u_\bullet$ is a unitary 
over $A^\sim\ox C([0,1])$, the triple 
$(\C,\mathcal{X}_{A\otimes C([0,1])},\Ci(u_\bullet))$ is an odd  
Kasparov $\C$-$A\otimes C([0,1])$-module. Composing with the evaluation 
map on $C[0,1]$,
the Kasparov modules
$(\C,X_A,\Ci(u_0))$ and $(\C,X_A,\Ci(u_1))$
define the same class in $KK^1(\C,A)$, and our map is well-defined
and injective.

For surjectivity we will display the inverse map.
Given an odd unbounded 
Kasparov $\C$-$A$-module 
$(\C,X_A,T)$ with $T$ self-adjoint and regular, 
we define $u = \Cd(T) =(T+i)(T-i)^{-1}$. 
Then $u-1_X=2i(T-i)^{-1}$ is compact by Proposition \ref{prop:Cayley_properties}, 
so that $u\in (\End^0_A(X))^\sim\subset (\K\ox A)^\sim$. 
Provided we obtain a well-defined map,  
additivity is obvious. 

Suppose that
$(\C,X_A,T)$ represents zero. Then (modulo degenerate Kasparov modules) 
the bounded 
transform $F_T=T(1+T^2)^{-1/2}$ of $T$ is 
operator homotopic to an invertible, since the 
only obstruction to triviality of the module is $1_X-F_T^2$.
This means that we can suppose that $T$ is invertible, 
and then we see that
$(T+i)(T-i)^{-1}$ has an arc containing $-1$ 
in its resolvent set by Remark \ref{rmk:cay-spec}. 
Thus the unitary $(T+i)(T-i)^{-1}$
is homotopic to $1_X$, and so represents zero in $K_1(A)$.

If $(\C,X_A,T_t)$ is an operator homotopy, $t\in[0,1]$, then we obtain
a norm continuous path of unitaries $\Cd(T_t)$. This is done in the Real case in 
Lemma \ref{lem:KK_to_Roe}, see Equation \eqref{eq:opr_htpy_OSU_htpy}, and the proof
carries over. Hence the unitaries $\Cd(T_0)$ and $\Cd(T_1)$ define the same
class in $K_1(\End_A^0(X))$.

We define the inverse map to be 
$$
KK^1(\C,A)\ni[(\C,X_A,T)]\mapsto \iota_* \circ \zeta_X([\Cd(T)])
$$
where $\zeta_X:\,K_1(\End^0_A(X))\to K_1(J_X)$ is the Morita isomorphism 
(defined analogously to Lemma \ref{lem:DK-morita}), $J_X=\ol{{\rm span}(X|X)_A}$
and $\iota:\,J_X\to A$ is the inclusion.

The inverse map is well-defined and so we now check that the two maps are indeed
mutual inverses.  
For $u\in M_N(A)$ we find that our recipe gives
$\Cd(\Ci(u))=u$ as an element of $\End_A^0\big(\overline{(u-1)A^N_A} \big)^\sim$.
Applying $\zeta_{\overline{(u-1)A^N}}$ to the class $[u]\in K_1\big(\End^0_A(\ol{(u-1)A^N})\big)$
gives $[u]^J\in K_1(J_u)$ where $J_u=\overline{A^N(u-1)A^N}.$ By Lemma \ref{lem:circle}, 
$\iota_*([u]^J)=[u]\in K_1(A)$.

In the other direction, given an odd Kasparov module
$(\C,X_A,T)$, we have that $\Ci(\Cd(T))=T$ as operators on $X$. 
To prove that we obtain an isomorphism of groups, we consider the two
homomorphisms
$\Ci_A: K_1(A) \to KK^1(\C,A)$ and 
$\Ci_{J_X}: K_1(J_X) \to KK^1(\C, J_X)$ defined by the Cayley transform.
These two homomorphisms are related by
\begin{equation}
\Ci_A\circ \iota_* \circ \zeta_X([\Cd(T)])=\iota_*\circ \Ci_{J_X} \circ \zeta_X([\Cd(T)])=\iota_* \circ \zeta_X^{KK} \circ \Ci_{\End^0_A(X)}([\Cd(T)])
\label{eq:two-cayleys}
\end{equation}
where 
$\calC^{-1}_{\End_A^0(X)} :K_1(\End_A^0(X)) \to KK^1(\C, \End^0_A(X))$ and  
$\zeta_X^{KK}=\cdot\hox_{\End^0_A(X)}[(\End^0_A(X),X_{J_X},0)]$ is the 
product with the Morita equivalence bimodule. The first equality essentially follows
from Lemma \ref{lem:circle}, while the second comes from a direct calculation and
simple homotopy. The details in the Real case are presented in 
Lemmas \ref{lem:Cay-iota} and \ref{lem:Cay-zeta}.

Applying \eqref{eq:two-cayleys} to the class $[(\C,X_A,T)]\in KK^1(\C,A)$, 
we have
\begin{align*}
\Ci_A\circ \iota_* \circ \zeta_X([\Cd(T)])&=
\iota_* \circ \zeta_X^{KK}[(\C, \, \overline{(T-i)^{-1}\End_A^0(X)}_{\End_A^0(X)}, \, \calC^{-1}\circ\calC( T))]\\
&=\iota_* \circ \zeta_X^{KK}[(\C,\, \End_A^0(X)_{\End_A^0(X)}, \, T)]  \\
&=\iota_*  [(\C,\,X_{J_X}, \, T)]
=[(\C,\,X_A, \, T)],
\end{align*}
where we used that $\calC(T)-1 = 2i(T-i)^{-1}$ and the resolvent of $T$ has dense range. 
\end{proof}

\begin{cor}
\label{cor:R-gen}
The generator of $KK^1(\C,C_0(\R))=\Z$ is 
represented by the unbounded Kasparov module
$$
\left(\Cliff_1, \, \begin{pmatrix}C_0(\R)\\C_0(\R)\end{pmatrix}_{C_0(\R)}, \,
\begin{pmatrix} 0 & x\\ x & 0\end{pmatrix}\right),
$$
where the odd generator of $\Cliff_1$ acts by $\begin{pmatrix} 0 & -1\\ 1 & 0\end{pmatrix}$.
\end{cor}
\begin{proof}
The generator of $K_1(C_0(\R))$ is given by
$$
u=e^{-2i\tan^{-1}(x)+i\pi}=-e^{-2i\tan^{-1}(x)}
$$
since under the isomorphism $C_0((-1,1))\to C_0(\R)$ 
given by $t\mapsto \tan(\pi t/2)$ the unitary
$u$ is mapped onto the generator of $K_1(C_0(-1,1))$ 
described in \cite[Example 4.8.7]{HR}. We then
calculate the Cayley transform of $u$ and find $\Ci(u)$ 
is the operator of multiplication by $x$. 
Representing the odd Kasparov module as a class in 
$KK^1(\C,A)$ gives the desired result.
\end{proof}

\subsection{The graph projection, even $K$-theory and index pairings}
\label{subsec:graph_proj}

Theorem \ref{thm:complex_cayley} is a special case of a more general result explicitly relating 
unbounded Kasparov theory with van Daele $K$-theory studied in Section \ref{sec:DK_KK}. 
Before going on to describe this relation in general, we complete the picture for complex ungraded algebras by providing an isomorphism $KK(\C,A)\cong K_0(A)$ which is compatible with both the Cayley transform and
the Kasparov product. The following example provides important motivation.

\begin{example} \label{ex:Bott_class}
The generator of $K_0(C_0(\R^2))$ is well-known to be the external product
of the generator of $K_1(C_0(\R))$ with itself \cite[Example 4.8.7]{HR}. 
It is also known to be the class $[p_B]-[1]$ where
$[1]$ denotes the class of the trivial line bundle and 
$$
p_B(x,y)=\frac{1}{1+x^2+y^2}\begin{pmatrix} 1 & x-iy\\ x+iy & x^2+y^2\end{pmatrix}
$$
is the Bott projector. Using the unbounded external Kasparov 
product (see \cite{BJ}) 
we easily find 
that the external product 
of two copies of the generator of 
$KK^1(\C,C_0(\R))$ from Corollary \ref{cor:R-gen} to be
represented by
$$
\Bigg(\C, \, \begin{pmatrix}C_0(\R^2)\\ C_0(\R^2)\end{pmatrix}_{C_0(\R^2)}, \, T=\begin{pmatrix} 0 & x-iy\\ x+iy & 0\end{pmatrix}\Bigg).
$$
To relate these two representatives, we recall the general
representation of $K$-theory classes below.
\hfill $\diamond$
\end{example}

Let $A$ be a complex $C^*$-algebra. Any class in $K_0(A)$
can be represented by a difference $[p]-[q]$
where for some $N\in\N$ and unitisation $A_b$
we have projections $p,\,q\in M_N(A_b)$ and some $W\in M_N(A_b)$
such that $WpW^*-q\in M_N(A)$. Excision says that we can always
take $A_b=A^\sim$, the minimal unitisation, and $W$ can be taken
to be a lift of a partial isometry over $A^\sim/A$.

More generally, Morita invariance of $K$-theory says that
if $p,\,q\in \Mult(A\ox\K)$ are projections in the multiplier algebra such that
there exists $W\in \Mult(A\ox\K)$ unitary with $WpW^*-q\in A\ox\K$,
then $[p]-[q]\in K_0(A\ox\K)\cong K_0(A)$ and all classes are of this form.

\begin{thm} \label{thm:even_K_to_KK}
If $T:\Dom(T)\subset X_A\to  X_A$ is an (unbounded regular) 
odd self-adjoint operator on the graded
$C^*$-module $X_A$ with compact resolvent, 
let $P_{T_+}\in \End_A(X)$ be the 
graph projection of $T_+: X_+\to X_-$ 
and $P_{ X_-}$ the
projection onto the negative part of $X_A$. Then
$$
[P_{T_+}]-[P_{X_-}]
$$
defines a class in $K_0(A)$. The map 
$$
KK(\C,A)\ni[(\C,X_A,T)]\mapsto [P_{T_+}]-[P_{X_-}]\in K_0(\End_A^0(X))\cong K_0(A)
$$
provides an inverse 
to the isomorphism $K_0(A) \to KK(\C, A)$
\begin{equation} \label{eq:K-iso}
    [p]-[q] \mapsto \Big[ \Big(\C, \, p(A^\sim)^N_A \oplus q(A^\sim)^N_A, \, \begin{pmatrix}0 & W^*\\ W & 0\end{pmatrix} \Big) \Big], 
     \quad WpW^*-q\in M_N(A).
\end{equation}
\end{thm}

\begin{proof}
That 
the groups $K_0(A)$ and $KK(\C,A)$
are isomorphic is \cite[Theorem 3, Section 6]{Kasp} in the unital case
and \cite[Corollary 2, Section 6]{Kasp} in general. The
remainder of the proof is a careful reiteration of 
\cite[Corollary 2, Section 6]{Kasp}.
By \cite{BJ} every Kasparov class in $KK(\C,A)$ with $A$ $\sigma$-unital 
can be represented by an unbounded Kasparov module 
$(\C,X_A,T)$, where unbounded means
`not necessarily bounded'. 

Given $(\C, X_A, T)$, the graph projection of $T_+$ is
\begin{align}
P_{T_+}&=\begin{pmatrix} (1+T_+^*T_+)^{-1} & (1+T_+^*T_+)^{-1}T_+^*\\
T_+(1+T_+^*T_+)^{-1} & T_+(1+T_+^*T_+)^{-1}T_+^*\end{pmatrix}\nonumber\\
&=\begin{pmatrix} (1+T_+^*T_+)^{-1} & (1+T_+^*T_+)^{-1}T_+^*\\
T_+(1+T_+^*T_+)^{-1} & P_{X_-}-(1+T_+T_+^*)^{-1}\end{pmatrix}.
\label{eq:graph-proj}
\end{align}
Hence $P_{T_+}=P_{X_-}\bmod \End_A^0(X)$, and so 
$[P_{T_+}]-[P_{X_-}]$ defines
a $K$-theory class for $\End_A^0(X)$ and so a class in $K_0(A)$. 
In addition there is an isomorphism
$$
v_+=\big((1+T^*_+T_+)^{-1/2},(1+T^*_+T_+)^{-1/2}T_+^*\big):\,\begin{pmatrix} X_+\\ X_-\end{pmatrix}\to X_+
$$
with $v_+v_+^*=1_{X_+}$ and $v_+^*v_+=P_{T_+}$. Thus {\em provided
that $(1+T^*_+T_+)^{-1/2}T_+^*\in \End_A^0(X)$ we have} 
\begin{equation}
[P_{T_+}]-[P_{X_-}]=[P_{X_+}]-[P_{X_-}].
\label{eq:diff-mods}
\end{equation}

To show that Equation \eqref{eq:K-iso} is an inverse to the graph projection map, 
we need to consider classes in $K_0(A\otimes \calK)$. 
The analogue of the map  \eqref{eq:K-iso} for $K_0(A\ox\K)\to KK(\C,A)$
can be written (for a separable Hilbert space $\H$) as 
\begin{equation} 
\label{eq:stabilised_K_map}
K_0(A\ox\K)\ni[p]-[q]\mapsto\Big[\Big(\C, \, p(\H\ox A^\sim)_A\oplus q(\H\ox A^\sim)_A, \, \begin{pmatrix} 0 & W^*\\ W & 0\end{pmatrix}\Big)\Big].
\end{equation}
We can consider the modules as being over $A$ or $A^\sim$, as the difference lives in $K_0(A)$.

Now let $(\C,X_A,T)$ be an even Kasparov module over $A$ and
map it to $[P_{T_+}]-[P_{X_-}]\in K_0(\End_A^0(X))\cong K_0(A)$. 
The formula \eqref{eq:graph-proj} for the graph projection shows 
that we obtain a Kasparov module
$$
\left(\C,\begin{pmatrix} P_{T_+}\begin{pmatrix} X_+\\X_-\end{pmatrix}\\ X_-\end{pmatrix}_A, \begin{pmatrix} 0 & 0 & 0\\ 0 & 0 & 1\\ 0 & 1 & 0\end{pmatrix} \right).
$$
One can now check directly that the map
$[(\C,X_A,T)]\mapsto [P_{T_+}]-[P_{X_-}]$
is well-defined and provides an inverse to 
the isomorphism in Equation \eqref{eq:stabilised_K_map}.
\end{proof}

\begin{rmk}
In the next section we will consider a more general 
approach which is compatible with real structures and
gradings. We will compare the more general
method with the graph projection approach
in Remark \ref{rmk:graphproj-vs-general}.
\hfill $\diamond$
\end{rmk}

Returning to Example \ref{ex:Bott_class} and  
 the Kasparov module
$$
\Bigg(\C, \, \begin{pmatrix}C_0(\R^2)\\ C_0(\R^2)\end{pmatrix}_{C_0(\R^2)},\, T=\begin{pmatrix} 0 & x-iy\\ x+iy & 0\end{pmatrix}\Bigg),
$$
we obtain precisely $P_{x+iy}=p_B$ the Bott projection.
Observe that with this representative of the class $[p_B]-[1]$, we can not 
take the formal difference of modules
``$[X_+]-[X_-]$'', as 
$(x\pm iy)(1+x^2+y^2)^{-1/2}$ is not in the minimal
unitisation of $C_0(\R^2)$.
One way to think about this issue is to observe that
both $X_\pm$ are of the form $X_\pm=Y_\pm\ox_{C_0(\R)^\sim}C_0(\R)$,
with $Y_\pm$ finite projective $C_0(\R)^\sim$ modules. The modules
$Y_\pm$ can not both be trivial, as $x\pm iy$ would not provide an
operator between them: see \cite{RSi}. 
The graph projection of
the ``intertwining operator'' $x+iy$ provides the most direct way to
access the $K$-theoretic difference of these two (seemingly trivial) 
modules.

The graph projection
approach has been exploited many times before in 
noncommutative approaches to index theory, see for example~\cite{CMHyperbolic, vanErp, Natsume, NestTsygan}.

The Cayley transform can be conveniently 
used to express the index pairing between $K_1$-theory and $K^1$-homology, 
as this pairing becomes a Kasparov product 
when the transform is applied to the $K_1$-class. We describe such pairings and products in the Appendix. Together 
with other known properties of the index 
pairing, we obtain the following result.
\begin{prop}
Let $(\A,\H,\D)$ be an odd spectral triple with 
$\overline{\A\H}=\H$, and let $u\in M_N(\A^\sim)$ be 
unitary with Cayley transform $\Ci(u)$. 
Suppose that we have an approximate unit $v_n\in C^*((u-1),(u^*-1))$
such that $[\D,v_n](1-u^*)\to0$ $*$-strongly.
Then with $\tilde{\D}=\D|_{(u-1)\Dom (\D)}$,
the index pairing between $K_1(A)$ and the spectral triple is given by
\begin{align*}
\langle[u],[\D]\rangle &=sf(\D,u\D u^*)={\rm Index}(PuP-(1-P)) \\
&={\rm Index}(\Ci(u)+i\tilde{\D}:\,\ol{(u-1)\H^N} \to \ol{(u-1)\H^N}),
\end{align*}
where $sf(\D,u\D u^*)$ is the spectral flow, \cite{Ph1,Ph2,CP2}, and 
$P=\chi_{[0,\infty)}(\D)$ is the non-negative spectral projection.
\end{prop}
\begin{proof}
We know from \cite{Ph1,Ph2,CP2,KNR} (for instance)
that the index pairing of $[u]$ and the class of $(\A,\H,\D)$ is given by
$$
\langle[u],[\D]\rangle=sf(\D,u\D u^*)={\rm Index}(PuP-(1-P)).
$$
Applying the Cayley transform to the unitary $u$, 
the index pairing becomes the Kasparov product of the class introduced in 
Proposition \ref{prop:K_1_to_KK} with the $KK$-class defined by the spectral triple, 
$$
[(\C,\overline{(u-1)A_A^N},\Ci(u))]\ox_A[(\A,\H,\D)]
$$
followed by the standard 
isomorphism $KK(\C,\C)\ni[(\C,\H,T)]\mapsto {\rm Index}(T_+)\in\Z$.
Observing that $\overline{(u-1)A^N}\ox_A\H=\overline{(u-1)\H^N}$, a representative 
of this product is given by
$$
\left(\C,\, \overline{(1-u)\H^N} \otimes \C^2, \,
\begin{pmatrix} 0 & \Ci(u)-i\tilde{\D}\\ \Ci(u)+i\tilde{\D} & 0\end{pmatrix}\right),
$$
see Theorem \ref{thm:product}. Taking the index, 
\begin{equation*}
{\rm Index}(PuP)=sf(\D,u\D u^*)={\rm Index}(\Ci(u)+i\tilde{\D}:\,\ol{(u-1)\H^N} \to \ol{(u-1)\H^N}).  \qedhere
\end{equation*}
\end{proof}

More general odd Kasparov products can be handled by Theorem \ref{thm:product} 
in the appendix.

\begin{example}
Let $A=C(S^1)$ and take the usual spectral triple $\mu=(C^\infty(S^1),L^2(S^1),\frac{1}{i}\frac{d}{d\theta})$ 
for the circle.
For $u=e^{-i\theta}$, $-\pi\leq\theta\leq\pi$, we can represent the Cayley
transform on the domain of functions vanishing (fast enough) at $\theta=0$ by
$\Ci(u)=\cot(\theta/2)$.
Example \ref{eg:zed} checks that 
the product $[e^{-i\bullet}]\ox_{C(S^1)}\mu$ is represented  by the spectral triple
$$
\left(\C, \, L^2(S^1)\otimes\C^2, \,
\begin{pmatrix} 0 & -\frac{d}{d\theta}+\cot(\theta/2)\\ \frac{d}{d\theta}+\cot(\theta/2)
& 0\end{pmatrix}\right).
$$
The equation $ \frac{dy}{d\theta}\pm\cot(\theta/2)y=0$ is satisfied 
for the function $y =C(\sin(\theta/2))^{\pm 2}$ with 
$C$ constant,
and while $(\sin(\theta/2))^{ 2}$ is square-summable, $(\sin(\theta/2))^{- 2}$ is not. 
Hence $\Ker( \frac{d}{d\theta}+\cot(\theta/2)) = \mathrm{span}\{(\sin(\theta/2))^{ 2}\}$ and 
 ${\rm Index}\big(\frac{d}{d\theta}+\cot(\theta/2)\big)=1$. We have therefore obtained
the correct index, giving a check of signs.
\hfill $\diamond$
\end{example}


\section{A graded Cayley isomorphism between $DK$-theory and $KK$-theory} \label{sec:DK_KK}

In this section we define a graded version of the Cayley transform on $C^*$-modules 
which allows us to define an explicit map between van Daele $K$-theory and $KK$-theory. 
As in the complex ungraded case, we show this map is an isomorphism and consider some 
applications of this result in Sections \ref{sec:app-real} and \ref{sec:TI_application}.

\subsection{The Cayley transform on graded Hilbert modules}

Here we extend results on the Cayley transform to \emph{odd} operators on 
 \emph{graded} Hilbert $C^*$-modules. Throughout this section 
we assume that the 
Hilbert $C^*$-module $X_A$ is balanced graded, i.e., 
$\End_A(X)$ has at least one OSU.

\begin{defn}
Let $X_A$ be a countably generated and balanced graded $C^*$-module with OSU $e \in \End_A(X)$. Suppose 
$T$ is an odd self-adjoint regular (possibly unbounded) right-$A$-linear operator that 
anti-commutes with $e$. Define
$$
\Cd_e(T) := e(T+e)(T-e)^{-1}
$$
as the graded Cayley transform of $T$ relative to $e$.
\end{defn}

We first note that $\Cd_e(T)$ is well-defined as  $(T\pm e)^2 = 1 + T^2$, so $(T-e)^{-1}$ is bounded. Moreover the range of 
$(T-e)^{-1}$ is $\Dom(T)$, and so $\Cd_e$ is everywhere
defined and bounded.

\begin{lemma}\label{lem-Cd1}
The operator 
$\Cd_e(T)$ is an odd self-adjoint unitary on $X_A$. 
Moreover, if $(1+T^2)^{-1}$ is compact then $\Cd_e(T)-e$ is compact.
\end{lemma}
\begin{proof}
Clearly $\Cd_e(T)$ is odd, and since $eT=-Te$ we have
$T(T-e)^{-1}=(T+e)^{-1}T$ and $e(T-e)^{-1}=(-T-e)^{-1}e$. Then
we see that $\Cd_e(T)$ is self-adjoint by computing
\begin{align*}
&(e(T+e)(T-e)^{-1})^*=(T-e)^{-1}(T+e)e
=T(T+e)^{-1}e-e(T+e)^{-1}e\\
&=T(e(T+e))^{-1} - (e(T+e)e)^{-1} 
 = (Te-1)(-T+e)^{-1} 
=e(T+e)(T-e)^{-1}.
\end{align*}
Now we compute the square by
\begin{align*}
\Cd_e(T)^2&=e(T+e)(T-e)^{-1}e(T+e)(T-e)^{-1}
=e(T+e)e(-T-e)^{-1}(T+e)(T-e)^{-1}\\
&=-e(T+e)e(T-e)^{-1}
=-e^2(-T+e)(T-e)^{-1}=1.
\end{align*}
We have  
$$ \Cd_{e}(T)-e =e((T+e)-(T-e))(T-e)^{-1}=2(T-e)^{-1}. $$ 
Since compact operators are closed under continuous functional calculus and an ideal, compactness of 
$(1+T^2)^{-1}$ implies compactness of $|T-e|^{-1}$ which then implies compactness of $(T-e)^{-1}$. 
As a consequence, if $(1+T^2)^{-1}$ is compact then $\Cd_e(T)-e$ is compact as well. 
\end{proof}

\begin{rmk}
\label{rmk:cay-grad-ungr}
We can recover the ungraded Cayley transform as a special case of our 
graded map. Namely, if $Y_A$ is an ungraded $C^*$-module, we consider 
$X_A = Y_A \otimes \C^2$ with the obvious grading. Then for $S$ a self-adjoint regular 
operator on $Y_A$, we consider $T = S \otimes \sigma_2$ and 
$e=1\otimes \sigma_1$.
 Using $\begin{pmatrix} 0 & a\\ b & 0 \end{pmatrix}^{-1} = \begin{pmatrix} 0 & b^{-1}\\ a^{-1} & 0 \end{pmatrix}$ we obtain
\begin{align*}
  \Cd_e(T) &= \begin{pmatrix} 0 & 1 \\ 1 & 0\end{pmatrix}\begin{pmatrix} 0 & -iS+1 \\ +iS+1& 0 \end{pmatrix}\begin{pmatrix} 0 & (iS-1)^{-1} \\ (-iS-1)^{-1} & 0\end{pmatrix} 
  = \begin{pmatrix} 0 & \Cd(S)^* \\ \Cd(S) & 0\end{pmatrix}.
\end{align*}
If $(1+S^2)^{-1} \in \End_A^0(Y)$, then  $\Cd(S) - 1 \in \End_A^0(Y)$ and $\Cd_e(T) - 1\otimes \sigma_1 \in \End_A^0(X)$.
\hfill $\diamond$
\end{rmk}

For the inverse Cayley transform, we again consider self-adjoint odd unitaries 
relative to a base point OSU $e$.

\begin{defn}
Let $X_A$ be a balanced graded $C^*$-module with $U, e \in \End_A(X)$ odd self-adjoint unitaries. 
Define
$$
\Ci_{e}(U) :=  e(U+e)(U-e)^{-1}
$$
with domain  $(U-e)X_A$.
\end{defn}

We let $\ol{(U-e)X_A}$ be the closure of 
$\Dom(\Ci_{e}(U))$ in $X_A$ where, by construction, 
$\Ci_e(U)$ is densely defined. 

\begin{lemma}\label{lem-Cd2} 
Let $X_A$ be a countably generated and balanced graded $C^*$-module with OSU $e \in \End_A(X)$. 
If $U \in \End_A(X)$ is an odd self-adjoint unitary, the 
operator $\Ci_{e}(U)$ is an odd self-adjoint regular (possibly unbounded) operator on $\overline{(U-e)X_A}$ which anti-commutes with $e$. 
Moreover, if $U-e$ is compact, then $(1+\Ci_e(U)^2)^{-1/2}$ is a compact operator on $\overline{(U-e)X_A}$.  
\end{lemma}
\begin{proof}
It is immediate that $ \Ci_{e}(U)$ is odd.
Since $e$ and $U$ are OSUs we have $e(U\pm e) = (e\pm U)U$.
Hence  $e(U-e)X_A =(e-U) U X_A = (U-e) X_A$ showing that the domain $(U-e)X_A$
is preserved by $e$.
Furthermore, for any $\psi\in (U-e)X_A$,
$$e(U\pm e)^{-1}\psi = \big((U\pm e)e^{-1}\big)^{-1}\psi = \big(U(e\pm U)\big)^{-1}\psi = (e\pm U)^{-1}U\psi.$$
Therefore  $e$ anti-commutes with $(U+e)(U-e)^{-1}$ and hence also with 
$\Ci_{e}(U)$ (on the domain) since
$$
e\Ci_e(U)=e^2(U+e)(U-e)^{-1}=e(e+U)U(U-e)^{-1}=e(e+U)(e-U)^{-1}e
=-\Ci_e(U)e.
$$

To show that $\Ci_{e}(U)$ is self-adjoint and regular, 
we employ \cite[Theorem 10.4]{Lance}. Consider the operator 
$F = \frac{1}{2}\calC_e(U)(2-Ue-eU)^{1/2} = \frac{1}{2}e(Ue+1)(Ue-1)^{-1}(2-Ue-eU)^{1/2}$.
This operator is self-adjoint by direct computation using the 
normality of $Ue$, has
norm bounded above by 1, and also
\begin{align*}
F^2&=\frac{1}{4}(2+eU+Ue)\quad\mbox{and so}\\
1-F^2
&=\frac{1}{2}-\frac{1}{4}(eU+Ue)
=\frac{1}{4}(2-eU-Ue) 
= \frac{1}{4}(eU-1)(Ue-1)
\end{align*}
which shows that $1-F^2$ is positive. 
The operator $\Ci_{e}(U)$ is defined on the  range of 
$(U-e)$, which by the unitarity of $e$ is the same as the range of $(Ue-1)$. 
The operator $(1-F^2)^{1/2}=\frac{1}{2}|(Ue-1)|$ then has dense range equal to
$(U-e)X_A$.

We can now apply~\cite[Theorem 10.4]{Lance} which implies that the operator
$F(1-F^2)^{-1/2}$ is a densely defined, regular 
self-adjoint operator on $\ol{(U-e)X_A}$. This operator is
\begin{align*}
F(1-F^2)^{-1/2}
&=\frac{1}{2}e(Ue+1)(Ue-1)^{-1}(2-Ue-eU)^{1/2} \Big(\frac{1}{4}(2-Ue-eU)\Big)^{-1/2}\\
&=e(Ue+1)(Ue-1)^{-1} = e(U+e)(U-e)^{-1} = \Ci_{e}(U)
\end{align*}
and so we find that $\Ci_{e}(U)$ is regular and self-adjoint.

Next we compute $1+\Ci_{e}(U)^2 = 1+ \Ci_{e}(U)^* \Ci_{e}(U)$, where
\begin{align} \label{eq:Cayley_bdd_transform}
   1+ \Ci_{e}(U)^* \Ci_{e}(U) &= 1+ (U-e)^{-1}(U+e)(U+e)(U-e)^{-1} \nonumber \\
    &= 1+ (2+ eU+Ue)(2-eU-Ue)^{-1} \nonumber \\
    &= 4(2-eU-Ue)^{-1} = 4(U-e)^{-2}.
\end{align}
Therefore $(1+\Ci_{e}(U)^2)^{-1/2} = \frac{1}{2}|U-e|$, which is
compact if $U-e$ is compact. 
\end{proof}

\begin{rmk}
\label{rmk:cay-gr-ungr-2}
We again show how to recover the ungraded case. 
Take $Y_A$ ungraded and $X_A = Y_A \otimes \C^2$. Then odd self-adjoint 
unitaries $U$ take the form $U = \begin{pmatrix} 0 & V^* \\ V& 0 \end{pmatrix}$ with 
$V\in \End_A(Y)$ unitary. We then compute for $e = 1\otimes \sigma_1$,
\begin{align*}
\Ci_{e}(U) &=
\begin{pmatrix} 0 & 1 \\ 1 & 0\end{pmatrix}\begin{pmatrix} 0 & (V^*+1) \\ (V+1)& 0 \end{pmatrix}\begin{pmatrix} 0 & (V-1)^{-1} \\ (V^*-1)^{-1} & 0\end{pmatrix}  \\
  &=\begin{pmatrix} 0 & (V+1)(V-1)^{-1} \\ (V^*+1)(V^*-1)^{-1} & 0 \end{pmatrix} \\
  &= \Ci(V) \otimes \sigma_2
\end{align*}
with $\Ci(V)$ the ungraded Cayley transform of $V$. 
\hfill $\diamond$
\end{rmk}

\begin{rmk}
If $X_A$ is a balanced $\Z_2$-graded module over the non-trivially $\Z_2$-graded
algebra $A$, then it is not necessarily the case that the even and odd halves
$X_A=(X_+\oplus X_-)_A$ are isomorphic with $X_A=(X_+\oplus UX_+)_A$
and the isomorphism $U$ providing an OSU 
$\tilde{U}=(\begin{smallmatrix} 0 & U^*\\ U & 0\end{smallmatrix})$ in $\End_A(X)$.
The issue is that $\tilde{U}$ need not be adjointable.
\hfill $\diamond$
\end{rmk}

In a sense which we will make precise, 
the maps $\Cd_e$ and $\Ci_e$ are mutual inverses.
\begin{prop} 
\label{prop:Cayley_inverses}
Let $X_A$ be a countably generated and balanced graded $C^*$-module with OSU $e \in \End_A(X)$. If $T$ is an odd self-adjoint regular operator which anti-commutes with $e$ then 
$\Ci_e\circ\Cd_e(T) = T$ on $X_A$. If $U\in \End_A(X)$ is an OSU,  
then $\Cd_e\circ\Ci_e(U)$ is the restriction of  $U$ to $\overline{(U-e)X_A}$. 
\end{prop}
\begin{proof}
Let $x$ be a right-$A$-linear operator on $X_A$ and $Y_A\subset X_A$ be a submodule on which $(x-e)^{-1}$ is well-defined. 
Apart from their domains, both expressions $\Ci_{e} \circ\Cd_e(x) $ and $\Cd_e\circ\Ci_e(x)$ are equal to 
$e\big(e(x+e)(x-e)^{-1}+e\big)\big( e(x+e)(x-e)^{-1}-e\big)^{-1}$.
As $(x+e)(x-e)^{-1}-1 = 2e (x-e)^{-1}$ the above expression is well-defined on $Y_A$ and given by 
$$
e\big(e(x+e)(x-e)^{-1}+e\big)\big( e(x+e)(x-e)^{-1}-e\big)^{-1} = x.
$$
For the first statement we substitute $x=T$ and $Y_A =X_A$, 
which is possible as $T$ anti-commutes with $e$. For the second statement we take $x=U$ and
$Y_A = \overline{(U-e)X_A}$. 
\end{proof}

As in the ungraded case, for
$U \in \End_A(X)$ unitary, $\Cd_e \circ \Ci_e(U)$ is the 
restriction of $U$ to the $C^*$-module 
$\ol{(U-e)X_A}$, which need not recover all of $X_A$ in general. 
However, we have a graded analogue of Lemma \ref{lem:circle} 
describing the $K$-theory classes.

\begin{lemma} \label{lem:graded-circle}
Let $A$ be a balanced graded unital $C^*$-algebra 
and $U,\,V \in A$  odd self-adjoint unitaries. Define the ideal 
$J = \ol{A(U-V)A}$. Then $U,\,V$ defines a class in 
$[U]^J-[V]^J \in DK(J)$, and with $\iota: J \hookrightarrow A$ the inclusion
$\iota_*([U]^J-[V]^J)=[U]-[V]\in DK(A)$.
\end{lemma}
\begin{proof}
We see that $U = V$ modulo $J$ and so $q_{\ast}([U]-[V]) \in DK(A/J)$ will be 
trivial for $q: A \to A/J$ the quotient map.
By \cite[Proposition 2.3]{vanDaele1} there is an 
even unitary $w$ over 
$A/J$ in
the connected component of the identity such that
$$
q\left(\begin{pmatrix} U & 0\\ 0 & V \end{pmatrix}\oplus\sigma_1\oplus\cdots\oplus\sigma_1\right)
=w\left(\begin{pmatrix} 0 & 1\\ 1 & 0\end{pmatrix}\oplus\sigma_1\oplus\cdots\oplus\sigma_1\right) w^{-1}.
$$
Since $w$ is in the connected component of the identity, it lifts to 
an even unitary $\tilde{w}$ in $A$ connected to the identity and such that
\begin{equation}
W:=\tilde{w}^{-1}\left(\begin{pmatrix} U & 0\\ 0 & V\end{pmatrix}\oplus\sigma_1^{\oplus n}\right)\tilde{w}
\label{eq:jay-unitary}
\end{equation}
is equal to $\sigma_1^{\oplus (n+1)}$ modulo $J\hox\Cl_{1,1}$.  
Hence the unitary $W$ of \eqref{eq:jay-unitary} is a unitary over
$J^\sim\hox\Cl_{1,1}$.
Now we have
$$
[W]-[\sigma_1^{\oplus (n+1)}]
=\left[\begin{pmatrix} U & 0\\ 0 & V\end{pmatrix}\right]-[\sigma_1]\in DK(A)
$$
because $\tilde{w}$ is an even unitary over $A$ connected to the identity.
Since $W$  is a unitary over
$J^\sim\hox\Cl_{1,1}$, we may define a class in $DK(J)$ 
by $[W]^J-[\sigma_1^{\oplus (n+1)}]^J$, where the ${}^J$ just
indicates that we regard these as unitaries 
over $J^\sim\hox\Cl_{1,1}$. Applying the inclusion map
\[
\iota_*([W]^J-[\sigma_1^{\oplus (n+1)}]^J)=[W]-[\sigma_1^{\oplus (n+1)}]
=\left[\begin{pmatrix} U & 0\\ 0 & V\end{pmatrix}\right]-[1\otimes \sigma_1] 
= [U]-[V] \in DK(A)
\]
where we have applied the isomorphism from \eqref{eq:ubiquitous-isomorphism} in 
the last equality.
\end{proof}

Lastly, we note that the graded Cayley transforms do not involve any complex 
structure and therefore are valid also for operators on real Hilbert modules.

\subsection{The isomorphism of $DK$ and $KK$}
\label{subsec:Cayley_DK_iso}

Here we use our results on the graded Cayley transform to construct an explicit isomorphism 
between the van Daele $K$-group $DK(A)$ and $KK$-group $KK(\Cl_{1,0},A)$. To cover the complex and real case simultaneously, 
we work with $KKR$-theory and real structures.

Using~\cite{BJ} and Remark \ref{rk:KK(Cl_1,A)_normalisation}, 
we represent any class in $KKR(\Cl_{1,0},A)$ 
by an unbounded Kasparov module $(e, X_A,T)$, where 
$X_A$ is a countably generated and balanced graded Real $C^*$-module,
$e^2$ acts as the identity and $e$ anti-commutes with $T$.

\newcommand{\Chris}{\mathfrak C}

Let us first describe a map $\Chris:\,DK(A) \to KKR(\Cl_{1,0}, A)$. 

\begin{lemma}
\label{lem:Roe_to_KK} 
Let $A$ be a unital and balanced graded algebra with $V, \, W \in M_n(A)$ OSUs. 
The inverse Cayley transform induces a homomorphism
$\Chris:\,DK(A, \mathfrak{r}_A  )\to KKR(\Cl_{1,0},A)$,
$$
 \Chris([V] - [W]) = \big[ (W, \, \ol{(V-W)A^n_A}, \, \Ci_{W}(V)) \big],
 \quad \Ci_{W}(V) = {W}(V+W)(V-W)^{-1}.
$$
If $A$ is non-unital and weakly balanced graded, and $B$ is any balanced graded unital algebra containing 
$A$ as a graded ideal, then we can use
the Cayley transform to define a homomorphism $\Chris:\,DK(B,B/A, \mathfrak{r}_A  )\to KKR(\Cl_{1,0},A)$
$$
  \Chris([V]-[W]) = \big[ ( W, \, \ol{(V-W)B^{n}_A}, \, \Ci_{W}(V)) \big]
  = \big[ ( W, \, \ol{(V-W)A^{n}_A}, \, \Ci_{W}(V)) \big].
$$
If $A$ is not balanced nor weakly balanced, let $\xi_{\Cl_{1,1}} = \cdot \hox [( \Cl_{1,1}, {\C^2}_{\C}, 0 )]$ be the isomorphism given by the
external Kasparov product with the 
class of the Morita equivalence $( \Cl_{1,1}, \C^2_{\C}, 0 )$.
Then given OSUs $V, W \in M_n(A^\sim \hox \Cl_{1,1})$, 
the Cayley transform defines a homomorphism $\Chris:\,DK(A,\mathfrak{r}_A  )\to KKR(\Cl_{1,0},A)$
\begin{align}
  \Chris([V]-[W]) &= \xi_{\Cl_{1,1}}\big[ ( W, \, \ol{(V-W)(A\hox\Cl_{1,1})^{n}}_{A\hox{\Cl_{1,1}}}, \, \Ci_{W}(V) ) \big]\nonumber\\
  &= \big[ ( W, \, \ol{(V-W)(A\ox\C^2)^{n}_A}, \, \Ci_{W}(V)) \big].\label{eq:needed-label}
\end{align}
\end{lemma}

Observe that the map \eqref{eq:needed-label} encompasses all cases stated in 
the proposition.

\begin{proof}
We deal with the unital and balanced case, 
as the other cases are only notationally more complex.

As $V$ and $W$ are odd and Real, $\Ci_{W}(V)$ is also  
odd and $\Ci_{W}(V)^\mathfrak{r} = \Ci_{W}(V)$. We note that 
$W(V \pm W) = (W \pm V)V$, so $\ol{(W - V)WA^n_A} = \ol{(V-W)A^n_A}$. 
Thus the action of both the generator of $\Cl_{1,0}$ and 
$V$ preserve the $A$-module and the domain of $\Ci_W(V)$. 
Applying Lemma \ref{lem-Cd2}, $\Ci_W(V)$ is self-adjoint, 
regular, anti-commutes with 
the $\Cl_{1,0}$-action and $(1+\Ci_W(V)^2)^{-1/2}$ is compact. 
Hence we obtain a Real (unbounded) Kasparov module so 
all that is left is to make sure that
the map $\Chris$ is well-defined.

Suppose that we have a continuous path of 
odd self-adjoint unitaries  
$[0,1]\ni t\mapsto V_t,$ with $[V_t]-[W]\in \Ker q_*$ and 
$V^\mathfrak{r}_t = V_t$ for all $t$. 
The continuity of $V_t$ ensures that the pointwise $C^*$-module 
$\ol{(V_t - W)A^n_A}$ can be extended to a $A \otimes C([0,1])$-module, 
$\overline{(V_\bullet - W)A^n}_{A\ox C([0,1])}$, where 
the real structure 
on $A\otimes C([0,1]) \cong C([0,1],A)$ is such that $a^\rs(t) = (a(t))^{\rs_A}$.  
Recalling Equation \eqref{eq:Cayley_bdd_transform}, the 
bounded transform of $\Ci_{W}(V_t)$ is given by 
$$
   F_t = \frac{1}{2} W(V_t + W)( V_t - W)^{-1} |V_t - W|
$$
for all $t$. Once again the continuity of $V_t$ ensures that 
$\{F_t\}_{t\in[0,1]}$ is a well-defined and self-adjoint operator $F_\bullet$ on 
$\overline{(V_\bullet - W)A^n}_{A\ox C([0,1])}$.
Assembling this information and using the pointwise properties 
of $F_t$, we obtain a Kasparov module
$$
\big(  W, \, \overline{(V_\bullet - W)A^n}_{A\ox C([0,1])}, \, F_\bullet \big).
$$
We therefore obtain a homotopy in $KKR$ and, hence, $\Chris$ is well-defined. 
It is easily seen that $\Chris$ respects direct sums and so is a homomorphism.
\end{proof}

\begin{rmk}
Let us briefly note that if $WV+VW=0$, so  $W$ and $V$ are homotopic as OSUs, 
then the resulting Kasparov module
$(W, \, \ol{(V-W)A^n_A}, \, \Ci_{W}(V))$ is degenerate. We first observe that 
if $W$ and $V$ anti-commute, $(W-V)^{-1} = \tfrac{1}{2}(W-V)$ and so 
$$
  \Ci_W(V) = \frac{1}{2}W(W+V)  (V-W) = \frac{1}{2}(V - WVW) = V.
$$
Hence the Kasparov module simplifies to $(W, \, \ol{(V-W)A^n_A}, \, V)$ which 
is clearly degenerate.
\hfill $\diamond$
\end{rmk}

To define a map from $KK$ to $DK$, we need to know
that our construction is compatible with Morita invariance. 
To work with explicit cycles, we also need to consider
$C^*$-modules that are not full. The first issue arises
because for a Kasparov module $(e,X_A,T)$,
we most easily construct OSUs in $\End_A(X)$ 
and need to get back to the coefficient algebra.
This is the Morita invariance requirement.

If $X$ is not full, then the Cayley transformation will only have range in $J_X=\ol{{\rm span}(X|X)_A}$. 
Hence we also need to understand the dependence of the Cayley transformation
on the inclusion $J_X\hookrightarrow A$. The next two lemmas address these points.

\begin{lemma}
\label{lem:Cay-iota}
Let $A$ be a unital and balanced graded algebra with $V, \, W \in M_n(A)$ OSUs, and let $J=\overline{A^n(V-W)A^n}$. 
Then we have classes $[V]^J-[W]^J\in DK(J)$ and
$[V]-[W]\in DK(A)$ related by $\iota_*([V]^J-[W]^J)=[V]-[W]$
where $\iota:\,J\to A$ is the inclusion. 

Letting
$\Chris^A:\,DK(A, \mathfrak{r}_A  )\to KKR(\Cl_{1,0},A)$ and
$\Chris^J:\,DK(A, A/J) )\to KKR(\Cl_{1,0},J)$
be the homomorphisms of Lemma \ref{lem:Roe_to_KK}, we have
$$
 \Chris^A \circ \iota_*([V]^J - [W]^J) 
 =\Chris^A([V] - [W])= \iota_* \circ \Chris^J([V]^J - [W]^J). 
$$
\end{lemma}
\begin{proof}
The first statements are proved in Lemma \ref{lem:graded-circle}.
The subsequent equalities are true by the
construction of the Kasparov modules. 
\end{proof}

\begin{lemma} \label{lem:Cay-zeta}
Let $X_A$ be a balanced graded Real $C^*$-module with 
${J_X}=\ol{{\rm span}(X|X)_A}$. The  maps 
$\Chris^{{J_X} }: DK({J_X}) \to KKR(\Cl_{1,0}, {J_X})$ and  
$\Chris^{\End_A^0(X)} : DK(\End_A^0(X)) \to KKR(\Cl_{1,0}, \End_A^0(X))$ 
are such that 
$$
  \Chris^{{J_X}} \circ\zeta_X=\zeta_{X}^{KK}\circ\Chris^{\End^0_A(X)}
$$
where $\zeta_X: DK(\End_A^0(X))\xrightarrow{\simeq} DK({J_X})$ is the 
isomorphism of Equation \eqref{eq:another-fucking-iso} and 
$\zeta_{X}^{KK} = \cdot \hox_{\End^0_A(X)} \big[( \End_A^0(X), X_{J_X}, 0) \big]$ the 
Morita isomorphism in $KK$.
\end{lemma}
\begin{proof}
Let $[U]-[V]\in DK(\End_A^0(X))$, so that $U,\,V\in M_n(\End_A^0(X)^\sim \hox \Cl_{1,1})$ 
with $U-V\in M_n(\End^0_A(X)\hox \Cl_{1,1})$.
Because we deal with matrices over $\End_A^0(X)^\sim \hox \Cl_{1,1}$, the Morita isomorphism 
of Lemma \ref{lem:DK-morita} gives that
\begin{align*}
\Chris^{J_X}\circ\zeta_X([U]-[V])
&=\Chris^{J_X}([W_{2n} (\tilde{U}\oplus Z^{\oplus n} \oplus  V\oplus Z^{\oplus n}) W^*_{2n}] 
- [W_{2n}\hat{Z}^{\oplus n}W_{2n}^*]) \\
&\hspace{-0.5cm}=\left[\left(\hat{Z}^{\oplus n}, \, \ol{\begin{pmatrix} \tilde{U}&0&-1&0\\ 0 & 0 & 0 & 0\\-1&0 & V&0\\ 0 & 0 & 0 & 0\end{pmatrix}
 \big(X \oplus\hat{\H}\big)^{\oplus 2n}_{{J_X}}}, \, \Ci_{\hat{Z}^{\oplus n}}(\tilde{U}\oplus Z^{\oplus n}\oplus V\oplus Z^{\oplus n} )\right)\right],
\end{align*}
where we have removed $W_{2n}$ using unitary invariance of $KK$-classes.
Similarly, we have that 
\begin{align*}
\zeta_{X}^{KK} \circ \Chris^{\End^0_A(X)} ([U]-[V])
&= \zeta_{X}^{KK} ([(V, \, \overline{(U-V)\End^0_A(X)^n}_{\End^0_A(X)}, \, \Ci_V(U))])\\
&=\big[(V, \, \overline{(U-V)X_{J_X}^n}, \, \Ci_V(U))].
\end{align*}
As in the proof of Lemma \ref{lem:Roe_to_KK}, 
the continuous homotopy from $U$ to $\tilde{U}$
gives a homotopy of Kasparov modules.

To complete the proof, we homotopy the Kasparov module
representing $\Chris^{J_X}\circ\zeta_X([U]-[V])$. First observe that
$$
  \begin{pmatrix} \sin(t)V & 0 & \cos(t)1_{n} & 0 \\ 0 & Z^{\oplus n} & 0 & 0 \\ \cos(t)1_{n}  & 0 & \sin(t)V & 0 \\ 0 & 0 & 0 & Z^{\oplus n} \end{pmatrix}
$$
is a homotopy of OSUs. This yields a homotopy of $C^*$-modules
$$
\ol{\begin{pmatrix} \tilde{U}-\sin(t)V&0&-\cos(t)1_n &0\\0&0&0&0\\-\cos(t)1_n &0 & V-\sin(t)V&0\\0&0&0&0\end{pmatrix}
\big( X \oplus\hat{\H}\big)^{\oplus 2n}_{{J_X}}}
$$
to $\ol{(\tilde{U}-V)X^n_{{J_X}}}\oplus 0$. Simultaneously, we obtain  
a homotopy of operators from 
$\Ci_{\hat{Z}^{\oplus}}(\tilde{U}\oplus Z^{\oplus n}\oplus V\oplus Z^{\oplus n})$ 
to $\Ci_{V}(\tilde{U})\oplus 0$ compatible with the obvious path
of domains and the (constant) left action of $\Cl_{1,0}$. Thus
\begin{align*}
\Chris^{{J_X}} \circ\zeta_X([U]-[V])&=[(V, \, \ol{(\tilde{U}-V)X_{J_X}^n}, \, \Ci_V(\tilde{U}))] \\
 &=  \zeta^{KK}_X\circ \Chris^{\End^0_A(X)}([U]-[V]).\qedhere
\end{align*}
\end{proof}

We now consider the map $KKR(\Cl_{1,0},A) \to DK(A)$.

\begin{lemma} \label{lem:KK_to_Roe}
Let $(e, X_A, T)$ be an unbounded Real Kasparov module 
with  $e^2 = 1_X$ and 
$e$ anti-commuting with $T$. 
Let 
${J_X}=\ol{{\rm span}(X|X)_A}$, 
$\zeta_X: DK(\End_A(X),\End_A(X)/\End^0_{A}(X)) \xrightarrow{\simeq}  DK({J_X})$ the 
isomorphism of Corollary \ref{cor:rel-morita} and 
$\iota:{J_X}  \hookrightarrow A $ the inclusion.
Then 
the Cayley transform defines a homomorphism 
$\Adam: KKR(\Cl_{1,0},A) \to DK(A, \rs_A )$,
$$
 KKR(\Cl_{1,0}, A) \ni [(e,X_A, T)] \xmapsto{\Adam} 
   \iota_{\ast}  \circ\zeta_X   \big( [\calC_e(T)] - [e]  \big)  \in DK(A, \rs_A).
$$
\end{lemma}
\begin{proof}
Lemma \ref{lem-Cd1} tells us that 
$\Cd_e(T)=e(T+e)(T-e)^{-1}$ is odd, self-adjoint, unitary,  
$\Cd_e(T)-e \in \End^0_A(X)$ and $\Cd_e(T)^\rs = \Cd_e(T)^\rs$.  
Hence we obtain a class 
$[\calC_e(T)] - [e] \in DK( \End_A(X), \, \End_A(X)/ \End_A^0(X))$.
Thus $\iota_{\ast}\circ \zeta_X  \big( [\Cd_e(T)] - [e]\big)$ 
is a well-defined element in $DK(A,\rs_A)$ and we just need 
to check that the map respects the relevant equivalence relations. 
We use the equivalence relation on $KKR$ generated by 
unitary equivalence, addition of degenerate
Kasparov modules and operator homotopy~\cite[Section 17]{Blackadar}.

Any (bounded) 
degenerate Kasparov module $(e,X_A,F)$ 
has $F$ invertible and anticommuting with $e$.
So suppose that the operator $T$ of our  unbounded Kasparov module
$
\left(e, \,X_A, \, T\right)
$
is invertible, self-adjoint and graded commutes with the $\Cl_{1,0}$-action. The
phase of $T$ then defines a degenerate bounded 
Kasparov module, whose class in $KKR$
is zero. Consider  the homotopy
$V(\lambda) = e(T + e\lambda)(T-e\lambda)^{-1}$
for $\lambda\in[0,1]$. Using the normality of $T$, we compute that
for $\lambda,\,\varrho\in[0,1]$,
\begin{align*}
  e(T + \lambda e)(T-\lambda e)^{-1} -e (T + \varrho e)(T-\varrho e)^{-1} 
 &= 2e(\lambda - \varrho)T(T -\varrho e)^{-1}(T-\lambda e)^{-1}.
\end{align*}
Hence the map $\lambda \mapsto V(\lambda)$ is norm continuous as 
$T(T-\varrho e)^{-1}(T-\lambda e)^{-1}$ 
is uniformly bounded since $T$ is invertible. The path is also 
invariant under the real structure as $T^\rs =T$ and $e^\rs = e$. 
We obtain a homotopy of OSUs  such that
$$
  V(\lambda) - e \in \End_A^0(X), \qquad V(1) = e(T+e)(T-e)^{-1}\sim V(0) = e 
$$
and $[\Cd_e(T)] - [e] = 0$. 
Thus  degenerate Kasparov classes 
map to zero.

Given an operator homotopy $(e,X_A, F_t)$ 
of bounded Kasparov modules with
 $F_0=T(1+T^2)^{-1/2}$, we can define the class
$$
\left(\Cl_{1,0}, (X\ox C([0,1]))_{A\ox C([0,1])}, F_\bullet\right)
$$
as a bounded Kasparov module. As shown in \cite[Proposition 2.8, Theorem 2.9]{DM}, there is some 
self-adjoint regular $T_\bullet$ such that $F_{T_\bullet}$ is
 $F_\bullet$, and we can moreover 
take $T_0$ to be operator homotopic to $T$. Averaging allows us to ensure that $T_te+eT_t=0$ 
and $T_t^\rs = T_t$ for all $t\in[0,1]$.  Then we have a homotopy 
of unbounded operators $T_t$ such
that $T_t(1+T_t^2)^{-1/2}$ is operator norm continuous 
for all $t$. Then using $(T_t-e)^{-2}=(1+T^2)^{-1}$
we compute
\begin{align} \label{eq:opr_htpy_OSU_htpy}
  e(T_t+e)(T_t-e)^{-1}&=e(T_t+e)(1+T_t^2)^{-1/2}(1+T_t^2)^{1/2}(T_t-e)^{-1} \nonumber \\
  &=e(T_t+e)(1+T_t^2)^{-1/2}(1+T_t^2)^{1/2}(T_t-e)(T_t-e)^{-2} \nonumber \\
  &=e(T_t+e)(1+T_t^2)^{-1/2}(T_t-e)
(1+T_t^2)^{-1/2}
\end{align}
which is a product of norm continuous paths by assumption. 
So we obtain a homotopy of odd Real self-adjoint unitaries. 
Then the class
$[\Cd_{e}(T_t)] - [e]$ is 
constant in the relative group 
$DK(\End_A(X),\End_A(X)/\End_A^0(X))$ 
for all $t \in [0,1]$ and so $\Adam$ is constant under operator homotopies.

The invariance of the map under unitary equivalence is a simple check. 
Finally, because group addition is induced by the direct sum, 
it follows that  $\Adam$ is a homomorphism.
\end{proof}

We combine Lemmas \ref{lem:Roe_to_KK} and  \ref{lem:KK_to_Roe} to obtain our main result.

\begin{thm} 
\label{thm:KK_to_DK_iso}
The homomorphisms $\Adam:KKR(\Cl_{1,0},A) \to DK(A,\mathfrak r_A)$
and $\Chris: DK(A,\mathfrak r_A)\to KKR(\Cl_{1,0},A)$ 
are mutually inverse isomorphisms.
\end{thm}
\begin{proof}
We do not assume that the unitisation $A^\sim$ is balanced graded,
so the homomorphism  $\Chris: DK(A,\mathfrak r_A)\to KKR(\Cl_{1,0},A)$
is defined as in \eqref{eq:needed-label}. 

We have already shown that $\Adam:KKR(\Cl_{1,0},A) \to DK(A)$ and 
$\Chris^{A} : DK(A) \to KKR(\Cl_{1,0}, A )$ are well-defined. 
We just 
need to show they are mutual inverses.

We first consider $\Chris \circ \Adam$. Take an element $[(e, X_A, T)] \in KKR(\Cl_{1,0}, A)$ with $e^2 = 1_X$ 
and $e$ anti-commuting with $T$.  
We set ${J_X}=\ol{{\rm span}(X|X)_A}$ and compute using Lemmas \ref{lem:Cay-iota} and \ref{lem:Cay-zeta},
\begin{align*}
 [(e, X_A, T)] &\xmapsto{\Adam}  \iota_*\circ\zeta_X    \big([\Cd_{e}(T)] - [e] \big) \\
  &\xmapsto{{\Chris}}\Chris^{A}  \circ\iota_*\circ\zeta_X    \big([\Cd_{e}(T)] - [e] \big)\\
  &= \iota_*\circ\Chris^{{J_X}} \circ\zeta_X  \big([\Cd_{e}(T)] - [e] \big)\\
  &= \iota_*\circ\zeta_{X}^{KK} \circ\Chris^{\End^0_A(X)}   \big([\Cd_{e}(T)] - [e] \big)\\
  & =\iota_*\circ\zeta_{X}^{KK}  \big[ (e, \, \ol{(T-e)^{-1} \End_{A}^0(X)}_{\End_{A}^0(X)}, \, \Ci_{e}\circ\Cd_{e}(T) ) \big] \\
  &= \iota_* \circ \zeta_X^{KK} \big[( e, \, \End_A^0(X)_{\End_{A}^0(X)}, \, T) \big]  \\
  &= \iota_* \big[ ( e , \, X_{{J_X}}, \, T) \big]  \\
  &= \big[ (e, \, X_{A}, \, T ) \big] ,
  \end{align*}
where we have used 
that 
$\Cd_{e}(T) - e = 2(T - e)^{-1}$, 
$(T - e)^{-1}$ has dense range, and
$\Ci_{e}\circ\Cd_{e}(T) = T$ 
by Proposition \ref{prop:Cayley_inverses}.

We now consider $\Adam \circ \Chris$. We do not assume $A^\sim$ is balanced graded and so 
consider OSUs
$U,\, V \in M_n(A^\sim\hox\Cl_{1,1})$ with 
$U-V\in M_n(A\hox\Cl_{1,1})$. Our Cayley map then gives
$$
   \Chris\big([U]-[V] \big) = \big[ ( V, \, \ol{(U-V)(A\hox\C^2)^n_A}, \, \calC_V^{-1}(U) ) \big].
$$
We let $Y_A = \ol{(U-V)(A\hox\C^2)^n_A}$ and recall 
from Proposition \ref{prop:Cayley_inverses} that $\Ci_V\circ \calC_V(U) = U |_Y$.
We let 
$J_Y= \ol{\mathrm{span}(Y|Y)_A}$
and use 
Equation \eqref{eq:another-fucking-iso} and Lemma  \ref{lem:Cay-iota}  to compute 
\begin{align*}
\Adam\circ \Chris\big([U]-[V] \big)
&=  \iota_*\circ\zeta_{Y}  \big([U|_{Y_A}]^{\End^0_{A}(Y)}-[V|_Y]^{\End^0_{A}(Y)}\big)\\
&=  \iota_*\big([W_2 (\widetilde{U}|_{Y_A} \oplus Z \oplus V|_{Y_A}\oplus Z) W^*_2]^{J_Y}-[W_2\hat{Z}W_2^*]^{J_Y}\big)\\
&=  \iota_*\big([W_2(U|_{Y_A}\oplus Z\oplus V|_{Y_A}\oplus Z)W^*_2]^{J_Y}-[W_2\hat{Z}W_2^*]^{J_Y}\big)\\
&=  [W_2(U\oplus Z\oplus V\oplus Z)W^*_2]-[W_2\hat{Z}W_2^*]  \\
&= \left[\begin{pmatrix} U & 0 \\ 0 & V \end{pmatrix} \right] - \left[\begin{pmatrix} 0 & 1_n \\ 1_n & 0 \end{pmatrix} \right]
    \in DK(A\hox \Cl_{1,1}).
\end{align*}
This completes the proof. If $A$ is balanced graded, then we can apply 
(the inverse of) the isomorphism \eqref{eq:ubiquitous-isomorphism} 
to recover $[U]-[V] \in DK(A)$ explicitly.
\end{proof}

For completeness, let us list a few immediate corollaries of our result.

\begin{cor}
\label{cor:general}
\begin{enumerate}
  \item Let $A$ be a graded $C^*$-algebra. Then 
  $KK(\Cl_1, A) \cong DK(A)$. 
  \item Let $B$ be a real $C^*$-algebra, $B = A^{\mathfrak{r}_A}$ for some Real $C^*$-algebra $A$. 
  Then $KKO(Cl_{1,0}, B) \cong DK(B)$. 
  \item Recall the complex graded $K$-theory groups $K_j^{\mathrm{gr}}(A) := KK(\C, A\hat\otimes \Cl_j)$ 
  from~\cite{KPS}. Then $DK(A) \cong K_1^{\mathrm{gr}}(A)$.
  \item  Let $KR'(A) = [C_0(\R), A \hat\otimes \calK]$ denote the group of $\Z_2$-graded and 
  Real asymptotic morphisms from~\cite{TroutGraded}, where $C_0(\R)$ has grading $\alpha$ such that 
  $\alpha(f)(x) = f(-x)$ and real structure by complex conjugation. Then 
  $KR'(A\hox \Cl_{0,1}) \cong DK(A)$.
\end{enumerate}
\end{cor}
\begin{proof}
The first two results come from either ignoring the real structure or passing to a real 
subalgebra. For the third statement,  we use that 
$$
  DK(A) \cong KK(\Cl_1, A) \cong KK(\Cl_2, A\hat\otimes \Cl_1) \cong KK(\C, A\hat\otimes \Cl_1) =  K_1^{\text{gr}}(A),
$$
where we take the external product by $1_{KK(\Cl_1,\Cl_1)}$ and use the 
Morita equivalence between $\Cl_2$ and $\C$. 
Similarly, it is shown in~\cite[Theorem 4.7]{TroutGraded} that 
$KR'(A) \cong KKR(\C, A)$. Hence 
\[
  KR'(A \hox \Cl_{0,1}) \cong KKR(\C, A\hox \Cl_{0,1}) \cong KKR(\Cl_{1,0}, A) \cong DK(A) .  \qedhere
\]
\end{proof}

\begin{rmk}
\label{rmk:graphproj-vs-general}
Let us briefly consider the map $\Adam$ applied to a 
complex Kasparov module $(\C,X_A,T)$ with $A$ trivially graded. 
We first inflate this Kasparov module to a class in $KK(\Cl_{1}, A\otimes \Cl_{1})$ by 
taking the external product with the ring identity of $KK(\Cl_{1}, \Cl_{1})$. 
Given the Kasparov module 
$(e,(X\hox\Cl_1)_{A\hox\Cl_1},T\hox1)$, we choose the `ordered basis' 
$$
(X\hox\Cl_1)_{A\hox \Cl_1}=
\big[ (X_+\ox \Cl_{1,+})\oplus (X_-\ox \Cl_{1,-}) \oplus  (X_+\ox \Cl_{1,-}) \oplus (X_-\ox \Cl_{1,+}) \big]_{A \hox \Cl_1}
$$
and then compute 
\begin{align*}
  \Cd_e(T) - e = 2(T-e)^{-1}  = \begin{pmatrix} 0_2 & (\tilde{T} - \sigma_3)^{-1} \\ (\tilde{T}-\sigma_3)^{-1} & 0_2 \end{pmatrix}, 
  \qquad \tilde{T} = \begin{pmatrix} 0 & T_- \otimes 1 \\ T_+ \otimes 1 & 0 \end{pmatrix} 
\end{align*}
with $T_\pm : X_\pm \to X_\mp$. Further expanding and supressing the tensor product notation
$$
  (\tilde{T} - \sigma_3)^{-1} = 
  \begin{pmatrix} - (1+ T_{-}T_+)^{-1} & T_{-}( 1+ T_{+}T_{-})^{-1} \\ T_{+}(1+T_{-}T_{+})^{-1} & (1+T_{+}T_{-})^{-1} \end{pmatrix} 
  = P_{T_-} - P_{X_+}.
$$
Hence, as an operator on 
$(X \hat\otimes \Cl_{1})_{A\hox \Cl_1} \cong \binom{X_+}{X_-}^{\oplus 2}_{A\hox \Cl_1}$, 
$\Cd_e(T) - e$ acts as $(P_{T_-} - P_{X_+}) \otimes \sigma_1$.
Therefore our general Cayley map $\Adam$ is precisely the negative of the graph projection 
map we employed in Section \ref{subsec:graph_proj}.
\hfill $\diamond$
\end{rmk}

\section{Applications to real and complex $K$-theory}
\label{sec:app-real}
In this section
we consider some special cases of Theorem \ref{thm:KK_to_DK_iso} 
to study examples and problems coming from real and complex $K$-theory.

\subsection{Unitary descriptions of $K$-theory}
\label{subsec:unitary_KO}

Given a complex and ungraded $C^*$-algebra $A$ with real structure $\rs_A$, we 
know from~\cite[{\S}5]{Kasp} that there are isomorphisms 
$KKR(\Cl_{r,s}, A) \cong KO_{r-s}(A^{\rs_A})$, where this identification is 
shown via a (generalised) Clifford-module index, see~\cite[Section 2.2]{SchroderKTheory}.

Alternatively, descriptions of $KO$-theory using Real $C^*$-algebras and unitaries have 
appeared in~\cite{BL15} and~\cite[Section 5.6]{Kellendonk15}. In this section, we show 
in a few cases how these unitary descriptions of $KO$-theory are compatible with our 
Cayley isomorphism.
We note that many of these descriptions will be of use to us for studying the 
bulk invariants of topological insulators in Section \ref{sec:TI_application}.

\vspace{0.1cm}

\begin{example}[Trivially graded algebras and $KO_1$]  \label{ex:triv_graded_Cayley}
Let $A$ be trivially graded and 
 $(e,X_A,T)$ an unbounded Real Kasparov module representing an element  in $KKR(\Cl_{1,0},A)$.
As $A$ is trivially graded, without loss of generality 
we can write $X_A \cong Y_A \oplus Y_A$ with $Y_A$ an ungraded $C^*$-module. 
Because $T$ anti-commutes with the generator 
$\Cl_{1,0}$ generator, 
our Kasparov module reduces to the form
$$
  \big( e, \, Y_A \otimes  \C^2, \,T= T_+  \otimes f \big), \qquad e = \begin{pmatrix} 0 & 1 \\ 1 & 0 \end{pmatrix}, \;\; 
  f = \begin{pmatrix} 0 & -1 \\ 1 & 0 \end{pmatrix},
$$
with $T_+^* = -T_+$.  The real structure on $\C^2$ is pointwise complex conjugation, 
which ensures that 
$e^\mathfrak{r} = e$, $f^\mathfrak{r} = f$ and $T_+^\rs = T_+$.

Because $T_+$ is skew-adjoint $(T_+ \pm 1)$ is invertible and we compute 
\begin{align*}
  \calC_e(T) = e(T+e)(T-e)^{-1} &= 
   \begin{pmatrix} 0 & 1 \\ 1 & 0 \end{pmatrix}  \begin{pmatrix} 0 & -T_+ +1 \\ T_+ +1 & 0 \end{pmatrix} 
    \begin{pmatrix} 0 & -T_+ -1 \\ T_+ - 1 & 0 \end{pmatrix}^{-1} \\
     &= \begin{pmatrix} 0 & (T_+ + 1)(T_+ - 1)^{-1} \\ (T_+ - 1)(T_+ + 1)^{-1} & 0 \end{pmatrix}.
\end{align*}
One finds that $U_{T_+} =(T_+ + 1)(T_+ - 1)^{-1}$ is unitary, $(U_{T_+})^\mathfrak{r} = U_{T_+}$ 
and 
$$
  1- (T_+ + 1)(T_+ - 1)^{-1}  = -2(T_+ -1)^{-1} \in \End_A^0(Y)
$$
as we have an unbounded Kasparov module. Thus for ${J_Y} = \overline{\mathrm{span}(Y | Y)_A}$ we have a class 
$[U_{T_+}] \in K_1(\End_A^0(Y)) \cong K_1({J_Y})$ where we use an ungraded version 
of the isomorphism of Equation \eqref{eq:another-fucking-iso} from Section \ref{subsec:pain}. 
If we ignore real structures, then denoting $\iota: {J_Y} \hookrightarrow A$ and 
$\zeta_Y: K_1( \End_A^0(Y)) \xrightarrow{\simeq} K_1({J_Y})$, we obtain a map
$$
 KK(\Cl_1,A) \ni  \big[ ( \Cl_{1}, \, Y_A \otimes \C^2 , \, T_+ \otimes f ) \big] 
  \mapsto \iota_\ast \circ \zeta_Y \big[(T_+ + 1)(T_+ - 1)^{-1} \big] \in  K_1(A)
$$
which is a skew-adjoint analogue of (the inverse of) the isomorphism in Theorem \ref{thm:complex_cayley}.
Similarly, passing to real subalgebras 
$ \iota_\ast \circ \zeta_{Y^{\rs_Y}} \big[(T_+ + 1)(T_+ - 1)^{-1} \big] \in  KO_1(A^{\rs_A})$.

Let us also consider the inverse map. If $U \in A^\sim$ is 
a unitary and $U^{\mathfrak{r}_A} = U$, then $(U+1)(U-1)^{-1}$ is an unbounded skew-adjoint 
operator and 
\begin{equation} \label{eq:ungraded_skew_cayley_inverse}
  \Big(  \sigma_1, \, \ol{(U-1) A_A} \otimes \C^2, \, (U+1)(U-1)^{-1} \otimes f \Big)
\end{equation}
is an unbounded Kasparov module, where the real structure on $\ol{(U-1) A_A}$ comes 
from $\mathfrak{r}_A$ and the real structure on $\C^2$ is pointwise complex conjugation. 
A direct check or Theorem \ref{thm:KK_to_DK_iso} (combined with the equivalence between van Daele and 
operator $K$-theory for ungraded algebras) gives that the map 
$KKO(Cl_{1,0},A^{\mathfrak{r}_A}) \to KO_1(A^{\mathfrak{r}_A})$ or $KK(\Cl_1, A) \to K_1(A)$ 
is an isomorphism with the inverse given by the unbounded 
Kasparov module in Equation \eqref{eq:ungraded_skew_cayley_inverse}.
\hfill $\diamond$
\end{example}

\vspace{0.1cm}

\begin{example}[An isomorphism $KKR(\Cl_{0,1},A) \to KO_{-1}(A^{\mathfrak{r}_A})$] \label{ex:imag_cayley}
Here we consider the Cayley map for elements in $KKR(\Cl_{0,1},A)$ that reduces to 
our original ungraded complex Cayley isomorphism from Theorem \ref{thm:complex_cayley} if we ignore 
the real structure.

Let $(\Cl_{0,1}, X_A, T)$ be a Real 
Kasparov module with $A$ trivially graded and  $f\in\Cl_{0,1}$ the generator.  Making 
analogous simplifications as Example \ref{ex:triv_graded_Cayley}, 
we write the Kasparov module as
$$
  \Big( \Cl_{0,1}, \, (Y \oplus Y)_A, \, T=\begin{pmatrix} 0 & S \\ S & 0 \end{pmatrix} \Big), \quad 
  f\mapsto \begin{pmatrix} 0 & -1 \\ 1 & 0 \end{pmatrix}, \,\, 
  (y_1,y_2)^{\mathfrak{r}_Y} = (y_1^{\mathfrak{r}_Y}, y_2^{\mathfrak{r}_Y}),
$$
which also implies that $S=S^*$ and $S^\mathfrak{r} = S$ on the ungraded $C^*$-module $Y_A$. 
As $S$ is self-adjoint, unbounded and has 
compact resolvent, we can apply the ungraded Cayley transform
$$
   U_S = \Cd(S) = (S+i)(S-i)^{-1},
$$
which by Proposition \ref{prop:Cayley_properties} 
is unitary and $U_S - 1 \in \End_A^0(Y)$. Applying the real structure on $\End_A(Y)$,
$U_S^\mathfrak{r} = U_S^*$. Hence, we obtain a map from cycles in $KKR(\Cl_{0,1},A)$ to 
unitaries  $u \in \End_A^0(Y)^\sim$ such that $u^\mathfrak{r} = u^*$. 

The group $KO_{-1}(A^{\mathfrak{r}_A})$ can be characterised by  equivalence 
classes of complex unitaries in $M_n(A^\sim)$ such that $u^\mathfrak{r} = u^*$~\cite[Section 5.6]{Kellendonk15}.
We also compare our presentation of $KO_{-1}$ to that of Boersema and Loring, who characterise $KO_{-1}(A,\rho)$ 
as equivalence classes of unitaries $u \in M_n(A^\sim)$ such that $u^\rho = u$ for $\rho$ an 
\emph{anti-multiplicative} involution~\cite{BL15}. We can recover this 
picture by defining $\rho = \ast \circ \mathfrak{r}$, so that $u^\mathfrak{r}=u^*$ implies 
that $u^\rho = u$. 
Hence, our Cayley map determines a class $\iota_\ast \circ \zeta_Y [U_S] \in KO_{-1}(A^\mathfrak{r})$.

Now, suppose that $A$ is a Real $C^*$-algebra and 
the real structure in $M_n(A)$ is applied entrywise. Given  $u\in M_n(A^\sim)$ unitary 
and such that $u^{\mathfrak{r}_A} = u^*$, by  Proposition \ref{prop:Cayley_properties}   
there is a well-defined self-adjoint operator $\Ci(u)$,
$$
  \Dom(\Ci(u)) =(u-1)A^n, \qquad \Ci(u)v = i(u+1)(u-1)^{-1}v,\quad v\in\Dom(\Ci(u)).
$$
Using the obvious real structure on the $C^*$-module $\ol{(u-1)A^n_A}$, we check that 
$$
  \Ci(u)^\mathfrak{r} = -i(u^*+1)(u^*-1)^{-1} = -i(u^*+1)u((u^*-1)u)^{-1} = i(u+1)(u-1)^{-1} = \Ci(u)
$$
and so the argument in Proposition \ref{prop:K_1_to_KK} extends to give that 
$$
   \Big( \Cl_{0,1}, \, \ol{(u-1)A^n_A} \otimes \C^2 , \, \begin{pmatrix} 0 & \Ci(u) \\ \Ci(u) & 0 \end{pmatrix} \Big), 
   \quad f\mapsto \begin{pmatrix} 0 & -1 \\ 1 & 0 \end{pmatrix}, 
$$
is an unbounded $KKR$-cycle with real structure on $\C^2$ by complex conjugation.
Following the proof of Theorem \ref{thm:complex_cayley}, we obtain that the maps
\begin{align*}
   &KO_{-1}(A^{\mathfrak{r}_A}) \ni [u] \mapsto  \left[ \big( \Cl_{0,1}, \, \ol{(u-1)A^n_A} \otimes \C^2, \, \Ci(u) \otimes \sigma_1 \big) \right] 
      \in KKR(\Cl_{0,1}, A) \\
   &KKR(\Cl_{0,1}, A) \ni \left[ \big( \Cl_{0,1}, \, Y_A  \otimes \C^2, \, S \otimes \sigma_1 \big) \right] \mapsto 
     \iota_\ast \circ \zeta_Y \left[ \Cd(S) \right] \in KO_{-1}(A^{\mathfrak{r}_A}) 
\end{align*}
are well-defined and mutual inverses. The main difference is that the identity element in 
$KO_{-1}(A^\mathfrak{r})$ is given by the class of $i$ times the unit of $A$,
$[i 1_A]$ and we need to ensure that any homotopy 
of unitaries respects the condition $v_t^\mathfrak{r} = v_t^*$.

Clearly if we ignore the real structure, then we recover our original ungraded Cayley map 
$K_1(A) \to KK^1(\C,A)$ from Theorem \ref{thm:complex_cayley}.
\hfill $\diamond$
\end{example}

\vspace{0.1cm}

\begin{example}[Unitary and projective descriptions of $K_0$] \label{ex:K_0_cayley}
We consider $A\otimes\Cl_{1,0}$ with $A$ ungraded and Real. In this case, 
any odd self-adjoint unitary is of the form $x \otimes e$ with $e$ the generator 
of $\Cl_{1,0}$ and $x=x^*=x^{\rs_A}$ and unitary. Suppose that we have two self-adjoint 
Real unitaries $x \in M_n(A^\sim),\, y \in M_m(A^\sim)$ with $x-y \in M_N(A)$. Then we 
obtain an element
$[x\otimes e] - [y \otimes e] \in DK(A\otimes \Cl_{1,0})$. 

Applying our Cayley map, we first note that, because $x-y$ is compact (over $A$),
$$
   \ol{(x\otimes e - y\otimes e)(A\otimes \Cl_{1,0})^N}_{A \otimes \Cl_{1,0}} \cong 
   \frac{1}{2}\ol{(x\otimes 1 - y \otimes 1)(A\otimes \Cl_{1,0})^N}_{A\otimes \Cl_{1,0}}
$$
Furthermore, on its domain, the Cayley transform $\Ci_{y\otimes e}(x\otimes e)$ acts as the zero-map 
and so our map $DK(A\otimes \Cl_{1,0}) \to KKR(\Cl_{1,0}, A\otimes \Cl_{1,0})$ reduces to 
\begin{align*}
  [x\otimes e] - [y \otimes e] &\mapsto \big[ ( \Cl_1, \, \tfrac{1}{2}(x\otimes 1 -y\otimes 1)(A\otimes \Cl_1)^N_{A\otimes \Cl_1}, \, 0 ) \big] \\
    &= \big[ ( \C, \, \tfrac{1}{2}(1-x)A^n_A \oplus \tfrac{1}{2}(1-y)A^m_A, \, 0 ) \big] \hat\otimes_\C \big[( \Cl_1, \, {\Cl_1}_{\Cl_1}, \, 0 )\big] \\
    &=  \big[ ( \C, \, \tfrac{1}{2}(1-x)A^n_A \oplus \tfrac{1}{2}(1-y)A^m_A, \, 0 ) \big] \hat\otimes_\C \, 1_{KKR(\Cl_{1,0},\Cl_{1,0})}.
\end{align*}
Hence, given  projections $p,\, q \in M_N(A)$, with $p^{\rs_A} = p$ and $q^{\rs_A} = q$, 
we recover the usual map $KO_0(A^{\rs_A}) \to KKR(\C,A)$ via 
the self-adjoint unitaries $x=1-2p$, $y=1-2q$ and our van Daele map. 
If we ignore the real structure, then our Cayley map recovers the 
isomorphism $K_0(A) \to KK(\C, A)$ from Section \ref{subsec:graph_proj}.
\hfill $\diamond$
\end{example}

\vspace{0.1cm}

\begin{example}[$KKR(\Cl_{1,0}, M_2(A)\otimes \Cl_{0,1}) \to KO_2(A^{\mathfrak{r}_A})$]
 Suppose that we have 
the algebra $M_2(A)\otimes \Cl_{0,1}$ with $A$ unital, trivially graded and the real structure on $M_2(A)$ given 
entrywise by $\mathfrak{r}_A$. Any $C^*$-module $Y_{M_2(A)}$ can be decomposed into an 
\emph{ungraded} sum $(X \oplus X)_{M_2(A)}$. We use the presentation  
$\Cl_{0,1} \cong \C\oplus \C$ with grading by the 
flip automorphism and real structure $(\alpha,\beta)^{\rs_{0,1}} = (\ol{\beta},\ol{\alpha})$. 
Then if we take a class in $KKR(\Cl_{1,0},M_2(A)\otimes \Cl_{0,1})$, 
we can write
$$
  \Big( \Cl_{1,0}, \,  \big( (X\oplus X)\otimes \Cl_{0,1}\big)_{M_2(A)\otimes \Cl_{0,1}}, \, T \Big), 
   \qquad \Cl_{1,0}= C^*( e), \,\, e = \begin{pmatrix} 0 & -i\ \\ i & 0 \end{pmatrix} \otimes (1,-1) 
$$
where using the real structure 
$\big((x_1,x_2)\otimes (\alpha,\beta) \big)^\mathfrak{r} = (x_1^{\mathfrak{r}_X},x_2^{\mathfrak{r}_X})\otimes (\ol{\beta},\ol{\alpha})$ 
we see that 
$e^\mathfrak{r} = (-1)^2 e = e$. Similarly the right-action of $\Cl_{0,1}$ is given by  multiplication by 
$1_2\otimes (i,-i)$. The decomposition of the Kasparov module means that 
we can write $T$ in the form $T= S\otimes(1,-1)$, where $S$ is a self-adjoint unbounded 
operator on $X\oplus X$, $S\sigma_2 + \sigma_2 S=0$ and 
$S^\mathfrak{r} = S$. We then compute that 
$$
   \Cd_e(T) =  e(T+e)(T-e)^{-1} = \sigma_2(S+\sigma_2)(S-\sigma_2)^{-1} \otimes (1,-1).
$$
Letting $U_S = \sigma_2(S+\sigma_2)(S-\sigma_2)^{-1} \in \End_{M_2(A)}^0(X\oplus X)^\sim$, we see that 
$U_S^\mathfrak{r} = -U_S$, $U_S^*=U_S$ and 
\begin{align*}
  U_S^2 &= \sigma_2(S+\sigma_2)(S-\sigma_2)^{-1}\sigma_2(S+\sigma_2)(S-\sigma_2)^{-1}  \\
    &= \sigma_2 (S+\sigma_2)(S-\sigma_2)^{-1}(-S+\sigma_2)\sigma_2 (S-\sigma_2)^{-1} \\
    &= -\sigma_2 (S+\sigma_2)\sigma_2(S-\sigma_2)^{-1} 
    = -\sigma_2^2 (-S+\sigma_2)(S-\sigma_2)^{-1} = 1.
\end{align*}
Hence, $U_S$ is an (ungraded) self-adjoint and imaginary unitary. 

Summarising our discussion, given a class in $KKR(\Cl_{1,0},M_2(A)\otimes \Cl_{0,1})$, we can construct 
a unitary operator $V \in \End_{M_2(A)}^0(X\oplus X)^\sim$ such that $V^*=V$ and $V^\mathfrak{r} = -V$. 
Applying the (ungraded) Morita invariance from Equation \eqref{eq:another-fucking-iso}, 
we recover the unitary description of $KO_2(A^{\mathfrak{r}_A})$ as homotopy classes of self-adjoint unitaries 
with $u^\rs = -u$ given in~\cite{BL15, Kellendonk15}. Such self-adjoint and imaginary unitaries 
 can be abstractly characterised as 
spectrally flattened Hamiltonians with a particle-hole symmetry. We will return to this point in 
Section \ref{sec:TI_application}.
\hfill $\diamond$
\end{example}

\vspace{0.1cm}

\begin{example}[$KO_3$ and $KKR$]  \label{ex:KO3_to_KK}
Using the K\"{u}nneth formula for real $K$-theory~\cite{Boersema02}, we can express 
$KO_3(A^{\rs_A}) \cong KO_{-1}(A^{\rs_A} \otimes\mathbb{H})$, where $\mathbb{H}$ is considered 
as a real ungraded $C^*$-algebra. In particular, we use the presentation 
$\mathbb{H} \cong M_2(\C)^{\mathrm{Ad}_{-i \sigma_2} \circ \mathfrak{c}}$ with 
$\mathfrak{c}$ complex conjugation. We 
again note that this is an ungraded isomorphism (putting in a grading, the right hand side 
of the isomorphism becomes $Cl_{0,2}$).

To note this equivalence concretely, 
we use the description of $KO_3$ from \cite[Section 5.6]{Kellendonk15}, which 
characterises $KO_3(A^{\rs_A})$ as equivalence classes of unitaries $u\in M_n(A^\sim)$ such 
that $u^{\rs_A} = -u^*$. Given such a $u \in A$ we consider the matrix 
$v=u \otimes \sigma_1 \in A \otimes M_2(\C)$, where one can check that 
$v^{\mathrm{Ad}_{-i \sigma_2} \circ \rs_A} = v^*$ and 
as such we get a class $[u \otimes \sigma_1] \in KO_{-1}(M_2(A)^{\mathrm{Ad}_{-i \sigma_2}\circ \rs_A} )$. 
Hence we can apply the map from Example \ref{ex:imag_cayley}  to get a Real Kasparov module
$$
  \left( \Cl_{0,1} \, \ol{((u\otimes \sigma_1) - 1_2)M_2(A)}_{A\otimes M_2(\C)} \otimes \C^2, \, 
  \mathcal{C}^{-1}(u\otimes \sigma_1) \otimes \sigma_1 \right)
$$
with $(a_1, a_2)^\rs = (a_1^{\mathrm{Ad}_{-i \sigma_2} \circ \rs_A}, a_2^{\mathrm{Ad}_{-i \sigma_2} \circ \rs_A} )$ and 
the left $\Cl_{0,1}$-action generated by $1\otimes (-i\sigma_2)$.
Passing to real subalgebras and applying the K\"{u}nneth formula, we obtain an element in
$KKO(Cl_{0,1}, A^{\rs_A} \otimes M_2(\C)^{\mathrm{Ad}_{-i \sigma_2} \circ \mathfrak{c}}) \cong KKO(Cl_{0,1}, A^{\rs_A} \otimes Cl_{0,4})$. 
\hfill $\diamond$
\end{example}

\subsection{Short exact sequences and boundary maps} \label{subsec:DK_bdry}

Here we consider the compatibility of our Cayley isomorphism with the boundary map 
of van Daele $K$-theory and $KK$-theory. Suppose that
\begin{equation}  \label{eq:gradedSES}
  0 \to I \to E \xrightarrow{q} A \to 0 
\end{equation}
is a short exact sequence of graded $C^*$-algebras with a completely positive linear splitting.
If the algebras 
possess a real structure, 
then we also assume that these maps are 
equivariant with respect to this structure. 
By~\cite[Theorem 1.1]{Skandalis85} 
there are connecting homomorphisms 
$KKR(I,B) \xrightarrow{\delta} KKR(A, B\hat\otimes \Cl_{1,0})$ 
and $KKR(B,A) \xrightarrow{\delta} KKR(B, I\hat\otimes \Cl_{1,0})$. 
We will consider a special case of the latter of these boundary maps using 
van Daele $K$-theory and our 
Cayley isomorphism. 
We note that boundary maps in van Daele $K$-theory and their
compatibility with Kasparov theory has already been extensively studied by Kubota~\cite[Section 5]{Kubota15a}. 
In particular, the boundary maps in van Daele $K$-theory inherit many properties from 
$KK$-theory such as naturality.

If the quotient algebra $A$ is unital, we assume that it is balanced (it contains an OSU). 
If $A$ is non-unital, we assume that $\mathrm{Mult}(A)$ contains an OSU $e$
and use the description of van Daele that includes a 
base point $DK_e(A) \cong DK(A^{\sim e}, A^{\sim e} / A)  \cong DK(A)$ from Lemma \ref{lem:unitisation}, where 
$$
DK_e(A) 
=  \big\{[x]-[y]\in GV_{e}(A^{\sim e}):\,x-(e_k\oplus -e_{n-k}),\,y-(e_k\oplus -e_{n-k})\in M_n(A),\ \mbox{some }n,k\big\}
$$
and $A^{\sim e} \subset \mathrm{Mult}(A)$  the algebra generated by $A$ and $e$. 
 Let us recall the formula 
for the boundary map in van Daele $K$-theory.

\begin{lemma}[\cite{vanDaele2}, Proposition 3.4] \label{lem:vD_bdry_map} 
The boundary map of the 
short exact sequence \eqref{eq:gradedSES}, $\delta: DK(A) \to DK(I\hat\otimes \Cl_{1,0})$, is given by 
\begin{equation} \label{eq:vD_bdry_map}
\delta([x_1]-[x_2]) = [Y_1] - [Y_2], \qquad   
Y_i =   - \exp(\pi \tilde{x}_i  \hat\otimes \kg)(1 \hat\otimes\kg),
\end{equation}
where $\tilde{x}_i\in E$ is an odd self-adjoint lift of $x_i$ and $\kg$ is the odd 
generator of $\Cl_{1,0}$. We may assume that $\| \tilde{x}_i\|=1$.
\end{lemma}
If $e$ is a choice of base point in $A$ which lifts to an OSU in $E$ (we may simply take the image of a base point in $E$) then one easily finds
$\delta([x]-[e]) = [Y] - [1\hot \kg]$ and so we may simplify the formulas (as does van Daele) by writing $\delta([x]) = [Y]$.
We also remark that Equation \eqref{eq:vD_bdry_map} is simpler than the formula given in~\cite[Proposition 3.4]{vanDaele2}. 
This is because we have chosen $1 \hox \kg$ as a constant basepoint and noting that
$$
   -\big( \sin( \pi \tilde{x} \hox 1) + (1\hox \kg) \cos( \pi \tilde{x} \hox 1) \big) = - \exp( \pi \tilde{x} \hox \kg) (1 \hox \kg),
$$ 
which can be shown by the Taylor series expansion.

\begin{prop} \label{prop:vDboundary_to_KK}
Let $x \in M_n(A^{\sim e})$ be an OSU. 
Under the isomorphism of Theorem \ref{thm:KK_to_DK_iso} and 
the identification $KKR(\Cl_{1,0}, I \hat\otimes \Cl_{1,0}) \cong KKR(\C, I)$, the class 
$\delta([x]) \in DK(I \hat\otimes \Cl_{1,0})$ can be identified with the class of the unbounded Kasparov module
$$
 \Big(\C, \, \overline{\cos (\tfrac{\pi}{2} \tilde{x}) I^n_I}, \, \tan(\tfrac{\pi}{2} \tilde{x} ) \Big),
$$
with $\tilde{x} \in M_n(E)$ an odd self-adjoint lift of $x$.
\end{prop}
\begin{proof}
We use Lemma \ref{lem:vD_bdry_map} and calculate the inverse Cayley transform, where 
\begin{align*}
  \Ci_{1 \hat\otimes\kg}(Y) =& (1 \hat\otimes \kg) \big(-\exp(\pi \tilde{x} \,\hox\, \kg)(1 \hat\otimes \kg) + (1\hat\otimes \rho)\big)
   \big( -\exp(\pi \tilde{x} \,\hox\, \kg)(1\hat\otimes \rho) -(1\hat\otimes \rho) \big)^{-1}\\
 &= (1\hat\otimes \kg)\big( -\exp(\pi \tilde{x} \,\hox\, \kg) +1 \big) \big( -\exp(\pi \tilde{x}\,\hox\, \kg) - 1\big)^{-1} \\
 &=  -(1\hat\otimes \kg) \tanh( \tfrac{\pi}{2} \tilde{x} \,\hox\, \kg)
\end{align*}
with domain ${\sinh (\frac{\pi}{2} \tilde{x} \,\hox\,  \kg)(I^n \hat\otimes  \Cl_{1,0})}$. 
Using that $\tilde{x}$ and $\kg$ are odd, 
$(\tilde{x}\hat\otimes \kg)^{2j+1} = (-1)^j \tilde{x}^{2j+1} \hat\otimes \kg$. Therefore,
$$
  -(1\hat\otimes \kg) (-1)^{j} (\tilde{x} \hat\otimes \kg)^{2j+1}  = 
  -(1\hat\otimes \kg) (-1)^j (-1)^j (\tilde{x}^{2j+1} \hat\otimes \kg) = \tilde{x}^{2j+1}\hat\otimes 1
$$
and so by the Taylor series expansion, 
$-(1\hat\otimes \kg) \tanh( \frac{\pi}{2} \tilde{x} \,\hox\, \kg) = \tan(\frac{\pi}{2} \tilde{x})\, \hat\otimes \,1$ 
on the domain $\cos (\tfrac{\pi}{2} \tilde{x}) I^n_I\hat\otimes \Cl_1$. Hence our Kasparov module can be 
factorised
$$
  \big( \Cl_{1,0}, \ol{\cos (\tfrac{\pi}{2} \tilde{x}) I^n_I} \hat\otimes {\Cl_{1,0}}_{\Cl_{1,0}}, \, \tan(\tfrac{\pi}{2} \tilde{x})\, \hat\otimes \,1 \big) 
  =   \big(\C, \, \overline{\cos (\tfrac{\pi}{2} \tilde{x}) I^n_I}, \, \tan(\tfrac{\pi}{2} \tilde{x} ) \big) \hat\otimes_\C 
    ( \Cl_{1,0}, {\Cl_{1,0}}_{\Cl_{1,0}}, 0), 
$$
and removing the element $1_{KKR(\Cl_{1,0},\Cl_{1,0})}$ gives the 
identification $KKR(\Cl_{1,0}, I \hat\otimes \Cl_{1,0}) \cong KKR(\C, I)$.
\end{proof}

\begin{cor} \label{cor:KK_bdry_cayley}
Let $(\Cl_{1,0}, X_A, T)$ be an 
(unbounded) Kasparov module such that the operator $T$ anti-commutes with 
the left Clifford generator $e$. 
The image of this Kasparov module under the composition 
$KKR(\Cl_{1,0}, A) \xrightarrow{\delta_{KK}} KKR( \Cl_{1,0}, I \hat\otimes \Cl_{1,0}) \xrightarrow{\simeq} KKR(\C, I)$ 
can be represented by the Kasparov module
$$
  \Big(\C, \, \overline{\cos (\tfrac{\pi}{2} \widetilde{\Cd}_e(T) ) I_I }, \, 
    \tan(\tfrac{\pi}{2} \widetilde{\Cd}_e(T) ) \Big)
$$ 
with $\widetilde{\Cd}_e(T) \in E$ a lift of $\Cd_e(T)$.
\end{cor}
\begin{proof}
By~\cite[Proposition 5.13]{Kubota15a}, the diagram
\[
  \xymatrix{
     KKR(\Cl_{1,0}, A) \ar[d]_{\simeq} \ar[rr]^{\delta_{KK}} & &  KKR(\Cl_{1,0}, I \hat\otimes \Cl_{1,0}) \\
      DK(A) \ar[rr]_{\delta_{DK}}  & &  DK(I \hat\otimes \Cl_{1,0}) \ar[u]_{\simeq}
  }
\]
is commutative. The result then immediately follows from Proposition \ref{prop:vDboundary_to_KK}.
\end{proof}

A representative of the boundary map considered in Corollary \ref{cor:KK_bdry_cayley} is 
guaranteed to exist by~\cite[Theorem 1.1]{Skandalis85}. The advantage of the corollary 
is that it gives an explicit representative that is constructed from the given 
unbounded Kasparov module.
To finish this section, we also give a simple representative of the boundary map in 
van Daele as a bounded Kasparov module. 

\begin{prop} \label{prop:bdd_bdry_Kasmod}
Let $x \in M_n(A^{\sim e})$ be an OSU.
The class $\delta([x]) \in DK(I\hat\otimes \Cl_{1,0}) \cong KKR(\C, I )$ 
can be identified with the element
$$
   \big[( \C, \, I^n_I, \, \tilde{x}  )\big]  \in KKR(\C, I)
$$
with $\tilde{x}\in M_n(E)$  an odd self-adjoint lift of $x$.
\end{prop}
\begin{proof}
Applying the the bounded transform to the Kasparov module from
Proposition \ref{prop:vDboundary_to_KK}, we get the bounded 
operator $\sin(\tfrac{\pi}{2} \tilde{x}) \in \End_I(I^n)$.
We can then take a straight-line operator homotopy from 
$\sin(\tfrac{\pi}{2} \tilde{x})$ to $\tilde{x}$.
\end{proof}

Proposition \ref{prop:bdd_bdry_Kasmod} implies that the non-triviality of the 
class $\delta([x])$ as an element of $KKR(\C, I)$ is entirely contained in the 
failure of the lift $\tilde{x}$ to be invertible.
 For the case of $x$ related to a bulk Hamiltonian, such a condition 
can be linked to the presence of topological boundary spectrum.

\section{Applications to topological phases}  \label{sec:TI_application}

Van Daele $K$-theory has recently been employed by the second 
author and others to provide a 
classification of topological phases of materials with 
respect to an algebra of observables $A$~\cite{Kellendonk15, Kellendonk16, AlldridgeMax}. 
We now use our Cayley isomorphism to consider the corresponding 
class in $KK$-theory.

One reason for 
representing our bulk invariant as a Kasparov module is that we 
are then free to apply the full machinery of Kasparov theory to 
conduct further study on the invariants of interest. 
For example, if the algebra $A$ is a crossed product or groupoid 
algebra typically studied in the $C^*$-algebraic approach to  
condensed matter theory~\cite{Bel86}, then we immediately 
obtain a bulk-boundary correspondence for pairings of our bulk 
invariant with a `Dirac element' that extracts the  strong numerical phase 
of the system~\cite{BKR,BMes,BR1,KRSB02,KSB04,Kubota15b, PSBbook}.

\subsection{Bulk invariants for topological insulators}

For simplicity we will assume that $A$ is unital, which is roughly 
equivalent to working under a tight-binding approximation. 
We first  briefly review some physical terms.
\begin{defn}[Abstract insulators and symmetries]
We say a self-adjoint element $h \in A$ is an insulator if 
$h$ has a spectral gap. Taking a constant shift if necessary, we assume 
that an insulator $h$ is such that $0 \notin \sigma(h)$.

We say that an insulator $h$ has a chiral symmetry if $A$ is graded 
and $h$ is an odd element under this grading.

Let $A$ be a $C^*$-algebra with real structure $\rs_A$. 
\begin{enumerate}
  \item An insulator $h$ has a time-reversal symmetry (TRS) if 
  $h^{\rs_A} = h$. 
  \item An insulator $h$ has a particle-hole symmetry (PHS) if 
  $h^{\rs_A} = -h$. 
\end{enumerate}
\end{defn}

Because we take insulators $h$ to be self-adjoint invertible 
operators, the spectrally flattened operator $\hf:= h |h|^{-1}$ is a 
self-adjoint unitary. Therefore, if $h$ has a chiral symmetry, then 
$\hf$ gives an element in $V(A)$. Provided we have another 
odd self-adjoint unitary for comparison, we obtain an element 
in $DK(A)$. We call this van Daele element the 
{\em bulk invariant} of the topological phase.
If there is no chiral symmetry we take the tensor 
product $A\otimes \Cl_1$ and consider $\hf\otimes \rho$ 
instead, which is an OSU.

\subsubsection{With chiral symmetry}

We consider a chiral symmetry which is inner in 
the sense that the grading of $A$ is given by 
$\mathrm{Ad}_\Gamma$ for some $\Gamma=\Gamma^* \in A$.

\begin{example}[Chiral symmetry, no real structure]
If we do not make any reference to real structures on the 
graded algebra $A$, then taking the projection 
$\Pi_+=\frac12(1+\Gamma)$,  $A$ is isomorphic to $A_{++}\otimes \Cl_2$ 
with $A_{++} = \Pi_+ A \Pi_+$ 
a trivially graded algebra~\cite[Proposition 3.5]{Kellendonk15}.  
This isomorphism depends on a choice of OSU $e \in A$. 
Using this isomorphism, the operator of interest is 
$u_h = \Pi_+ e \hf \Pi_+$, which is unitary and  gives  a class $[u_h] \in K_1(A_{++})$. 
Hence we can apply our complex ungraded Cayley 
map (Theorem \ref{thm:complex_cayley}) to obtain the $KK^1$-class of the bulk invariant,
\begin{equation*}
\big[(\Cl_1, \, \ol{(u_h -1){A_{++}}}_{A_{++}} \otimes \C^2,\,  \Ci(u_h) \otimes \sigma_1 ) \big], \; 
u_h = \Pi_+ e \hf \Pi_+, \;  \Ci(u_h) = i(u_h+1)(u_h-1)^{-1} .  
\vspace{-0.3cm}
\end{equation*}
\hfill $\diamond$
\end{example}

The above expressions are for insulators with complex symmetries.
We now consider symmetries involving a real structure like TRS or PHS. 

\vspace{0.1cm}

\begin{example}[Real grading]
If we have an inner chiral symmetry with real
grading operator $\Gamma=\Gamma^{\rs_A}$ then $A_{++}^{\rs_A} = A_{++}$, $A^{\rs_A}\cong
A_{++}^{\rs_A}\otimes Cl_{1,1}$. 
If $(e \hf)^{\rs_A} = e\hf$, then $u_h^{\rs_A} = u_h$ is a unitary in $A_{++}^{\rs_A}$ and we 
are in the same situation as Example \ref{ex:triv_graded_Cayley}. 
Hence the Kasparov module of interest is
$$
  \left[ \big( \Cl_{1,0}, \, \ol{(u_h-1){A_{++}}}_{A_{++}} \otimes \C^2, \, (u_h+1)(u_h-1)^{-1} \otimes f \big) \right] \in KKR(\Cl_{1,0}, A_{++}) \cong KO_1(A_{++}^{\rs_A}).
$$
If $(e \hf)^{\rs_A} = -e\hf$, then $iu_h$ is unitary and $(iu_h)^{\rs_A} = iu_h \in A_{++}^{\rs_A}$. Our Cayley map then 
gives the Kasparov module
$$
   \big( \Cl_{1,0}, \, \ol{(u_h+i)A_{++}}_{A_{++}} \otimes \C^2, \, (u_h-i)(u_h+i)^{-1} \otimes f \big).  
 \vspace{-0.6cm}
$$
\hfill $\diamond$
\end{example}

\vspace{0.1cm}

\begin{example}[Imaginary grading]
If the grading operator $\Gamma$ is imaginary, $\Gamma^{\rs_A} = -\Gamma$, then for
$a_{++} \in A_{++}$, $a_{++}^{\rs_A} \in A_{--} = \Pi_- A \Pi_-$ with 
$\Pi_- = \frac{1}{2}(1-\Gamma)$. In this situation the real subalgebra $A^{\rs_A}\cong
A_{++}^{\mathrm{Ad}_e\circ{\rs_A}}\otimes Cl_{2,0}$~\cite[Theorem 3.10]{Kellendonk15}. 
If $(e\hf)^{\rs_A} = e\hf$, we check that 
$$
  (\Pi_+ e \hf \Pi_+)^{\mathrm{Ad}_e \circ {\rs_A}} = e \Pi_- e\hf \Pi_- e = \Pi_+ \hf e \Pi_+ = (\Pi_+ e \hf \Pi_+)^*.
$$
That is, $u_h^{\mathrm{Ad}_e \circ {\rs_A}} = u_h^*$ and we are in the case 
of Example \ref{ex:imag_cayley}. Therefore, for $\Ci(u_h) = i(u_h+1)(u_h-1)$,
$$
   \left( \Cl_{0,1}, \, \ol{(u_h-1)A_{++}}_{A_{++}} \otimes \C^2 , \, \begin{pmatrix} 0 & \Ci(u_h) \\ \Ci(u_h) & 0 \end{pmatrix} \right), 
   \quad  (v_1,v_2)^\mathfrak{r} = (v_1^{\mathrm{Ad}_e \circ {\rs_A}}, v_2^{\mathrm{Ad}_e \circ {\rs_A}})
$$
is the Real Kasparov module of interest and determines a class in 
$KKR(\Cl_{0,1}, A_{++}) \cong KO_{-1}(A_{++}^{\mathrm{Ad}_e\circ{\rs_A}})$.

Suppose now that $(e\hf)^{\rs_A} = -e\hf$. Then $u_h^{{\mathrm{Ad}_e\circ{\rs_A}}} = -u_h^*$ and we are in 
the setting of Example \ref{ex:KO3_to_KK}. Hence we can consider the ungraded unitary 
$$
  v_h = \begin{pmatrix} 0 & u_h \\ u_h & 0 \end{pmatrix} \in M_2(A),  \qquad 
  v_h^{\mathrm{Ad}_{-i\sigma_2} \circ \mathrm{Ad}_e \circ {\rs_A}} = v_h^*,
$$
so $v_h \in KO_{-1}(M_2(A_{++})^{\mathrm{Ad}_{-i \sigma_2}\circ \mathrm{Ad}_{e}\circ {\rs_A}}) 
  \cong KO_{-1}(A_{++}^{\mathrm{Ad}_{e}\circ {\rs_A}} \otimes \mathbb{H}) \cong KO_3(A_{++}^{\mathrm{Ad}_{e}\circ {\rs_A}})$. 
Applying the Cayley transformation, 
\[
  \left( \Cl_{0,1}, \, \ol{(v_h - 1_2)M_2(A)}_{A\otimes M_2(\C)} \hat\otimes \C^2, \, \Ci(v_h)\otimes \sigma_1 \right)
\]
is an unbounded $KKR$-cycle 
 with real structure $\mathrm{Ad}_{-i \sigma_2}\circ \mathrm{Ad}_{e}\circ {\rs_A}$ 
applied pointwise on the direct sum. The unbounded cycle represents an 
element in $KKO(Cl_{0,1}, A_{++}^{\mathrm{Ad}_e\circ{\rs_A}}\otimes Cl_{0,4})$.
\hfill $\diamond$
\end{example}

\subsubsection{Without chiral symmetry}

If there is no chiral symmetry, then the relevant algebra is $A\otimes \Cl_1$ (potentially with a real structure), 
where we have the odd self-adjoint unitary $\hf \otimes \rho$.

\begin{example}[No symmetry or TRS only]
We consider the two OSUs $\hf \otimes \rho$ and $1 \otimes \rho$ (where $1 \otimes \rho$ 
plays the role of a base point) and the element 
$[\hf\otimes \rho] - [1 \otimes \rho] \in DK(A \otimes \Cl_1)$ encodes the obstruction of  
a homotopy of $\hf$ to a trivial Hamiltonian. 
We are in the setting of Example \ref{ex:K_0_cayley}, where there is a map 
$$
  [\hf\otimes \rho] - [1 \otimes \rho] \mapsto \big[ ( \C, \, \frac{1}{2}(\hf -1 )A_A, \, 0) \big] \hat\otimes_{\C} 1_{KK(\Cl_1,\Cl_1)}.
$$
More simply still, we recover the class of the Fermi projection 
$[\frac{1}{2}(\hf -1)] = [\chi_{(-\infty,0]}(h)] \in K_0(A)$.

If in addition $A$ has a real structure with $h^{\rs_A} = h$, then the same argument applies and 
we get the $KKR$-class $\big[( \C, \frac{1}{2}(\hf -1) A_A, 0 ) \big]$ or the class of the 
projection $[\frac{1}{2}(\hf -1)] \in KO_0(A^{\rs_A})$. Note that in many cases of interest, 
$A^{\rs_A} \cong A^{\tilde{\rs}_A} \otimes \mathbb{H}$ for some other real structure $\tilde{\rs}_A$. 
In such a situation, by the K\"{u}nneth formula 
$[\frac{1}{2}(\hf -1)] \in KO_0(A^{\tilde{\rs}_A} \otimes \mathbb{H}) \cong KO_4(A^{\tilde{\rs}_A})$.
\hfill $\diamond$
\end{example}


\vspace{0.1cm}

\begin{example}[Particle-hole symmetric Hamiltonians] 
Let $A$ be a trivially graded and complex $C^*$-algebra and suppose 
there is a real structure $\mathfrak{r}_A$ on $A$ with 
 $\hf^{\mathfrak{r}_A} = -\hf$ for some insulator $h \in A$. 
That is, $h$ has a particle-hole symmetry. 

The algbera $A$ is ungraded so we consider the element $\hf\otimes (1,-1) \in \Cl_1$, where we use the 
real structure on $\Cl_1 \cong \C\oplus \C$ given by $(\alpha,\beta)^{\mathfrak{r}_{0,1}} = (\ol{\beta},\ol{\alpha})$. 
One then checks that $\hf\otimes (1,-1)$ is self-adjoint, square one and 
$$
   (\hf \otimes (1,-1))^{ \mathfrak{r}_A\otimes \mathfrak{r}_{0,1}}  = \hf^{\mathfrak{r}_A} \otimes (1,-1)^{\mathfrak{r}_{0,1}} = -\hf \otimes (-1,1) = \hf \otimes (1,-1).
$$

Let us consider the bulk phase of the odd self-adjoint unitary $\hf \otimes (1,-1)$ 
relative to a fixed base point. Namely, suppose that $A$ has a 
real skew-adjoint unitary $J$, $J^*=-J$, $J^2=-1$ and $J^{{\rs_A}} = J$. Then 
$iJ\otimes (1,-1) \in A \otimes \Cl_1$ is an odd self-adjoint unitary 
invariant under $\mathfrak{r}_A\otimes \mathfrak{r}_{0,1}$. 
We therefore obtain a class 
$$
   [\hf \otimes (1,-1)] - [iJ \otimes (1,-1)] \in DK(A \otimes \Cl_1, {\rs_A} \otimes \rs_{0,1}).
$$
Applying our Cayley map, we note that $\Ci_{iJ \otimes (1,-1)}(\hf) = iJ(\hf +iJ)(\hf -iJ)^{-1} \otimes (1,-1)$ 
and so our Kasparov module is 
$$
  \left( \Cl_{1,0}, \, \ol{(\hf - iJ)\otimes (1,-1)(A\otimes \Cl_1)}_{A\otimes \Cl_{0,1}}, \,  iJ(\hf +iJ)(\hf -iJ)^{-1} \otimes (1,-1) \right), 
$$
where the left Clifford action is multiplication by $iJ \otimes (1,-1)$. The class of this 
Kasparov module gives an element in $KKR(\Cl_{1,0}, A\otimes \Cl_{0,1}) \cong KO_2(A^{{\rs_A}})$. 
\hfill $\diamond$
\end{example}

\subsection{Boundary invariants for topological insulators}

We now consider the boundary map of the bulk invariant  from a (Real, graded) short exact sequence 
with positive linear splitting 
$$
  0 \to I \to E \to A \to 0
$$
and the corresponding image in $DK$-theory and $KK$-theory. Our work complements 
recent descriptions of the boundary $K$-theory class of topological phases via the Cayley 
transform by Schulz-Baldes and Toniolo~\cite{SBToniolo}. Similarly, Alldridge, Max and Zirnbauer 
use the van Daele boundary map and Roe's isomorphism $DK(I\hat\otimes \Cl_{1,0}) \to KKR(\C,I)$ 
to write down a bounded representative of the boundary invariant~\cite{AlldridgeMax}. 
Our work has different motivations and constructions to~\cite{AlldridgeMax} though there are clear similarities.

Let $A$ be unital and $h\in A$ an insulator with spectral gap $\Delta$ at $0$; we may suppose that $\Delta = (-\delta,\delta)$. 
We set $t_\Delta = \frac{2\pi}{|\Delta|} = \frac{\pi}{\delta}$, a characteristic time.
Let $\tilde{h}$ be a lift of $h$ in $E$. 
Then 
\begin{equation} \label{eq:h_lift}
 \tilde{a} = \frac{\tilde{h}}{\delta} P_\Delta(\tilde{h}) + P_{\geq\delta}(\tilde{h}) - P_{\leq -\delta}(\tilde{h})
\end{equation}
is a lift of the spectrally flattened $\hf$.  

As we will show, our explicit lift $\tilde{a}$ combined with our general results about 
the boundary map from Section \ref{subsec:DK_bdry} will allow 
us to write down the boundary invariants of topological insulators explicitly in terms of the  
lift $\tilde{h}$.

\subsubsection{Without chiral symmetry}
If $A$ is trivially graded then $x=\hf\otimes e$ is an OSU of $A\otimes \Cl_1$ 
where $e$ is the square one generator of $\C l_1$. It follows that $\tilde{ x} = \tilde{ a}\otimes e$ is 
a lift of $x$. We recall Lemma \ref{lem:vD_bdry_map}, which gives the element 
$\delta([\hf \otimes e]) \in DK(I\otimes \Cl_1 \hat\otimes \Cl_1)$. Using the identification 
$I\otimes\Cl_1\hot\Cl_1 \cong I \otimes \Cl_2$, the class $\delta([\hf \otimes e])$ is represented by the element 
$[Y] - [1\otimes \rho]$ with $\rho$ a generator of $\Cl_2$ and 
\begin{align*}
  Y &=  \exp(\pi \tilde{a} \otimes e \rho) (1\otimes \rho) \\
   &= - \exp\big(-i\pi (\tilde{a} \otimes \Gamma)\big)(1\otimes \rho) \\
   &= \big(P_\Delta(\tilde{h})^\perp \otimes 1 - (P_\Delta(\tilde{h})\otimes 1) 
   \exp(-it_\Delta (\tilde{h} \otimes  \Gamma))\big) (1\otimes\rho),
\end{align*}
where $\Gamma = i e\rho $ is the grading operator on $I\otimes \Cl_2$. 

\begin{example}[Boundary $KK$-class, No symmetries]
Let us consider the Kasparov module representing the boundary invariant without reference to a real structure. By 
Proposition \ref{prop:vDboundary_to_KK}, the boundary class $\delta([x]) \in DK(I \otimes \Cl_2)$ is represented 
by the unbounded Kasparov module
$$
   \left( \C, \, \ol{\cos(\tfrac{\pi}{2} \tilde{a}\otimes e) (I \otimes \Cl_1)}_{I\otimes \Cl_1}, \, \tan( \tfrac{\pi}{2} \tilde{a} \otimes e) \right).
$$
We note that 
because $\hf \in A$ and is not a matrix, we do not have to take a direct sum of the 
module $I_I$.
We can simplify this Kasparov module by noting that 
$\ol{\cos(\frac{\pi}{2} \tilde{a}\otimes e) (I \otimes \Cl_1)} \cong  \ol{\cos(\frac{\pi}{2} \tilde{a}) I} \otimes \Cl_1$ 
and $\tan( \frac{\pi}{2} \tilde{a} \otimes e) = \tan( \frac{\pi}{2} \tilde{a}) \otimes e$. 
Recalling our definition of $\tilde{a}$, Equation \eqref{eq:h_lift}, and 
writing $P_\Delta := P_\Delta(\tilde{h})$, we 
further reduce our boundary Kasparov module to 
$$
  \left( \C, \, \ol{ \cos( \tfrac{1}{2} t_\Delta \tilde{h}) P_\Delta I_I} \otimes \Cl_{1}, \, \tan(\tfrac{1}{2}t_\Delta \tilde{h}) \otimes e \right),
$$
where we denote by $\Cl_{1}$ the $C^*$-module ${\Cl_{1}}_{\Cl_1}$.

If we consider bounded representatives of the $KK$-class, then by Proposition \ref{prop:bdd_bdry_Kasmod}, 
the boundary invariant is represented by the Kasparov module 
$$
   \left[ \big( \C, \, I_I \otimes \Cl_1, \, \tilde{a} \otimes e \big) \right] = 
   \left[ \big( \C, P_\Delta  I_I \otimes \Cl_1, \, \tilde{h} \otimes e \big) \right] \in KK(\C, I \otimes \Cl_1).
$$
In many cases of interest, the lift $\tilde{h}$ is the restriction of $h$ to a system with boundary and 
$P_\Delta$ the projection onto edge spectrum. Hence 
our boundary Kasparov module closely lines up with the physical intuition of a boundary topological invariant.
\hfill $\diamond$
\end{example}

\vspace{0.1cm}

\begin{example}[Boundary $KK$-class, TRS and PHS Hamiltonians]
If $h$ has a TRS, $h^{\rs_A} = h$, then $\tilde{a}$ is real and self-adjoint, so $e$ must be real and self-adjoint.
Hence $Y\in I^{\rs_I}\otimes Cl_{2,0}$ and we have the real Kasparov module
$$
  \left( \C, \, \ol{ \cos( \tfrac{1}{2} t_\Delta \tilde{h}) P_\Delta  I_I} \otimes \Cl_{1,0}, \, \tan(\tfrac{1}{2}t_\Delta \tilde{h}) \otimes e \right)
$$
which gives a class in $KKR(\C, I \otimes  \Cl_{1,0}) \cong KO_{-1}(I^{\rs_I})$, where the last isomorphism 
is given by considering the ungraded Cayley transform $\Cd \big(\tan(\frac{1}{2}t_\Delta \tilde{h})\big)$. 
The bounded representative of this $KKR$-class is
$\big( \C, \, P_\Delta I_I \otimes \Cl_{1,0}, \, \tilde{h} \otimes e \big)$.

 If $h$ has a PHS, $h^{\rs_A} = -h$, then $\tilde{h}$ is imaginary and self-adjoint. Hence $e$ must 
  be imaginary and self-adjoint for $Y$ to be real and self-adjoint. 
Therefore  $Y\in I^{\rs_I}\otimes Cl_{1,1}$ and our boundary invariant is represented by the Kasparov module 
$$
  \left( \R, \, \ol{ \cos( \tfrac{1}{2} t_\Delta \tilde{h}) P_\Delta    I^{\rs_I}_{I^{\rs_I}}} \otimes Cl_{0,1}, \, \tan(\tfrac{1}{2}t_\Delta \tilde{h}) \otimes (1,-1) \right)
$$  
and corresponding class in $KKO(\R, I^{\rs_I} \otimes Cl_{0,1}) \cong KO_{1}(I^{\rs_I})$. The bounded 
representative is given by the Kasparov module 
$\big( \R, \, P_\Delta  I^{\rs_I}_{I^{\rs_I}} \otimes Cl_{0,1}, \, \tilde{h} \otimes (1,-1) \big)$.
\hfill $\diamond$
\end{example}

\subsubsection{With chiral symmetry}

We first note a general result on graded algebras that will be of use to us. 
Suppose $B$ is $\Z_2$-graded and 
the grading is implemented by a self-adjoint unitary $\Gamma \in \mathrm{Mult}(B)$. 
Then the map 
\begin{equation} \label{eq:pull_out_clifford}
\eta(b \hat\otimes \rho^k):= b\Gamma^k\otimes \rho^{k+|b|}
\end{equation}
defines a graded isomorphism
between  the graded tensor product
$B\hot\Cl_1$ with grading $\mathrm{Ad}_\Gamma$ on $B$ and the ungraded tensor product 
$B\otimes\Cl_1$ with trivial grading on $B$.

If an insulator $h\in A$ has a chiral symmetry, then the lift $\tilde{a}$ is an odd self-adjoint 
lift of $\hf$. Therefore, by Lemma \ref{lem:vD_bdry_map}, the class of $[\hf]$ under the 
boundary map in van Daele $K$-theory is represented by 
$$
  Y = - \exp(\pi \tilde{a} \hat\otimes \kg)(1\hat\otimes\kg) = 
  \left( P_\Delta(\tilde{h})^\perp - P_\Delta(\tilde{h}) \exp(t_\Delta \tilde{h}\hat\otimes \rho) \right) (1\hat\otimes \rho).
$$

\begin{example}[Chiral symmetry only]
We can again apply Proposition \ref{prop:vDboundary_to_KK}
and obtain a representative of $\delta([\hf])$ in $KK(\C,I)$ as the class of the Kasparov module
$$
  \left( \C, \, \ol{\cos(\tfrac{\pi}{2} \tilde{a}) I_I}, \, \tan(\tfrac{\pi}{2}\tilde{a}) \right).
$$
Because we have used the specific lift $\tilde{a}$, we write $P_\Delta = P_\Delta(\tilde{h})$ and 
simplify this Kasparov 
module to 
$$
   \left( \C, \, \ol{ \cos( \tfrac{1}{2} t_\Delta \tilde{h}) P_\Delta I_I }, \, \tan(\tfrac{1}{2} t_\Delta \tilde{h}) \right).
$$
Taking the bounded transform, we use Propositon \ref{prop:bdd_bdry_Kasmod} 
and obtain the boundary invariant
\begin{equation} \label{eq:chiral_bdry}
  \big[ \big( \C, \, I_I, \, \tilde{a} \big) \big] = \big[ \big( \C, \, P_\Delta  I_I, \, \tilde{h} \big) \big] 
  \in KK(\C, I).
\end{equation}

Suppose now that $I$ is inner-graded, e.g. the grading is implemented by an inner chiral symmetry 
on the boundary. Then using the isomorphism from Equation \eqref{eq:pull_out_clifford}, 
we know that $\delta([\hf]) \in DK(I \hat\otimes \Cl_1) \cong DK(I \otimes \Cl_1) \cong K_0(I)$. 
Equation \eqref{eq:chiral_bdry} gives 
a representative of this class, but not a canonical one as the Kasparov module in 
\eqref{eq:chiral_bdry} uses the grading on $I$.
\hfill $\diamond$
\end{example}

\vspace{0.1cm}

\begin{example}[TRS with chiral symmetry]
Let us now consider the boundary map of chiral symmetric Hamiltonians with a real structure. 
If $h$ has a TRS, $h^{\rs_A} = h$, then $\tilde{h}$ is real and 
$Y \in I^{\rs_I} \otimes Cl_{1,0}$. Our Kasparov module of interest is 
$$
 \left( \C, \, \ol{ \cos( \tfrac{1}{2} t_\Delta \tilde{h}) P_\Delta   I_I }, \, \tan(\tfrac{1}{2} t_\Delta \tilde{h}) \right)
$$
Taking the bounded transform, the boundary invariant is also represented by the Kasparov module
$$
   \big[\big( \C, \, P_\Delta I_I , \, \tilde{h} \big) \big] \in KKR(\C, I).
$$

If $I$ is inner-graded by the element $\Gamma \in \mathrm{Mult}(I)$ and is such that 
$\Gamma^\rs = \Gamma$, then the map $\eta$ from Equation \eqref{eq:pull_out_clifford} 
gives an isomorphism $I^{\rs_I} \hat\otimes Cl_{1,0}$ to $I^{\rs_I} \otimes Cl_{1,0}$, where 
$I^{\rs_I}$ has trivial grading on the right-hand side. Hence our boundary invariant can also 
be regarded as an element in $KO_0(I^{\rs_I})$. 
\hfill $\diamond$
\end{example}

\vspace{0.1cm}

\begin{example}[PHS with chiral symmetry]
Suppose that $h$ has a PHS so $\tilde{h}$ and the lift $\tilde{a}\in E$ are imaginary. 
In order to apply the van Daele boundary map, we first need to construct a real 
OSU.  
We consider the algebra $A \hat\otimes \Cl_{0,2}$, where one can check 
that $\hf \, \hat\otimes \, i f_1 f_2$ is an odd Real self-adjoint unitary.
Therefore, for $e$ the self-adjoint odd generator in $\Cl_{1,2}$ and $\omega = e_1 f_1 f_2 \in \Cl_{1,2}$ 
the orientation element, Lemma \ref{lem:vD_bdry_map} 
gives that 
$$
  Y = - \exp \big( \pi \tilde{a} \,\hat\otimes\, i\omega \big)(1\hat\otimes e) = 
  \left( P_\Delta(\tilde{h})^\perp - P_\Delta(\tilde{h}) \exp(t_\Delta \tilde{h}\, \hat\otimes \, i\omega) \right) (1\hat\otimes e)
$$
represents the class $\delta([\hf \, \hat\otimes \, i f_1 f_2]) \in DK(I \hat\otimes \Cl_{1,2})$. 

We can now apply Proposition \ref{prop:vDboundary_to_KK} to obtain the unbounded Kasparov module 
$$
  \Big( \C, \, \ol{\cos( \tfrac{\pi}{2} \tilde{a}\, \hat\otimes \, i f_1 f_2 )I \hat\otimes \Cl_{0,2}}_{I \hox \Cl_{0,2}}, \, 
   \tan(\tfrac{\pi}{2} \tilde{a}\, \hat\otimes \, i f_1 f_2 ) \Big)
$$
representing the boundary. We can 
simplify this Kasparov module to 
$$
  \Big(\C, \, \ol{\cos(\tfrac{\pi}{2} \tilde{a})I_I} \, \hat\otimes \, \Cl_{0,2}, \, \tan(\tfrac{\pi}{2} \tilde{a}) \, \hat\otimes \, if_1f_2 \Big)
$$
and using the explicit lift $\tilde{a}$ from Equation \eqref{eq:h_lift}, the boundary Kasparov module becomes
$$
  \Big( \C, \, \ol{ \cos( \tfrac{1}{2}t_\Delta \tilde{h}) P_\Delta I_I} \,\hat\otimes\, \Cl_{0,2}, \, 
    \tan(\tfrac{1}{2}t_\Delta \tilde{h}) \,\hat\otimes\, if_1f_2 \Big).
$$
We also take the bounded transform, where by a straight-line homotopy, the boundary class 
is represented by the Kasparov module
$$
  \left[ \big(  \C, \, P_\Delta I_I \,\hat\otimes\, \Cl_{0,2}, \, \tilde{h} \,\hat\otimes\, if_1f_2 \big) \right] \in KKR( \C, I \hat\otimes \Cl_{0,2}).
$$
If the grading on $I$ is inner and the grading operator imaginary, $\Gamma^\rs = -\Gamma$, 
then  the isomorphism $\eta$ from 
 Equation \eqref{eq:pull_out_clifford} is such that $I^{\rs_I} \hat\otimes Cl_{0,2} \cong I^{\rs_I} \otimes Cl_{0,2}$ 
 with $I^{\rs_I}$ trivially graded on the right-hand side. 
 Therefore, for inner and imaginary chiral symmetries, the boundary invariant gives a class 
 in $KO_{2}(I^{\rs_I})$.
 \hfill $\diamond$
\end{example}


\appendix
\section{Kasparov products with the Cayley transform}

In this appendix we outline how our Cayley map on $K$-theory 
is compatible with  the constructive form 
of the Kasparov product. 
We will address the complex case. The
real case can be adapted from the complex one,
for while the algebraic details change, the analytic details are
the same.

We will typically be interested in products of odd (ungraded) Kasparov modules 
$(\A, X_B, D)$, but note that our results also hold 
when $A=\overline{\A}$ and $B$ are $\Z_2$-graded. 
In the graded case, the triple $(\A, X_B, D)$ should be 
interpreted as the Kasparov module 
$(\A \hat\otimes \Cl_1, X_B^{\oplus 2}, \, D \hat\otimes \sigma_1)$, where 
$D: \Dom(D) \to X_B$ is even and the left Clifford action is generated by $-i\sigma_2$.

Given an unbounded Kasparov module $(\A,X_B,D)$
and an even unitary $u\in \A\subset A$ we 
are interested in representatives of the
product
$$
[(\C,\overline{(u-1)A_A},\Ci(u))]\ox_A[(\A,X_B,D)] \in KK(\C, B), \qquad \Ci(u) = i(u+1)(u-1)^{-1}.
$$
To construct a representative, we need to work on the module
$\overline{(u-1)X_B} \cong \overline{(u-1)A}\ox_AX_B$ (or rather two copies
of this module) and consider the  operators
$$
\Ci(u)\pm i\tilde{D}:\,(u-1)\Dom(D)\subset \overline{(u-1)X_B}
\to \overline{(u-1)X_B},
$$
where we need to make sense of the restriction $\tilde{D}=D|_{(u-1)\Dom(D)}$
in spite of the possible lack of complementability of $\ol{(u-1)X_B}$ in $X_B$.

\begin{lemma}
\label{lem:self-adj}
Let $(\A,X_B,D)$ be an unbounded Kasparov module 
and  $u\in\A^\sim$  unitary. Suppose that there exists an approximate unit 
$(v_n)\subset C^*((u-1),(u^*-1))$ such that for all $n$ the commutator $[D,v_n]$
is defined and bounded, $v_nX_B\subset (u-1)X_B=(u^*-1)X_B$,
and finally $[D,v_n](u^*-1)\to 0$ $*$-strongly. 

Then with $\tilde{\D}=s$-$\lim v_n\D|_{(u-1)X}$
we find that $(\Ci(u)\pm i\tilde{D})^*=\Ci(u) \mp i\tilde{D}$.
\end{lemma}
\begin{proof}
We prove the lemma for $\Ci(u)+i\tilde{D}$ 
since the other case is proved 
identically.
Let $y\in\Dom((\Ci(u)+i\tilde{D})^*)$. That is, there exists $z\in \overline{(u-1)X_B}$
such that for all elements $x\in \Dom(\Ci(u)+i\tilde{D})=(u-1)\Dom(D)$ we have
$$
((\Ci(u)+i\tilde{D})x\mid y)_B=(x\mid z)_B.
$$
Write $x=(u-1)\xi$ with $\xi\in\Dom(D)$. We observe that 
$v_n[\D,u]=[\D,v_n(u-1)]-[\D,v_n](u-1)\to [\D,u]$ strongly. Then 
\begin{align*}
((\Ci(u)+i\tilde{D})x\mid y)_B
&=((\Ci(u)+i\tilde{D})(u-1)\xi\mid y)_B\\
&=(i(u+1)\xi\mid y)_B+(i\lim v_n[D,u]\xi\mid y)_B+(\lim v_n(u-1)iD\xi\mid y)_B\\
&=(i(u+1)\xi\mid y)_B+(i[D,u]\xi\mid y)_B+((u-1)iD\xi\mid y)_B\\
&=(\xi\mid -i(1+u^*)y)_B
+(\xi\mid i[D,u^*]y)_B+(D\xi\mid -i(u^*-1)y)_B\\
&=(\xi\mid (u^*-1)z)_B.
\end{align*}
Rearranging the last equality shows that 
\begin{equation}
(D\xi\mid -i(u^*-1)y)_B
=(\xi\mid (u^*-1)z)_B + (\xi\mid i(u^*+1)y)_B - (\xi\mid i[D,u^*]y)_B
\label{eq:in-dom}
\end{equation}
and as this holds for all $\xi\in\Dom(D)$, we see that
$(u^*-1)y\in \Dom(D^*)=\Dom(D)$. Moreover, 
since $z=(\Ci(u)+i\tilde{D})^*y$ 
we learn that
$$
-iD(u^*-1)y=(u^*-1)(\Ci(u)+i\tilde{D})^*y + i(u^*+1)y -i[D,u^*]y
$$
or
\begin{equation}
(u^*-1)(\Ci(u)+i\tilde{D})^*y = -iD(u^*-1)y + i[D,u^*]y -i(u^*+1)y.
\label{eq:u-1-u+1}
\end{equation}
If $y\in\Dom(D)$ then 
$$
-iD(u^*-1)y+i[D,u^*]y=-i(u^*-1)D y
$$
and in this case we see from 
Equation \eqref{eq:u-1-u+1} that $(u^*+1)y$
is in  $(u^*-1)\Dom(D)\subset(u^*-1)X_B=\Dom(\Ci(u))$. 
Thus we can multiply through by
$(u^*-1)^{-1}$ and find that
$$
(\Ci(u)+i\tilde{D})^*y=-iD y+\Ci(u)y.
$$
Hence 
$$
\Dom((\Ci(u)+i\tilde{D})^*)\cap\Dom(D)=\Dom(\Ci(u)-i\tilde{D}),
$$
and we need only show that $y\in\Dom(D)$.
So let $(v_n)\subset C^*((u-1),(u^*-1))$ 
be as in the statement of the Lemma.
Then we find that
\begin{align*}
-iD(u^*-1)y+&i[D,u^*]y=\lim_n v_n(-iD(u^*-1)y+i[D,u^*]y)\\
&=\lim_n\Big(-i[v_n,D](u^*-1)y-iD(u^*-1)v_ny+i[v_n,[D,u^*]]y+i[D,u^*]v_ny\Big)\\
&=\lim_n\Big(-i[v_n,D](u^*-1)y-i(u^*-1)D v_ny+i[v_n,[D,u^*]]y\Big)\\
&=\lim_n\Big(-i[v_n,D](u^*-1)y-i(u^*-1)D v_ny\Big)
\end{align*}
where the last equality follows since $[v_n,[D,u^*]]\to 0$ strongly.
Now $v_ny\in (1-u)X$ and $v_ny\in \Dom(D)$ by Equation \eqref{eq:in-dom}.
Thus if $[v_n,D](u^*-1)\to 0$ strongly we deduce 
that $y$ is in the closure
of $(u-1)\Dom(D)$ in the graph norm of $D$, 
and so in $\Dom(D)$.
\end{proof}

\begin{thm}
\label{thm:product}
Let $(\A,X_B,D)$ be an odd unbounded Kasparov module and 
$u \in \calA^\sim$ an even unitary. 
Suppose $C^*((u-1),(u^*-1))$ has an 
approximate unit $(v_n)$ as in Lemma \ref{lem:self-adj} 
and $\| [D, u] \| < 2$. 
Then 
$$ 
  \Big( \C, \, \overline{(u-1)X_B}\oplus \overline{(u-1)X_B}, \, \Ci(u)\hat{+}\tilde{D} \Big), \quad 
    \Ci(u)\hat{+}\tilde{D}:=\begin{pmatrix} 0 & \Ci(u) -i\tilde{D}\\ \Ci(u)+i\tilde{D} & 0\end{pmatrix}
$$
is an unbounded Kasparov module  
 representing the 
Kasparov product  of the class of the Cayley transform 
$\big[(\C,\overline{(u-1)A_A},\Ci(u))\big] \in KK_1(\C, A)$ with 
$\big[(\A,X_B,D)\big] \in KK^1(A, B)$.
\end{thm}

{\bf Remark} If $[D,u]$ is bounded we can ensure 
that $\Vert [D,u]\Vert< 2$ is satisfied by rescaling $D$.
\hfill $\diamond$

\begin{proof}
We employ the main result of \cite{Kuc1}.
First, if we make the  identification 
$\overline{(u-1)A}\ox_A X_B\cong \overline{(u-1)X_B}$ by left multiplication,
then the map
$$
X_B\ni x\mapsto (u-1)aD x-D|_{(u-1)X}(u-1)ax=-[D,(u-1)a]x
$$ 
is 
bounded. This proves that 
Kucerovsky's connection condition is satisfied \cite{Kuc1}.

Since $\Dom(\Ci(u))\subset\Dom(\Ci(u)\pm i\tilde{D})$, the 
domain condition is
satisfied, and we need only check Kucerovsky's positivity condition and
that we have a Kasparov module.

The assumption that $[\D,v_n](u^*-1)\to 0$ strongly 
and Lemma \ref{lem:self-adj} tells us that  $\Ci(u)\hat{+}\tilde{D}$
is self-adjoint. 
For regularity, let $\phi:\,B\to \C$ be a state, and form the Hilbert space
$X\ox_BL^2(B,\phi)$. The sequence $v_n\ox 1$ satisfies the same domain
mapping properties with respect to  $(\Ci(u)\hat{+}\tilde{D})\ox 1$ as
$v_n$ did for $\Ci(u)\hat{+}\tilde{D}$, and so the above arguments show that
$(\Ci(u)\hat{+}\tilde{D})\ox 1$ is self-adjoint. As $\phi$ was an arbitrary state, 
the local global-principle \cite{KaLe2,Pierrot} implies the regularity of 
$\Ci(u)\hat{+}\tilde{D}$.

To check the positivity condition we first compute the anti-commutator.
So
\begin{align*}
&\begin{pmatrix} 0 & \Ci(u)\\ \Ci(u) & 0\end{pmatrix}
\begin{pmatrix} 0 & \Ci(u)-i\tilde{D}\\ \Ci(u)+i\tilde{D} & 0\end{pmatrix}\\
&\qquad+
\begin{pmatrix} 0 & \Ci(u)-i\tilde{D}\\ \Ci(u)+i\tilde{D} & 0\end{pmatrix}
\begin{pmatrix} 0 & \Ci(u)\\ \Ci(u) & 0\end{pmatrix}\\
&=2\begin{pmatrix} \Ci(u)^2 & 0\\ 0 &\Ci(u)^2\end{pmatrix}+
\begin{pmatrix} i[\Ci(u),D] & 0\\ 0 & -i[\Ci(u),D]\end{pmatrix}\\
&=2\begin{pmatrix} 4(u-1)^{-1}(u^*-1)^{-1}-1 & 0\\ 0 & 4(u-1)^{-1}(u^*-1)^{-1}-1\end{pmatrix} \\
&\qquad +
\begin{pmatrix} 2(u-1)^{-1}[u,D](u^*-1)^{-1} & 0\\ 0 & -2(u-1)^{-1}[u,D](u^*-1)^{-1}\end{pmatrix}\\
&=2\begin{pmatrix} (u-1)^{-1}\big(4-[u,D]u^*\big)(u^*-1)^{-1}-1 & 0\\ 0 & 
(u-1)^{-1}\big(4+[u,D]u^*\big)(u^*-1)^{-1}-1\end{pmatrix}
\end{align*}
Recasting this computation in terms of quadratic forms shows that
the required positivity holds when $[D,u]u^*\leq 4$, which 
is satisfied since we assume
$\Vert[D,u]\Vert<2$. 

Because the left-action is by the complex numbers, commutators 
of the left action with $\Ci(u)\hat{+}\tilde{D}$ are
trivially bounded. Thus all that remains is  to check the 
compact resolvent condition.
The self-adjointness of $\Ci(u)\hat{+}\tilde{\D}$ on $(u-1)\Dom(\D)$
tells us that 
$$
(i\pm \Ci(u)\hat{+}\tilde{\D})^{-1}\ol{(u-1)X_B}^{\oplus 2}=(u-1)\Dom(\D)^{\oplus 2}=(u-1)(i\pm \D)^{-1}X_B^{\oplus 2}
\hookrightarrow X_B^{\oplus 2}
$$
where the last inclusion is compact since $u-1\in \A$ and 
$\D$ has locally compact resolvent.
\end{proof}

\begin{example}
\label{eg:zed}
For $z\in S^1$ we let $\rho(z) =2-z-\overline{z}$ and 
define $v_n=\rho(\rho+1/n)^{-1}$.
Elementary trigonometry shows that
$$
\Big[\frac{1}{i}\frac{d}{d\theta},v_n\Big](1-\overline{z})
=2(1-v_n)\Big(\frac{\sin(\theta)}{1-\cos(\theta)+1/2n}-\cot(\theta/2)\Big)e^{i\theta/2}\sin(\theta/2)+2(1-v_n)e^{i\theta/2}\cos(\theta/2)
$$
which does indeed go to zero strongly on $L^2(S^1)=\overline{(z-1)L^2(S^1)}$.
\hfill $\diamond$
\end{example}



\begin{thebibliography}{99}


\bibitem{AlldridgeMax}
A. Alldridge, C. Max and M. R. Zirnbauer. \emph{Bulk-boundary correspondence for disordered
free-fermion topological phases}. Comm. Math. Phys., online first (2019),
\url{https:dx.doi.org/10.1007/s00220-019-03581-7}


\bibitem{BJ}
 S. Baaj and P. Julg. \emph{Th\'{e}orie bivariante de Kasparov et op\'{e}rateurs 
  non born\'{e}s dans les $C^*$-modules hilbertiens}. 
  {C. R. Acad. Sci. Paris S\'{e}r. I Math}, \textbf{296} (1983), no. 21, 875--878.

\bibitem{Bel86}
J. Bellissard. \emph{{$K$}-theory of {$C^\ast$}-algebras in solid state physics}. In 
Statistical mechanics and field theory: mathematical aspects ({G}roningen, 1985), volume 257 
of Lecture Notes in Phys., Springer, Berlin (1986), 99--156.





\bibitem{Blackadar}
B. Blackadar. \emph{$K$-Theory for Operator Algebras}. Volume 5 of {Mathematical Sciences Research Institute Publications}, 
Cambridge Univ. Press, Cambridge (1998). 

\bibitem{Boersema02}
J. L. Boersema. \emph{Real $C^*$-algebras, united $K$-theory and the K\"{u}nneth formula}. 
$K$-theory, \textbf{26} (2002), 345--402.

\bibitem{BL15}
J. L. Boersema and T. A. Loring. \emph{$K$-Theory for real $C^*$-algebras via unitary elements with symmetries}. 
New York J. Math. \textbf{22} (2016), 1139--1220.



\bibitem{BKR} 
C. Bourne, J. Kellendonk and A. Rennie. \emph{The $K$-theoretic bulk-edge correspondence for topological insulators}.
Ann. Henri Poincar\'e, {\bf 18} (2017), no. 5, 1833--1866.

\bibitem{BMes}
C. Bourne and B. Mesland. \emph{Index theory and topological phases of aperiodic lattices}. 
Ann. Henri Poincar\'{e}, \textbf{20} (2019), no. 6, 1969--2038.

\bibitem{BR1}
C. Bourne and A. Rennie. \emph{Chern numbers, localisation and the bulk-edge correspondence for continuous models of topological phases}.
Math. Phys. Anal. Geom., {\bf21} (2018), no. 3, art. 16, 62 pages.

\bibitem{CP2} 
A. L. Carey and J. Phillips. \emph{Spectral flow in {F}redholm modules, eta invariants and the {JLO} cocycle.}
{$K$-Theory}, \textbf{31} (2004), no. 2, 135--194.

\bibitem{CMHyperbolic}
A. Connes and H. Moscovici. \emph{Cyclic cohomology, the {N}ovikov conjecture and hyperbolic groups}. 
Topology, \textbf{29} (1990), no. 3, 345--388.


\bibitem{vanDaele1} 
A. van Daele. \emph{$K$-theory for graded Banach algebras I}. {Quart. J. Math. Oxford Ser. (2)}, 
\textbf{39} (1988), no. 154, 185--199.

\bibitem{vanDaele2}
A. van Daele. \emph{$K$-theory for graded Banach algebras II}. {Pacific J. Math.}, 
\textbf{134} (1988), no. 2, 377--392.



\bibitem{DM} K. van den Dungen and B. Mesland. {\em Homotopy
equivalence in unbounded $KK$-theory}. Ann. $K$-Theory, to appear. arXiv:1907.04049 (2019).

\bibitem{vanErp}
E. van Erp. \emph{The index of hypoelliptic operators on foliated manifolds}. 
J. Noncommut. Geom., \textbf{5} (2011), no. 1, 107--124.


\bibitem{Exel} R. Exel, {\em A Fredholm operator approach to Morita equivalence}. $K$-Theory, {\bf 7} (1993), 285--308. 


\bibitem{FM13}
D. S. Freed and G. W. Moore. \emph{Twisted equivariant matter}. 
{Ann. Henri Poincar\'{e}}, \textbf{14} (2013), no. 8, 1927--2023.



\bibitem{HR} 
N. Higson and J. Roe. {\em Analytic $K$-homology}. Oxford University Press, Oxford (2000).

\bibitem{KaLe2} J. Kaad and M. Lesch. \emph{A local-global principle for regular operators on Hilbert modules}.  
J. Funct. Anal., {\bf 262} (2012), no. 10, 4540--4569.


  
  \bibitem{KNR} 
  J. Kaad, R. Nest and A. Rennie.  \emph{$KK$-theory and spectral flow in von Neumann algebras}.
 J. $K$-Theory, {\bf 10} (2012), no. 2, 241--277.

\bibitem{KasparovStinespring}
G. G. Kasparov. \emph{Hilbert {$C^{\ast}$}-modules: theorems of {S}tinespring and {V}oiculescu}. 
J. Operator Theory, \textbf{4} (1980), no. 2, 133--150.

\bibitem{Kasp} 
  G. G. Kasparov. \emph{The operator $K$-functor and extensions of $C^*$-algebras}. 
  {Math. USSR Izv.}, \textbf{16} (1981), 513--572.

\bibitem{Kellendonk15} 
J. Kellendonk. \emph{On the $C^*$-algebraic approach to topological 
phases for insulators}. {Ann. Henri Poincar\'e}, 
\textbf{18} (2017), no. 7, 2251--2300.

\bibitem{Kellendonk16}
J. Kellendonk. \emph{Cyclic cohomology for graded $C^{\ast,r}$-algebras and its pairings 
with van Daele $K$-theory}. Comm. Math. Phys., \textbf{368} (2019), no. 2, 467--518. 

\bibitem{KRSB02}
J. Kellendonk, T. Richter and H. Schulz-Baldes. \emph{Edge current channels and {C}hern numbers in the integer quantum {H}all effect}. 
Rev. Math. Phys., \textbf{14} (2002), no. 1, 87--119.

\bibitem{KSB04}
J. Kellendonk and H. Schulz-Baldes. \emph{Boundary maps for $C^*$-crossed products with 
$\mathbb{R}$ with an application to the quantum {H}all effect}. Comm. Math. Phys., 
\textbf{249} (2004), no. 3, 611--637.

\bibitem{Kitaev09}
A. Kitaev. \emph{Periodic table for topological insulators and superconductors}. In V. Lebedev \& M. Feigel{'}man, editors, 
volume 1134 of American Institute of Physics Conference Series, pages 22--30 (2009).

\bibitem{Kubota15a}
Y. Kubota. \emph{Notes on twisted equivariant {K}-theory for $C^*$-algebras}. 
{Internat. J. Math.}, \textbf{27} (2016), no. 6, 1650058.

 \bibitem{Kubota15b}
Y. Kubota. \emph{Controlled topological phases and bulk-edge correspondence}. 
{Comm. Math. Phys.}, \textbf{349} (2017), no. 2, 493--525.

\bibitem{Kuc1} D. Kucerovsky. 
\emph{The $KK$-product of unbounded modules}.
$K$-Theory, \textbf{11} (1997), 17--34.

\bibitem{KPS} A. Kumjian, D. Pask and A. Sims. 
{\em Graded $C^*$-algebras, graded $K$-theory, and twisted $P$-graph $C^*$-algebras}. 
J. Operator Theory, \textbf{80} (2018), no. 2, 295--348. 


\bibitem{Lance} 
E. C. Lance. \emph{Hilbert $C^*$-Modules: A Toolkit for Operator Algebraists}.
Cambridge University Press, Cambridge (1995).



\bibitem{Mingo} J. A. Mingo. {\em $K$-theory and multipliers of stable $C^*$-algebras}. Trans. Amer. Math. Soc., 
{\bf 299} (1987), no. 1, 397--411.


\bibitem{Natsume}
T. Natsume. \emph{A new role of graph projections in index theory}. In 
Quantum and non-commutative analysis ({K}yoto, 1992), volume 16 of Math. Phys. Stud., pages 285--290 (1993).

\bibitem{NestTsygan}
R. Nest and B. Tsygan. \emph{Algebraic index theorem}. Comm. Math. Phys., \textbf{172} (1995), no. 2, 223--262.


\bibitem{Ph1} J. Phillips. \emph{Self-adjoint Fredholm operators and
spectral flow}. Canad. Math. Bull., {\bf 39} (1996), 460--467.

\bibitem{Ph2} J. Phillips.
\emph{Spectral flow in type I and type II factors-a new approach}.
Fields Inst. Commun., {\bf 17} (1997), 137--153.



\bibitem{Pierrot} F. Pierrot. {\em Op\'erateurs r\'eguliers dans les $C^*$-modules et struture des
$C^*$-alg\'ebres des groupes de Lie semisimple complexes simplement connexes}. J. Lie Theory,
{\bf 16} (2006), 651--689.



\bibitem{PSBbook}
E. Prodan and H. Schulz-Baldes. \emph{Bulk and Boundary Invariants for Complex Topological Insulators: From $K$-Theory to Physics}. 
Springer, Cham (2016).



\bibitem{RSi} A. Rennie and A. Sims. {\em Non-commutative vector bundles for non-unital algebras}. 
SIGMA {\bf 13}, (2017), 041, 12 pages, arXiv:1612.03559.


\bibitem{Roe98} J. Roe.  \emph{Kasparov products and dual algebras}.
J. Funct. Anal., {\bf 155} (1998), 286--296.

\bibitem{Roe04}
J. Roe. \emph{Paschke duality for real and graded $C^*$-algebras}. {Q. J. Math.}, 
\textbf{55} (2004), no. 3, 325--331.

\bibitem{SchroderKTheory}
H. Schr\"oder. \emph{$K$-Theory for Real $C^*$-algebras and Applications}. 
Taylor \& Francis, New York (1993).

\bibitem{SBToniolo}
H. Schulz-Baldes and D. Toniolo. \emph{Dimensional reduction and scattering formulation for even topological invariants}. 
arXiv:1811.11781 (2018).

\bibitem{Skandalis85}
G. Skandalis. \emph{Exact sequences for the Kasparov groups of graded algebras.} 
Canad. J. Math., \textbf{37} (1985), 193--216.

\bibitem{Thiang16}
G. C. Thiang. 
\emph{On the K-theoretic classification of topological phases of matter.} Ann. Henri Poincar\'e. 
\textbf{17} (2016), no. 4, 757--794.

\bibitem{TroutGraded}
J. Trout. \emph{On graded $K$-theory, elliptic operators and the functional calculus}. 
Illinois J. Math., \textbf{44} (2000), no. 2, 294--309.

\bibitem{Wood} R. Wood. \emph{Banach algebras and Bott periodicity}. Topology, \textbf{4} (1965/66), 371--389.


\end{thebibliography}
\end{document}